\documentclass{article}
\usepackage{amsmath}
\usepackage{amsthm}
\usepackage{amssymb}
\usepackage[thinlines]{easytable}
\usepackage{graphicx}  

\newtheorem{theorem}{Theorem}
\newtheorem{proposition}[theorem]{Proposition}
\newtheorem{corollary}[theorem]{Corollary}
\newtheorem{lemma}[theorem]{Lemma}

\graphicspath{{Effect of Recombination on the Speed of Evolution/}}

\begin{document}

%+Title
\title{The Effect of Recombination on the Speed of Evolution}
\author{Nantawat Udomchatpitak\footnote{The author is supported in part by NSF grant DMS-1707953.}} 
\date{\today}
\maketitle
%-Title

%+Abstract
\begin{abstract}
	It has been a puzzling question why some organisms reproduce sexually. Fisher and Muller hypothesized that reproducing by sex can speed up the evolution. They explained that in the sexual reproduction, recombination can combine beneficial alleles that lie on different chromosomes, which speeds up the time that those beneficial alleles spread to the entire population. We consider a population model of fixed size $N$, in which we will focus on two loci on a chromosome. Each allele at each locus can mutate into a beneficial allele at rate $\mu_N$. The individuals with 0, 1, and 2 beneficial alleles die at rates $1, 1-s_N$ and $1-2s_N$ respectively. When an individual dies, with probability $1-r_N$, the new individual inherits both alleles from one parent, chosen at random from the population, while with probability $r_N$, recombination occurs, and the new individual receives its two alleles from different parents. Under certain assumptions on the parameters $N, \mu_N, s_N$ and $r_N$, we obtain an asymptotic approximation for the time that both beneficial alleles spread to the entire population. When the recombination probability is small, we show that recombination does not speed up the time that the two beneficial alleles spread to the entire population, while when the recombination probability is large, we show that recombination decreases the time, which agrees with Fisher-Muller hypothesis, and confirms the advantage of reproducing by sex.
\end{abstract}
%-Abstract



\section{Introduction}
	It has been a puzzle in evolutionary biology why many organisms reproduce sexually. Sexually reproducing parents transmit just half of their genes to the offspring, which means that all beneficial alleles that the parent has might not be fully transmitted to the offspring. This does not happen to parents who reproduce asexually, since they transmit all their genes to the offspring. An advantage of sexual reproduction might come from recombination, which can combine portions of different chromosomes together. Fisher \cite{Fis} and Muller \cite{Mul}  hypothesized that sexual reproduction can speed up the evolution. They explained that in an asexual population, for two beneficial mutations to survive, the second beneficial mutation has to occur in an individual that already has the first beneficial mutation, while in a sexually reproducing population, both beneficial mutations might occur on different individuals and recombination can later combine both mutations, which leads to an evolutionary advantage over asexual reproduction.   
	
\subsection{The model}

	We consider a population of fixed size $N$ consisting of $N$ chromosomes, which come from $N/2$ organisms of the same species.  We are interested in two loci on the chromosome.  One of the two loci contains either an $a$ or $A$ allele, and another locus contains either a $b$ or $B$ allele. Both the $A$ and $B$ alleles are beneficial. At time 0, all individuals have $a$ and $b$ alleles. Independently, each $a$ allele mutates to $A$ at exponential rate $\mu_N$, and each $b$ allele mutates to $B$ at exponential rate $\mu_N$. Individuals with 0, 1 and 2 beneficial alleles will die independently at exponential rates $1, 1-s_N$ and $1-2s_N$, respectively. A new individual is created immediately to replace the individual who dies, in order to keep the population size fixed. With probability $1-r_N$, no recombination occurs, in which case the new individual receives both alleles from a randomly chosen individual in the population at that time. With probability $r_N$, recombination occurs, in which case the new individual receives the $a/A$ allele from a randomly chosen individual, and receives the $b/B$ allele from another independently randomly chosen individual. We will give an approximation for the first time that all individuals in the population have both beneficial alleles, when the population size is large. The result shows that this time is shorter when $r_N$ is large, consistent with the Fisher-Muller hypothesis.

\subsection{Previous works}
	Takahata \cite{Taka} considered a model of a population of finite size, where each individual consists of one chromosome. This model focuses on two loci on the chromosome. One locus contains either an $a$ or $A$ allele, and another locus contains either a $b$ and $B$ allele. The fitnesses of individuals of types $ab, Ab, aB$ and $AB$ are assumed to be $1, 1+s, 1+s$ and $1+t$ respectively. The model also assumed recurrent mutations from $a$ to $A$ and from $b$ to $B$, which means that mutations will never be exhausted. In the beginning, the frequency of type $ab$ is assumed to be 1. Via simulation, the numerical fixation time of both $A$ and $B$ is given for some values of $s$ and $t$ in the following parameter regimes: 1) $t=s=0$, 2) $t=2s>0$, 3) $t=2s<0$, 4) $t>2s>0$, and 5) $t>0>s$.  
	
	Some non-rigourous works discuss the benefits of recombination. Crow and Kimura \cite{Crow} argued that in large populations, sexual reproduction can incorporate more mutations due to recombination than asexual reproduction can. Several works pursued finding the relation between the speed of adaptation and the recombination rate. Neher, Kessinger, and Shraiman \cite{NKS} considered a linear chromosome model assuming a large mutation rate and a weak selective effect. They obtained that the rate of adaptation is proportional to the square root of the recombination rate. Weissman and Barton \cite{WB} considered the regime where the mutation rate is small, and they obtained that the rate of adaptation is proportional to the recombination rate. Weissman and Hallatschek \cite{Weis} considered the  intermediate mutation rate regime and obtained that the rate of adaptation is proportional to the recombination rate. Lastly, Neher, Shraiman, and Fisher \cite{Neh} considered a population model, where a large number of loci was considered. The recombination mechanism in this model is different from the other works mentions before. Under the assumptions that the selective advantage is weak and the recombination rate is much larger than the selective advantage, they obtained that in large populations, the rate of adaptation increases as the square of recombination rate.   

	We will now discuss some rigourous results. Cuthbertson, Etheridge, and Yu \cite{Cuth} considered a two loci model with finite population size $N$. Each individual can be one of the four possible types: $ab, Ab, aB$ and $AB$. Both $A$ and $B$ are considered to be beneficial, and they increase the fitness by $s_1$ and $s_2$ respectively, with the assumption that $s_1<s_2$. The mutation from $b$ to $B$ randomly occurs during the the time interval that $Ab$ is spreading in the population, and it appears as a type $aB$.  For both $A$ and $B$ to spread to the entire population, there are three requirements. First, the number of type $aB$ should become significant. Second, recombination between $A$ and $B$ must occur. Lastly, the number of type $AB$ should become significant, after which $AB$ is almost certain to fixate. The result shows that the fixation probability of $AB$ can be approximated by the solution to a specific system of ODEs.

	Bossert and Pfaffelhuber \cite{Boss} considered a diffusion model with 4 types: $ab, Ab, aB$ and $AB$, where the fitnesses of $ab, Ab, aB$ and $AB$ are in increasing order. The frequencies of these four types evolve according to a system of SDEs. In the beginning, the frequencies of types $Ab$ and $aB$ are assumed to be small, and there is no type $AB$ yet. They obtain approximate formulas for the fixation probability and fixation time of type $AB$.
	
	Both Cuthbertson, Etheridge, and Yu \cite{Cuth} and Bossert and Pfaffelhuber \cite{Boss} assume that at least one beneficial mutation is present at the beginning, and they do not allow an unlimited supply of new mutations. In the model studied in this paper, we assume that all individuals in the beginning do not have any beneficial mutations, and both beneficial mutations occur according to a Poisson process.  This model is similar to the model given by Takahata in the case $t=2s>0$, but with finite population size. 
	
	Lastly, we mention another work by Berestycki and Zhao \cite{Ber}. In their model, which involves branching Brownian motion in two dimensions, they showed that the fitnesses on two loci are negative correlated. They explained that recombination can reduce this negative correlation, and leads to a fitter population.
   
\subsection{Conditions of the parameters}
	There are four parameters in our model: $N, \mu_N, r_N$ and $s_N$. We assume that $\mu_N\in (0,1), s_N\in(0,1/2]$ and $r_N \in [0,1)$. For any two sequences $a_N$ and $b_N$, we say that $a_N\ll b_N$ if
		$$
		\lim_{N\rightarrow \infty}\frac{a_N}{b_N} = 0.
		$$ 
	We will assume that $\mu_N$ and $s_N$ satisfy the following conditions:
		\begin{equation}\label{Con1.0}
			s_N \ll 1,
		\end{equation}
		\begin{equation}\label{Con1.1}
			1\ll N\mu_N,
		\end{equation}
		\begin{equation} \label{Con1.2}
			N\mu_N^2 \ll s_N,
		\end{equation}
	and
		\begin{equation}\label{Con1.3}
			r_N\ln_+(Nr_N)\ll s_N,
		\end{equation}
	where $\ln_+(x)$ is defined to be $\ln(x)$ if $x\in(1,\infty)$, and 0 if $x \in [0,1]$. Note that (\ref{Con1.1}) and (\ref{Con1.2}) imply that
		$$
				\mu_N \ll s_N.
		$$

\subsection{Main theorem}
	\begin{theorem}\label{THM}
		Let $T$ be the first time that all individuals in the population are type $AB$, which we also call the fixation time of $AB$. For every positive integer $N$, and $r\in[0,1]$, we define
		\begin{equation}\label{t*}
  			t^*_N(r)=\frac{1}{s_N}\ln\bigg(\frac{Ns_N^3}{\mu_N\cdot\max\{N\mu_N^2,r\ln_+(Nr)\}}\bigg).
		\end{equation}
		Then, for every $\theta\in(0,1)$, we have that
			$$
			\lim_{N\rightarrow \infty}P\big((1-\theta)t^*_N(r_N) \leq T \leq (1+\theta)t^*_N(r_N)\big)=1.
			$$
	\end{theorem}
	
	This theorem suggests that the time that both beneficial alleles spread to the entire population is approximately $t^*_N(r_N)$, when $N$ is large. From (\ref{t*}), when there is no recombination,
		$$
		t^*_N(0)=\frac{1}{s_N}\ln\bigg(\frac{s_N^3}{\mu_N^3}\bigg).
		$$
	When $r_N\ln_+(Nr_N)>N\mu_N^2$, we observe that $t^*_N(r_N)<t^*_N(0)$. This means that when $r_N$ is large enough, it decreases the fixation time of $AB$, compared with when there is no recombination. From (\ref{Con1.2}) and (\ref{Con1.3}), for sufficiently large $N$, we have that $\max\{N\mu_N^2, r_N\ln_+(Nr_N)\}<s_N$, and
		$$
		t^*_N(0)=\frac{1}{s_N}\ln\bigg(\frac{s_N^3}{\mu_N^3}\bigg)\geq t^*_N(r_N) >\frac{1}{s_N}\ln\bigg(\frac{Ns_N^2}{\mu_N}\bigg)=\frac{2}{3}\cdot\frac{1}{s_N}\ln\bigg(\frac{s_N^3}{\mu_N^3}\bigg)+\frac{1}{s_N}\ln(N\mu_N)>\frac{2}{3}t^*_N(0).
		$$ 
	This implies that under our assumptions, which assume small recombination rates, in large populations, recombination can decrease the fixation time of $AB$ by no more than a factor of one-third.
	
	\begin{figure}
\includegraphics[width=0.7\linewidth]{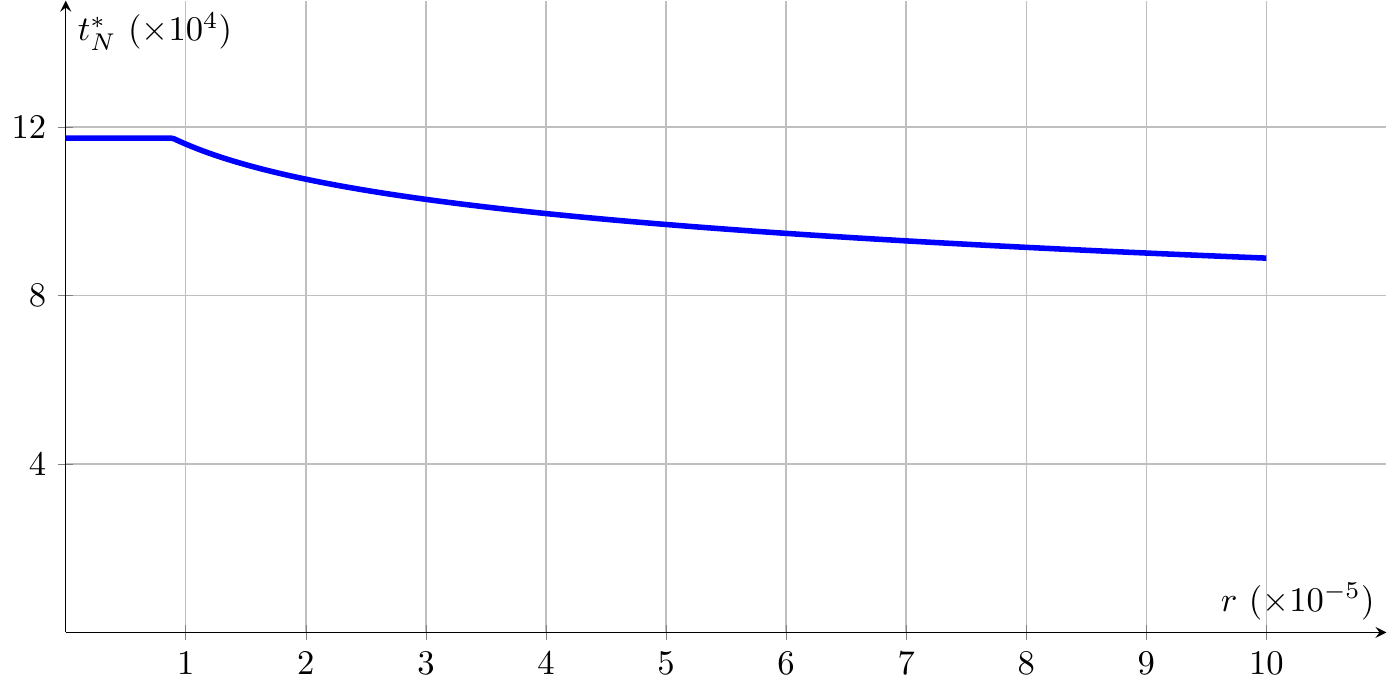}
\centering
\caption{\hspace{2pt}The graph of $t^*_N$ as the function of $r$, when $N=10^7, \mu=2\times 10^{-6}$ and $s=10^{-4}$. The $r$-axis is scaled by $10^{-5}$ and the $t^*_N$-axis is scaled by $10^4$.}
	\end{figure}	
	
	Lastly, we will show that these assumptions on the parameters are attainable. We consider when $\mu_N=N^{-a}, r_N=N^{-b}$ and $s_N=N^{-c}$ for some positive numbers $a, b$ and $c$. One can check that (\ref{Con1.0}), (\ref{Con1.1}), (\ref{Con1.2}) and (\ref{Con1.3}) are equivalent to $0<c<b$ and $(1+c)/2<a<1$. 
	
\section{Overview of the proof}
	From now on, we will refer to an individual with $ab$, $Ab$, $aB$, and $AB$ as type 0, 1, 2, and 3 respectively, and we will omit writing the subscript $N$ in $\mu_N, s_N$ and $r_N$. For $i=0,1,2,3$ and $t\geq 0$, we define $X_i(t)$ as the number of type $i$ individuals at time $t$ and define $\tilde X_i(t)=X_i(t)/N$, which is the fraction of type $i$ individuals in the population at time $t$. 

	Before we consider the behavior of the process $((X_0(t), X_1(t), X_2(t), X_3(t)),t\geq 0)$, we will first look at the condition $1\ll N\mu$. Intuitively, we don't want the mutations to occur too slowly, so that we see one beneficial mutation spread to the entire population, before any other mutations take hold. The process by which a beneficial allele spreads to the entire population is also known as a selective sweep. Suppose that a mutation from $a$ to $A$ is the first to occur, and assume that it doesn't go extinct. It will take time about $\frac{2}{s}\ln(N)$  to complete its selctive sweep (see section 6.1 of \cite{Dur}). During this time, a mutation from $b$ to $B$ occurs at total rate of $N\mu$. The number of descendants of one of these new mutations can be approximated by an asymmetric random walk. So, the chance that each of these mutations survives is about $s$. Hence, the number of mutations to $B$ that survive during the selective sweep of $A$ is approximately 
		$$
		N\mu\cdot s\cdot \frac{1}{s}\ln(N)=N\mu \ln(N).
		$$
	So, if $N\mu\ln(N)\ll 1$, then there is no $B$ that survives during the sweep of $A$. Hence, we will see $A$ spread to the entire population first, before $B$ appears and spreads. In this case, recombination does not speed up the time needed for the type $AB$ to take hold in the population. So, we should consider when $N\mu\ln(N)\gg 1$. Here, we make a slightly stronger assumption that $N\mu \gg 1$.

	Now, we will consider our process $((X_0(t), X_1(t), X_2(t), X_3(t)),t\geq 0)$. The behavior of our process is essentially reduced to two cases. For the first case, which we will call the \textit{recombination dominating case}, we assume that
		\begin{equation} \label{Con2.1}
			N\mu_N^2\ll r_N\ln(Nr_N) \ll s_N.
		\end{equation}
	For the second case, which we will call the \textit{mutation dominating case}, we assume that there is a positive constant $C$ such that for sufficiently large $N$, 
		\begin{equation} \label{Con2.2}
			r_N\ln_+(Nr_N)\leq CN\mu_N^2.
		\end{equation}
 	The reason for these names is that in the recombination dominating case, type 3 individuals start to appear from recombination between $A$ alleles from type 1 individuals and $B$ alleles from type 2 individuals, while in the mutation dominating case, the type 3 individuals start to appear from mutations from type 1 and type 2 individuals.
	 
	In the following table, we define times when we see significant changes in the behavior of the process.
	\begin{center}
 		\begin{TAB}[5pt]{|c|l|l|}{|c|c|c|c|c|c|}% (rows,min,max)[tabcolsep]{columns}{rows}
Time & recombination dominating & mutation dominating\\ 
 			$t_0$ & $\displaystyle{\frac{1}{s}\ln\bigg(\frac{s}{\mu\sqrt{Nr}}\bigg)-\frac{C_{0,r}}{s}}$ & $\displaystyle{\frac{1}{s}\ln\bigg(\frac{s}{N\mu^2}\bigg)-\frac{C_{0,m}}{s}}$\\ 
			$t_1$ & $\displaystyle{\frac{1}{s}\ln\bigg(\frac{s}{\mu}\bigg)-\frac{C_1}{s}}$ & $\displaystyle{\frac{1}{s}\ln\bigg(\frac{s}{\mu}\bigg)-\frac{C_1}{s}}$ \\ 
			$t_2$ & $\displaystyle{\frac{1}{s}\ln\bigg(\frac{s}{\mu}\bigg)+\frac{C_2}{s}}$ & $\displaystyle{\frac{1}{s}\ln\bigg(\frac{s}{\mu}\bigg)+\frac{C_2}{s}}$ \\ 
			$t_3$ & $\displaystyle{\frac{1}{s}\ln\bigg(\frac{s^2}{\mu r\ln(Nr)}\bigg)+\frac{C_3}{s}}$ & $\displaystyle{\frac{1}{s}\ln\bigg(\frac{s^2}{N\mu^3}\bigg)+\frac{C_3}{s}}$ \\ 
			$t_4$ & $\displaystyle{\frac{1}{s}\ln\bigg(\frac{s^2}{\mu r\ln(Nr)}\bigg)+\frac{C_4}{s}}$ & $\displaystyle{\frac{1}{s}\ln\bigg(\frac{s^2}{N\mu^3}\bigg)+\frac{C_4}{s}}$ 
		\end{TAB}
	\end{center}
	The constants $C_{0,r}, C_{0,m}, C_1, C_2, C_3$, and $C_4$  are defined in (\ref{C0r}), (\ref{C0m}), (\ref{C1}), (\ref{C2}), (\ref{C3}), and (\ref{C4}). The reader does not need to know what these constants are exactly at this point, but should notice that $C_i/s$ is the lower order term in the definition of the $t_i$. From now on, all statements are assumed to be true in both the recombination dominating case and the mutation dominating case, unless specified otherwise.

	\begin{figure}
\includegraphics[width=\linewidth]{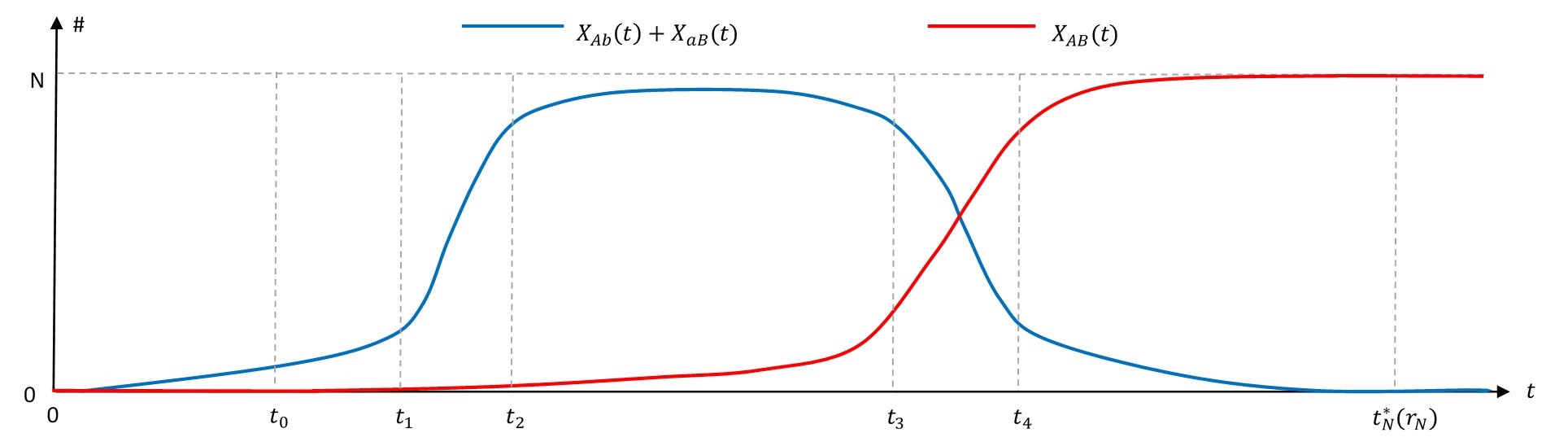}
\caption{The graphs of approximate numbers of individuals with one beneficial mutation (red) and two beneficial mutations (blue).}
\end{figure}

	Overall, the behavior of the numbers of type 1, 2 and 3 are similar in the sense that they first grow exponentially, then grow logistically. Both types 1 and 2 grow simultaneously, but type 3 will start to grow later, due to the late appearance of type 3 individuals. The behavior of the process is split into five time intervals, which will be discussed below. During the time interval $[0,t_1]$, which we will call \textit{phase 1}, most individuals are type 0. The type 1 and type 2 individuals appear from mutations from type 0 individuals.  Since type 1 and type 2 individuals die at rate $1-s$, while the majority of the population, which is type 0, dies at rate 1, the numbers of descendants of these type 1 and 2 ancestors grow exponentially at rate approximately $s$. Since the total rate of mutation from type 0 to type 1 is approximately $N\mu$, we have
		$$
		X_i(t)\approx \int_0^t N\mu\cdot e^{s(t-u)}du\approx \frac{N\mu}{s}e^{st}.
		$$ 
The type 3 individuals appear around time $t_0$. From this time, the number of type 3 individuals will grow exponentially at rate about $2s$, due to the fact that each type 3 individual dies at rate $1-2s$, while most individuals in the population die at rate 1. The following proposition describes the process at time $t_1$.
	\begin{proposition}\label{@t1}
	For $\epsilon>0$ and $\delta\in(0,1)$, there is an  event $A_{(1)}$, such that for sufficiently large $N$, we have that $P(A_{(1)})\geq 1-17\epsilon$, and the following statements hold:
		\begin{enumerate}
		\item On the event $A_{(1)}$, when $N$ is sufficiently large, for $i=1,2$, 
			\begin{equation}\label{Xit1}
			(1-\delta^2)e^{-C_1}N \leq X_i(t_1) \leq (1+\delta^2)e^{-C_1}N
			\end{equation}
		\item In the recombination dominating case, on the event $A_{(1)}$, there are positive constants $K^+_{1r}$ and $K^-_{1r}$ such that for sufficiently large $N$, 
			\begin{equation}\label{X3t1rec}
			\frac{K^-_{1r}Nr\ln(Nr)}{s}\leq X_3(t_1) \leq \frac{K^+_{1r}Nr\ln(Nr)}{s}.
			\end{equation}
		\item In the mutation dominating case, on the event $A_{(1)}$, there are positive constants $K^+_{1m}$ and $K^-_{1m}$ such that for sufficiently large $N$, 
			\begin{equation}\label{X3t1mut}
			\frac{K^-_{1m}N^2\mu^2}{s}\leq X_3(t_1) \leq \frac{K^+_{1m}N^2\mu^2}{s}.
			\end{equation}
		\end{enumerate}
	\end{proposition}
	This proposition says that when $N$ is sufficiently large, at time $t_1$, both type 1 and type 2 have established themselves in the population by having their numbers reaching the level of order $N$. However, $\tilde X_3(t_1)$ is only of order $r\ln(Nr)/s$ in the recombination dominating case, and is only of order $N\mu^2/s$ in the mutation dominating case, which from (\ref{Con1.2}) and (\ref{Con1.3}), means that number of type 3 at time $t_1$ is not yet comparable to those of type 1 and 2.

	During the time interval $[t_1,t_2]$, which we will call \textit{phase 2}, the numbers of type 1 and 2 now grow logistically, or more precisely,
		$$
		\tilde X_i(t)\approx \frac{1}{2}\bigg(\frac{1}{1+Be^{-s(t-t_1)}}\bigg),
		$$
	for $i=1,2$, where $B$ is some positive constant. The following proposition describes the process at time $t_2$.
	\begin{proposition}\label{@t2}
	For $\epsilon>0$ and $\delta\in(0,1)$, there is an  event $A_{(2)}$, such that for sufficiently large $N$, we have that $P(A_{(2)})\geq 1-21\epsilon$, and the following statements hold:
		\begin{enumerate}
		\item On the event $A_{(2)}$, for sufficiently large $N$, for $i=1,2$,
			$$
			\Big(\frac{1}{2}-\frac{3\delta^2}{2}\Big)N \leq X_i(t_2) \leq \Big(\frac{1}{2}-\frac{\delta^4}{4}\Big)N.
			$$
		\item In the recombination dominating case, on the event $A_{(2)}$, there are positive constants $K^+_{2r}$ and $K^-_{2r}$ such that for sufficiently large $N$, 
			$$
			\frac{K^-_{2r}Nr\ln(Nr)}{s}\leq X_3(t_2) \leq \frac{K^+_{2r}Nr\ln(Nr)}{s}.
			$$
		\item In the mutation dominating case, on the event $A_{(2)}$, there are positive constants $K^+_{2m}$ and $K^-_{2m}$ such that for sufficiently large $N$, 
			$$
			\frac{K^-_{2m}N^2\mu^2}{s}\leq X_3(t_2) \leq \frac{K^+_{2m}N^2\mu^2}{s}.
			$$
		\end{enumerate} 
	\end{proposition}
	This proposition says that at time $t_2$, almost half of the population becomes type 1, and almost the other half becomes type 2, while the number of type 3 individuals doesn't change much from time $t_1$.

	During the time interval $[t_2,t_3]$, which we will call \textit{phase 3}, the majority of the population has become type 1 or type 2. The number of type 3 individuals continues to grow exponentially from time $t_2$. However, since the majority of the population dies at rate $1-s$, and a type 3 individual dies at rate $1-2s$, the type 3 population grows exponentially at approximately rate $s$. The following proposition describes the behavior of the process at time $t_3$. 
	\begin{proposition} \label{@t3}
	For $\epsilon>0$ and $\delta\in(0,1)$, there is an  event $A_{(3)}$, such that for sufficiently large $N$, we have that $P(A_{(3)})\geq 1-25\epsilon-7\delta-\delta^2$, and the following statements hold:
		\begin{enumerate}
		\item For sufficiently large $N$, on the event $A_{(3)}$, we have
			$$ 
			X_0(t_3)<
				\begin{cases}
				\displaystyle{\delta e^{-(1-3\delta)(C_3-C_2)}N\cdot \bigg(\frac{r\ln(Nr)}{s}\bigg)^{1-3\delta}} &\mbox{in the recombination dominating case}\\
				\displaystyle{\delta e^{-(1-3\delta)(C_3-C_2)}N\cdot \bigg(\frac{N\mu^2}{s}\bigg)^{1-3\delta}} &\mbox{in the mutation dominating case}.
				\end{cases}
			$$
		\item In both cases, there is a positive constant $K_3$ such that for sufficiently large $N$, on the event $A_{(3)}$, we have
			$$
			K_3N\leq X_3(t_3)\leq \delta^2 N.
			$$
		\end{enumerate} 
	\end{proposition}
	This proposition says that by the time $t_3$, the number of type 3 individuals has reached order $N$. Moreover, from (\ref{Con1.2}) and (\ref{Con1.3}), there are almost no type 0 individuals left by time $t_3$.
	
	During the time interval $[t_3,t_4]$, which we will call \textit{phase 4}, the number of type 3 individuals grows logistically. The following proposition describes the behavior of the process at time $t_4$. 
	\begin{proposition} \label{@t4}
	For $\epsilon>0$ and $\delta\in(0,1)$, there is an  event $A_{(4)}$, such that for sufficiently large $N$, we have that $P(A_{(4)})\geq 1-26\epsilon-7\delta-\delta^2$, and on the event $A_{(4)}$, 
			$$
			\bigg(1-\frac{5\delta^2}{4}\bigg)N 
			\leq X_3(t_4)
			\leq \bigg(1-\frac{3K_3}{4}\bigg)N,
			$$
		and 
			$$
			X_1(t_4)+X_2(t_4)\geq \frac{K_3N}{2}.
			$$
	\end{proposition}
	This proposition implies that by time $t_4$, almost all individuals have become type 3, and only small fractions of type 1 and 2 individuals remain in the population.

	After time $t_4$, which we will call \textit{phase 5}, the number of individuals that are not type 3 can be approximated by a subcritical branching process. The non-type 3 population is heading toward extinction, and type 3 becomes fixated in the population. The fixation of type 3 will occur around time $t^*_N(r_N)$.

	In section \ref{trans}, we will discuss about transition rates of the process. In section \ref{M}, we construct martingales and submartigales, and give expectation and variance formulas. They will be used in the proofs of phases 1, 2, and 3 in sections \ref{phase1}, \ref{phase2}, and \ref{phase3}. In section \ref{phase1}, we will prove several lemmas on the process during phase 1, and at the end of the section, we give the proof of Proposition \ref{@t1}. Proposition \ref{@t2}, \ref{@t3} and \ref{@t4} will be proved in sections \ref{phase2}, \ref{phase3}, and \ref{phase4} respectively. Finally, the proof of Theorem \ref{THM} will be given at the end of section \ref{phase5}.
	
\section{On parameters and transition rates of the process}\label{trans}
\subsection{More inequalities on the parameters}
	\begin{lemma} The following statements hold.
		\begin{enumerate}
			\item In the recombination dominating case, 
				\begin{equation} \label{1<Nr}
				1 \ll Nr.
				\end{equation}
			\item In the mutation dominating case, 
				$$
				r\ll N\mu^2.
				$$ 
			\item In both cases,
				\begin{equation} \label{r<s}
				r \ll s,
				\end{equation} 
				\begin{equation}\label{rslogNs}
				\frac{r}{s}\ln(Ns)\ll 1, 
				\end{equation}
			and
				\begin{equation}\label{rslogsu}
				\frac{r}{s}\ln\bigg(\frac{s}{\mu}\bigg)\ll 1.\
				\end{equation}
		\end{enumerate}
	\end{lemma}
		
	\begin{proof}
	We will first prove statement 1. In the recombination dominating case, from conditions (\ref{Con1.1}) and (\ref{Con2.1}), 
		$$
		1 \ll (N\mu)^2 \ll Nr\ln(Nr),
		$$
	which implies that $1\ll Nr$. 
	
	Now, we will prove statement 2 by contradiction. Suppose there is a $c>0$ and an increasing sequence $\{N_k\}_{k=1}^{\infty}$ of natural numbers such that for all $k=1,2,3,...$, we have
			$$
			r_{N_k}>cN_k\mu^2_{N_k}.
			$$
	From (\ref{Con2.2}), we have that for all $k=1,2,3,..$,
			$$
			cN_k\mu^2_{N_k}\ln_+(cN_k^2\mu^2_{N_k})\leq r_{N_k}\ln_+(N_kr_{N_k})\leq CN_k\mu^2_{N_k}.
			$$
	This leads to a contradiction, since $1\ll N\mu$ implies that
			$$
			\ln_+(cN_k^2\mu^2_{N_k})\rightarrow \infty,
			$$
	as $k\rightarrow \infty$. 
		
	Lastly, we will prove statement 3. First, we will consider the recombination dominating case. By (\ref{Con1.3}) and (\ref{1<Nr}), 
		$$
		r\ll r\ln(Nr)\ll s.
		$$
	From (\ref{Con2.1}) and (\ref{r<s}), it follows that
		$$
		\frac{r}{s}\ln(Ns)= \frac{r}{s}\ln(Nr)+\frac{r}{s}\ln\bigg(\frac{s}{r}\bigg)\ll 1, 
		$$
	and because of (\ref{Con1.1}), for sufficiently large $N$,
		$$
		\frac{r}{s}\ln\bigg(\frac{s}{\mu}\bigg)\leq \frac{r}{s}\ln(Ns)\ll 1,
		$$
	which implies (\ref{rslogsu}). For the mutation dominating case, we define $r^*_N$ such that $Nr^*_N$ is the solution of
		$$
		x\ln(x)=\sqrt{(N\mu)^2\cdot Ns}.
		$$
	It follows that $N\mu^2 \ll r^*_N\ln(Nr^*_N) \ll s$. Therefore, by the same argument above, 
		\begin{equation}\label{-3.1}
		r^*_N \ll s,			
		\end{equation}
		\begin{equation}\label{-3.2}
		\frac{r^*_N}{s}\ln(Ns)\ll 1,
		\end{equation}
	and
		\begin{equation}\label{-3.3}
		\frac{r^*_N}{s}\ln\bigg(\frac{s}{\mu}\bigg)\ll 1.
		\end{equation}
	Also, from (\ref{Con2.2}) and the fact that $N\mu^2 \ll r^*_N\ln(Nr^*_N)$, for sufficiently large $N$,  we have $r_N\leq r^*_N$. This fact along with (\ref{-3.1}), (\ref{-3.2}) and (\ref{-3.3}) imply (\ref{r<s}), (\ref{rslogNs}) and (\ref{rslogsu}).
		\end{proof}
		
\subsection{Transition rates of the process} \label{rate}
	For the proof, we need to separate type 1 individuals into two 	groups: one that comes from mutation from type 0 individuals and another that comes from recombination between type 0 and type 3 individuals. We need to do the same for the other three types. The precise definitions are given below.

	\begin{enumerate}
	\item A type 1 (or 2) individual is called a \textbf{type 1m (or 2m) ancestor}, if it appears by mutation from a type 0 individual. 
	\item A type 1 (or 2) individual is called a \textbf{type 1r (or 2r) ancestor}, if it appears by recombination between a $b$ (or an $a$) allele from a type 0 individual and an $A$ (or a $B$) allele from a type 3 individual.  
	\item A type 1 individual $x$ is called an offspring of another type 1 individual $y$ if
		\begin{itemize}
		\item $x$ receives the $A$ allele from $y$, or
		\item $x$ receives the $b$ allele from $y$ and receives the $A$ allele from a type 3.
		\end{itemize}
	\item A type 2 individual $x$ is called an offspring of another type 2 individual $y$ if a
		\begin{itemize} 
		\item $x$ receives the $B$ allele from $y$, or
		\item $x$ receives the $a$ allele from $y$ and receives the $B$ allele from a type 3.
		\end{itemize}
	\item A type 1 (or 2) individual is called \textbf{type 1m (or 2m)}, if it descends from a type 1m (or 2m) ancestor. A type 1 (or 2) individual is called \textbf{type 1r (or 2r)}, if it descends from a type 1r (or 2r) ancestor.
	\item A type 3 individual is called a \textbf{type 3m ancestor}, if it appears from mutation from a type 1 individual or a type 2 individual. 
	\item A type 3 individual is called a \textbf{type 3r ancestor}, if it appears by recombination between an $A$ allele from a type 1 individual and a $B$ allele from a type 2 individual.
	\item A type 3 individual $x$ is called an offspring of another type 3 individual $y$ if
		\begin{itemize}
		\item $x$ receives the $A$ allele from $y$, or
		\item $x$ receives the $B$ allele from $y$ and receives the $A$ allele from a type 1 individual.
		\end{itemize}
	\item A type 3 individual is called \textbf{type 3m}, if it descends from a type 3m ancestor. A type 3 individual is called \textbf{type 3r}, if it descends from a type 3r ancestor.
	\item A type 0 individual is called a \textbf{type 0r ancestor}, if it appears from recombination between an $a$ allele from a type 1  individual and a $b$ allele from a type 2 individual.
	\item A type 0 individual $x$ is called an offspring of another type 0 individual $y$ if
		\begin{itemize}
		\item $x$ receives the $a$ allele from $y$, or
		\item $x$ receives the $b$ allele from $y$ and receives the $a$ allele from a type 2. 
		\end{itemize}
	\item  A type 0 individual is called a \textbf{type 0r} if it descends from a type 0r ancestor. 
	\end{enumerate}

	For $i=1,2,3$, we define $X_{im}(t)$ as the number of type $im$ at time $t$, and for $i=0,1,2,3$, we define $X_{ir}(t)$ as the number of type $ir$ at time $t$. Note that for $i=1,2,3$ and $t\geq 0$, we have  $X_i(t)=X_{im}(t)+X_{ir}(t)$. Next, we define $X^{(a,b]}_{im}(t)$ and $X^{(a,b]}_{ir}(t)$ to be the number of type $im$ and $ir$ individuals at time $t$, whose ancestor appears in the time interval $(a,b]$. It follows that if $0\leq t\leq b$, for $i=1,2,3$, we have that $X_{im}^{(0,b]}(t)=X_{im}(t)$, and for $i=0,1,2,3$, we have that $X_{ir}^{(0,b]}(t)=X_{ir}(t)$. We will call an individual type \textbf{im(a,b]} (or \textbf{ir(a,b]}), if it is of type im (or type ir) and its ancestor appears in the time interval $(a,b]$.  Lastly, we define $\tilde X_{im}(t), \tilde X_{ir}(t), \tilde X^{(a,b]}_{im}(t)$, and $\tilde X^{(a,b]}_{ir}(t)$ to be the fractions of type im, ir, im(a,b] and ir(ab] in the population at time $t$ respectively. 

	Now, consider the process $(X^{(a,b]}_{1m}(t),t\geq 0)$, First, we consider the rate that $X^{(a,b]}_{1m}(t)$ increases by 1. There are two ways to increase $X^{(a,b]}_{1m}(t)$.  First, a type 0 individual can mutate to a type 1 individual during the time interval $(a,b]$, creating a type 1m(a,b] ancestor, which occurs at total rate
		\begin{equation} \label{M1ab}
		M_1^{(a,b]}(t)=\mu X_0(t)1_{(a,b]}(t).
		\end{equation} 
	Second, an individual that is not of type 1m(a,b] can die, which occurs at total rate
		\begin{equation}\label{***2.1}
		X_0(t)+(1-s)(X_1(t)-X_{1m}^{(a,b]}(t))+(1-s)X_2(t)+(1-2s)X_3(t),
		\end{equation}
	and the new individual must be a type 1m(a,b]. The probability that recombination doesn't occur and the new individual has type 1m(a,b] is $(1-r)\tilde X^{(a,b]}_{1m}(t)$. If recombination occurs, the new individual can come from combining an $A$ allele from a type 1m(a,b] individual with a $b$ allele from a type 0 or 1 individual, or combining an $A$ allele from a type 3 individual with a $b$ allele from a type 1m(a,b] individual. (Note that recombination between an $A$ allele from a type 3 individual and a $b$ alelle from a type 0 individual creates an ancestor of type 1r.) So, the probability that recombination occurs and the new individual has type 1m(a,b] is
		\begin{equation}\label{***2.2}
		r\Big(\tilde X^{(a,b]}_{1m}(t)\tilde X_0(t)+\tilde X^{(a,b]}_{1m}(t)\tilde X_1(t)+\tilde X_3(t)\tilde X^{(a,b]}_{1m}(t)\Big)=r\tilde X^{(a,b]}_{1m}(t)\Big(\tilde X_0(t)+\tilde X_1(t)+\tilde X_3(t)\Big).
		\end{equation} 
	 Hence, the total rate that the number of descendants of type 1m(a,b] increases by 1 is
		\begin{align*}
		&\Big(X_0(t)+(1-s)(X_1(t)-X_{1m}^{(a,b]}(t))+(1-s)X_2(t)	+(1-2s)X_3(t)\Big)\\
		&\hspace{4 cm} \cdot\Big((1-r)\tilde X_{1m}^{(a,b]}(t)+r\tilde X_{1m}^{(a,b]}(t)(\tilde X_0(t)+\tilde X_1(t)+\tilde X_3(t))\Big). 
		\end{align*}
	Let us define
		\begin{equation} \label{B1mab}
		B_{1m}^{(a,b]}(t)=\Big( \tilde X_0(t)+(1-s)(\tilde X_1(t)-\tilde X_{1m}^{(a,b]}(t))+(1-s)\tilde X_2(t)+(1-2s)\tilde X_3(t)\Big)\Big(1-r\tilde X_2(t)\Big),
		\end{equation}
	and note that $X^{(a,b]}_{1m}(t)$ increases by 1 at rate $M_1^{(a,b]}(t)+B_{1m}^{(a,b]}(t)X_{1m}^{(a,b]}(t)$. 

	Similarly, the rate that the number of type 1m(a,b] individuals decreases by 1 is given by
		\begin{equation}\label{***2.3}
		(1-s)X_{1m}^{(a,b]}(t)\Big(1-(1-r)\tilde X_{1m}^{(a,b]}(t)-r\tilde X_{1m}^{(a,b]}(t)(\tilde X_0(t)+\tilde X_1(t)+\tilde X_3(t))\Big)+\mu X_{1m}^{(a,b]}(t),
		\end{equation}
	where $(1-s)X_{1m}^{(a,b]}(t)$ is the total rate that type 1m(a,b] individuals die at time $t$, 
		$$
		1-(1-r)\tilde X_{1m}^{(a,b]}(t)-r\tilde X_{1m}^{(a,b]}(t)(\tilde X_0(t)+\tilde X_1(t)+\tilde X_3(t))
		$$
	is the probability that we don't create a type 1m(a,b] individual, and $\mu X_{1m}^{(a,b]}(t)$ corresponds to the total rate that type 1m(a,b] mutates to type 3. We define
		\begin{equation} \label{D1mab}
		D_{1m}^{(a,b]}(t)=(1-s)\Big(1-\tilde X_{1m}^{(a,b]}(t)+r\tilde X_2(t)\tilde X_{1m}^{(a,b]}(t)\Big)+\mu,
		\end{equation}
	and note that the number of type 1m(a,b] individuals decreases by 1 at rate $D_{1m}^{(a,b]}(t)X_{1m}^{(a,b]}(t)$.

	We will now consider the process $(X^{(a,b]}_{1r}(t),t\geq 0)$. We will first consider the rate that $X^{(a,b]}_{1r}(t)$ increases by 1. There are two ways to increase $X^{(a,b]}_{1r}(t)$ by 1. First, an individual that is not of type 1r(a,b] dies, and the recombination between an $A$ allele from a type 3 individual and a $b$ allele from a type 0 individual occurs during the time interval $(a,b]$, which creates a type 1r(a,b] ancestor. This occurs at total rate of
		$$
		R_1^{(a,b]}(t)=\Big(X_0(t)+(1-s)\big(X_1(t)-X_{1r}^{(a,b]}(t)\big)+(1-s)X_2(t)+(1-2s)X_3(t)\Big)\Big(r\tilde X_0(t)\tilde X_3(t)1_{(a,b]}(t)\Big).
		$$
	Second, an individuals that is not of type 1r(a,b] dies, and a new type 1r(a,b] individual is born from the type 1r(a,b] individuals at that time. Similar to the way we obtain (\ref{***2.1}) and (\ref{***2.2}), by defining 
		\begin{equation} \label{B1rab}
		B_{1r}^{(a,b]}(t)=\Big( \tilde X_0(t)+(1-s)(\tilde X_1(t)-\tilde X_{1r}^{(a,b]}(t))+(1-s)\tilde X_2(t)+(1-2s)\tilde X_3(t)\Big)\Big(1-r\tilde X_2(t)\Big),
		\end{equation}
	one can see that the rate that $X^{(a,b]}_{1r}(t)$ increases by 1 is $R_1^{(a,b]}(t)+B_{1r}^{(a,b]}(t)X^{(a,b]}_{1r}(t)$.
	
	We will now consider the rate that  $X^{(a,b]}_{1r}(t)$ decreases by 1. One way that $X^{(a,b]}_{1r}(t)$ decreases by 1 is when a type 1r(a,b] individual dies and the new individual is not of type 1r(a,b] (i.e, the new individual is not born from a type 1r(a,b] individual, and it is not a type 1r(a,b] ancestor). Another way is when a type $1r(a,b]$ individual mutates to a type 3 individual. By the same reason we used to obtain (\ref{***2.3}), the rate that $X^{(a,b]}_{1r}(t)$ decreases by 1 is
		$$
		(1-s)X_{1r}^{(a,b]}(t)\Big(1-(1-r)\tilde X_{1r}^{(a,b]}(t)-r\tilde X_{1r}^{(a,b]}(t)(\tilde X_0(t)+\tilde X_1(t)+\tilde X_3(t))-r\tilde X_0(t)\tilde X_3(t)1_{(a,b]}(t)\Big)+\mu X_{1r}^{(a,b]}(t),
		$$
	and note that the term $	r\tilde X_0(t)\tilde X_3(t)1_{(a,b]}(t)$ is precisely the probability that a type 1r(a,b] ancestor is created. By defining	
		\begin{equation} \label{D1rab}
		D_{1r}^{(a,b]}(t)=(1-s)\Big(1-\tilde X_{1r}^{(a,b]}(t)+r\tilde X_2(t)\tilde X_{1r}^{(a,b]}(t)-r\tilde X_0(t)\tilde X_3(t)1_{(a,b]}(t)\Big)+\mu,
		\end{equation}
	one can see that the rate that $X^{(a,b]}_{1r}(t)$ decreases by 1 is $D_{1r}^{(a,b]}(t)X_{1r}^{(a,b]}(t)$.
	
	Now, we define
		\begin{align}
		B_{2m}^{(a,b]}(t)&=\Big( \tilde X_0(t)+(1-s)\tilde X_1(t)+(1-s)(\tilde X_2(t)-\tilde X_{2m}^{(a,b]}(t))+(1-2s)\tilde X_3(t)\Big)\Big(1-r\tilde X_1(t)\Big),\nonumber\\
		D_{2m}^{(a,b]}(t)&=(1-s)\Big(1-\tilde X_{2m}^{(a,b]}(t)+r\tilde X_1(t)\tilde X_{2m}^{(a,b]}(t)\Big)+\mu,\nonumber\\
		M_2^{(a,b]}(t)&=\mu X_0(t)1_{(a,b]}(t),\nonumber\\
		B_{2r}^{(a,b]}(t)&=\Big( \tilde X_0(t)+(1-s)\tilde X_1(t)+(1-s)(\tilde X_2(t)-\tilde X_{2r}^{(a,b]}(t))+(1-2s)\tilde X_3(t)\Big)\Big(1-r\tilde X_1(t)\Big),\nonumber\\
		D_{2r}^{(a,b]}(t)&=(1-s)\Big(1-\tilde X_{2r}^{(a,b]}(t)+r\tilde X_1(t)\tilde X_{2r}^{(a,b]}(t)-r\tilde X_0(t)\tilde X_3(t)1_{(a,b]}(t)\Big)+\mu,\nonumber\\
		R_2^{(a,b]}(t)&=\Big( X_0(t)+(1-s)X_1(t)+(1-s)(X_2(t)-X_{2r}^{(a,b]}(t))+(1-2s)X_3(t)\Big)\Big(r\tilde X_0(t)\tilde X_3(t)1_{(a,b]}(t)\Big),\nonumber\\
		B_{3m}^{(a,b]}(t)&=\Big( \tilde X_0(t)+(1-s)(\tilde X_1(t)+\tilde X_2(t))+(1-2s)\big(\tilde X_3(t)-\tilde X_{3m}^{(a,b]}(t)\big)\Big)\Big(1-r\tilde X_0(t)\Big),\label{B3mab}\\
		D_{3m}^{(a,b]}(t)&=(1-2s)\Big(1-\tilde X_{3m}^{(a,b]}(t)+r\tilde X_0(t)\tilde X_{3m}^{(a,b]}(t)\Big),\label{D3mab}\\
		M_3^{(a,b]}(t)&=\mu (X_1(t)+X_2(t))1_{(a,b]}(t),\label{M3ab}\\
		B_{3r}^{(a,b]}(t)&=\Big( \tilde X_0(t)+(1-s)(\tilde X_1(t)+\tilde X_2(t))+(1-2s)\big(\tilde X_3(t)-\tilde X_{3r}^{(a,b]}(t)\big)\Big)\Big(1-r\tilde X_0(t)\Big),\label{B3rab}\\
		D_{3r}^{(a,b]}(t)&=(1-2s)\Big(1-\tilde X_{3r}^{(a,b]}(t)+r\tilde X_0(t)\tilde X_{3r}^{(a,b]}(t)-r\tilde X_1(t)\tilde X_2(t)1_{(a,b]}(t) \Big),\label{D3rab}\\
		R_3^{(a,b]}(t)&=\Big(X_0(t)+(1-s)(X_1(t)+X_2(t))+(1-2s)\big( X_3(t)-X_{3r}^{(a,b]}(t)\big)\Big)\Big(r\tilde X_1(t)\tilde X_2(t)1_{(a,b]}(t)\Big),\label{R3ab}\\
		B^{(a,b]}_{0r}(t)&=\Big((\tilde X_0(t)-\tilde X^{(a,b]}_{0r}(t))+(1-s)(\tilde X_1(t)+\tilde X_2(t))+(1-2s)\tilde X_3(t)\Big)\Big(1-r\tilde X_3(t)\Big),\nonumber\\
		D^{(a,b]}_{0r}(t)&=\Big(1-\tilde X^{(a,b]}_{0r}(t)+r\tilde X_3(t)\tilde X^{(a,b]}_{0r}(t)-r\tilde X_1(t)\tilde X_2(t)1_{(a,b]}(t)\Big) +2\mu, \nonumber\\
		R^{(a,b]}_0(t)&=\Big((X_0(t)-X^{(a,b]}_{0r}(t))+(1-s)(X_1(t)+ X_2(t))+(1-2s)X_3(t)\Big)\Big(r\tilde X_1(t)\tilde X_2(t)1_{(a,b]}(t)\Big).\nonumber
		\end{align}
	By analogy, one can check that for $i=2, 3$, we have that $X_{im}^{(a,b]}(t)$ increases by 1 at rate $M_i^{(a,b]}(t)+B_{im}^{(a,b]}(t)X_{im}^{(a,b]}(t)$ and decreases by 1 at rate $D_{im}^{(a,b]}(t)X_{im}^{(a,b]}(t)$. Also, for $i=0, 2$ and 3, $X_{ir}^{(a,b]}(t)$ increases by 1 at rate $R_i^{(a,b]}(t)+B_{ir}^{(a,b]}(t)X_{ir}^{(a,b]}(t)$ and decreases by 1 at rate $D_{ir}^{(a,b]}(t)X_{ir}^{(a,b]}(t)$.
	
	For $i=1,2,3$, and $0\leq a < b\wedge t$, we define $G_i(t)=B_{im}^{(a,b]}(t)-D_{im}^{(a,b]}(t)$, which is the growth rate of the type im(a,b] population at time $t$. For $i=0,1,2,3$, and $0\leq a < b\wedge t$, we define $G_{ir}^{(a,b]}(t)=B_{ir}^{(a,b]}(t)-D_{ir}^{(a,b]}(t)$. This is the growth rate of the type ir(a,b] population at time $t$. Note that $G_i(t)$ does not depended on the interval $(a,b]$, because from (\ref{B1mab}), (\ref{D1mab}), and the fact that $\tilde X_0(t)+\tilde X_1(t)+\tilde X_2(t)+\tilde X_3(t) =1$, 
		\begin{align}
		G_1(t)&=B_{1m}^{(a,b]}(t)-D_{1m}^{(a,b]}(t)\nonumber \\
		&=\Big(1-(1-s)\tilde X_{1m}^{(a,b]}(t)-s\tilde X_1(t)-s\tilde X_2(t)-2s\tilde X_3(t)\Big)\Big(1-r\tilde X_2(t)\Big)\nonumber \\
		&\hspace{0.5 cm}-(1-s)\Big(1-\tilde X_{1m}^{(a,b]}(t)+r\tilde X_2(t)\tilde X_{1m}^{(a,b]}(t)\Big)-\mu\nonumber \\
		&=s\Big(1-\tilde X_1(t)-\tilde X_2(t)-2\tilde X_3(t)\Big)-r\tilde X_2(t)\Big(1-s\tilde X_1(t)-s\tilde X_2(t)-2s\tilde X_3(t)\Big)-\mu. \label{G1}
		\end{align}
	Similarly, we have
		\begin{align} 
		G_2(t)&=s\Big(1-\tilde X_1(t)-\tilde X_2(t)-2\tilde X_3(t)\Big)-r\tilde X_1(t)\Big(1-s\tilde X_1(t)-s\tilde X_2(t)-2s\tilde X_3(t)\Big)-\mu,\nonumber\\
		G_3(t)&=s\Big(2-\tilde X_1(t)-\tilde X_2(t)-2\tilde X_3(t)\Big)-r\tilde X_0(t)\Big(1-s\tilde X_1(t)-s\tilde X_2(t)-2s\tilde 	X_3(t)\Big).\label{G3}
		\end{align}
	Also, by similar calculation, we have
		\begin{align}
		G^{(a,b]}_{1r}(t)&=G_1(t)+(1-s)r \tilde X_0(t)\tilde X_3(t)1_{(a,b]}(t)\label{G1r}\\
		G^{(a,b]}_{2r}(t)&=G_2(t)+(1-s)r \tilde X_0(t)\tilde X_3(t)1_{(a,b]}(t)\label{G2r}\\
		G^{(a,b]}_{3r}(t)&=G_3(t)+(1-2s)r\tilde X_1(t)\tilde X_2(t)1_{(a,b]}(t)\label{G3r}\\
		G^{(a,b]}_{0r}(t)&=-s\Big(\tilde X_1(t)+\tilde X_2(t)+2\tilde X_3(t)\Big)-r\tilde X_3(t)\Big(1-s\tilde X_1(t)-s\tilde X_2(t)-2s\tilde X_3(t)\Big)\nonumber \\
&\hspace{0.7 cm}-2\mu+r\tilde X_1(t)\tilde X_2(t)1_{(a,b]}(t). \label{G0r} 
		\end{align}

	From the fact that $\tilde X_0(t)+\tilde X_1(t)+\tilde X_2(t)+\tilde X_3(t)=N$, and $s\ll 1$, it follows that for sufficiently large $N$,
		\begin{align}
		R_1^{(a,b]}(t)&\leq Nr\tilde X_0(t)\tilde X_3(t)1_{(a,b]}(t),\label{R1<}\\
		R_2^{(a,b]}(t)&\leq Nr\tilde X_0(t)\tilde X_3(t)1_{(a,b]}(t),\nonumber\\
		R_3^{(a,b]}(t)&\leq Nr\tilde X_1(t)\tilde X_2(t)1_{(a,b]}(t),\label{R3<}\\
		R_0^{(a,b]}(t)&\leq Nr\tilde X_1(t)\tilde X_2(t)1_{(a,b]}(t).\label{R0<}
		\end{align}
	
	Lastly, for $i=0,1,2,3$ and $0\leq a\leq t$,  we define $X_i^{[a]}(t)$ to be the number of type $i$ individuals at time $t$ that descend from one of the type $i$ individuals at time $a$. It follows that for $0\leq a \leq t \leq b$ and $i=1,2,3$, 
		$$
		X_i(t)=X_i^{[a]}(t)+X_{im}^{(a,b]}(t)+X_{ir}^{(a,b]}(t),
		$$
	and
		$$
		X_0(t)=X_0^{[a]}(t)+X_{0r}^{(a,b]}(t).
		$$
	Following the argument we used to obtain $B_{im}^{(a,b]}(t)$ and $D_{im}^{(a,b]}(t)$, for $0\leq a \leq t$, we define
		\begin{align}
		B_1^{[a]}(t)&=\Big( \tilde X_0(t)+(1-s)(\tilde X_1(t)-X_1^{[a]}(t))+(1-s)\tilde X_2(t)+(1-2s)\tilde X_3(t)\Big)\Big(1-r\tilde X_2(t)\Big),\nonumber\\
		D_1^{[a]}(t)&=(1-s)\Big(1-\tilde X_1^{[a]}(t)+r\tilde X_2(t)\tilde X_1^{[a]}(t)\Big)+\mu,\nonumber\\
		B_2^{[a]}(t)&=\Big( \tilde X_0(t)+(1-s)\tilde X_1(t)+(1-s)(\tilde X_2(t)-X_2^{[a]}(t))+(1-2s)\tilde X_3(t)\Big)\Big(1-r\tilde X_1(t)\Big),\nonumber\\
		D_2^{[a]}(t)&=(1-s)\Big(1-\tilde X_2^{[a]}(t)+r\tilde X_1(t)\tilde X_2^{[a]}(t)\Big)+\mu,\nonumber\\
		B_3^{[a]}(t)&=\Big( \tilde X_0(t)+(1-s)\tilde X_1(t)+(1-s)\tilde X_2(t)+(1-2s)(\tilde X_3(t)-X_3^{[a]})(t)\Big)\Big(1-r\tilde X_0(t)\Big),\label{B3a}\\
		D_3^{[a]}(t)&=(1-2s)\Big(1-\tilde X_3^{[a]}(t)+r\tilde X_0(t)\tilde X_3^{[a]}(t)\Big),\label{D3a}\\
		B_0^{[a]}(t)&=\Big( (\tilde X_0(t)-X_0^{[a]}(t))+(1-s)\tilde X_1(t)+(1-s)\tilde X_2(t)+(1-2s)\tilde X_3(t)\Big)\Big(1-r\tilde X_3(t)\Big),\nonumber\\
		D_0^{[a]}(t)&=\Big(1-\tilde X_0^{[a]}(t)+r\tilde X_3(t)\tilde X_0^{[a]}(t)\Big)+2\mu,\nonumber
\end{align} 
	and note that for $i=0,1,2,3$, the process $\big(X_i^{[a]}(t),t\geq a)$ increases by 1 at rate $B_i^{[a]}(t)X_i^{[a]}(t)$, and decreases by 1 at rate $D_i^{[a]}(t)X_i^{[a]}(t)$. Also, for all $t\geq a$ and $i=1,2,3$, we can check that 
		$$
		B_i^{[a]}(t)-D_i^{[a]}(t)=G_i(t).
		$$
	Lastly, we define $G_0(t)=B_0^{[a]}(t)-D_0^{[a]}(t)$ for all $t\geq a$. It follows that
		\begin{equation} \label{G0}
		G_0(t)= -s\Big(\tilde X_1(t)+\tilde X_2(t)+2\tilde X_3(t)\Big)-r\tilde X_3(t)\Big(1-s\tilde X_1(t)-s\tilde X_2(t)-2s\tilde X_3(t)\Big)-2\mu,
		\end{equation}
	and note that from (\ref{G0r}),
		\begin{equation}\label{G0,0r}
		G_{0r}^{(a,b]}(t)=G_0(t)+r\tilde X_1(t)\tilde X_2(t)1_{(a,b]}(t).
		\end{equation}

\section{Important Martingales and Submartingales} \label{M}
	In this section, we will define several martingales and submartingales that will be used frequently in the proof. First, for $i=1,2,3$ and for $0\leq a <b$, when $0\leq t < a$, we define $Z_{im}^{(a,b]}(t)=0$, and when $0\leq a < t$, we define
		\begin{equation} \label{Zm}
		Z_{im}^{(a,b]}(t)=e^{-\int_a^tG_i(v)dv}X_{im}^{(a,b]}(t)-\int_a^t M_i^{(a,b]}(u)e^{-\int_a^uG_i(v)dv}du.
		\end{equation}
	 Also, for $i=0,1,2,3$ and for $0\leq a <b$, when $0\leq t < a$, we define $Z_{ir}^{(a,b]}(t)=0$, and when $0\leq a < t$, we define
		\begin{equation} \label{Zr}
		Z_{ir}^{(a,b]}(t)=e^{-\int_a^tG_{ir}^{(a,b]}(v)dv}X_{ir}^{(a,b]}(t)-\int_a^t R_i^{(a,b]}(u)e^{-\int_a^uG_{ir}^{(a,b]}(v)dv}du.
		\end{equation}
	It follows that for $t\geq a$,
		\begin{align}
		X^{(a,b]}_{im}(t)&=\int_a^t M_i^{(a,b]}(u)e^{\int_u^tG_i(v)dv}du+Z^{(a,b]}_{im}(t)e^{\int_a^tG_i(v)dv},\label{Xm}\\
		X^{(a,b]}_{ir}(t)&=\int_a^t R^{(a,b]}_i(u)e^{\int_u^tG^{(a,b]}_{ir}(v)dv}du+Z^{(a,b]}_{ir}	(t)e^{\int_a^tG^{(a,b]}_{ir}(v)dv}.\label{Xr}
		\end{align}
	Let $(\mathcal{F}_t)_{t\geq 0}$ be the natural filtration of the process $((X_0(t),X_1(t),X_2(t),X_3(t)),t\geq 0)$.

\begin{proposition} \label{ZMar}
	For $i=1, 2,3$, the process $(Z_{im}^{(a,b]}(t),t\geq a)$ is a mean-zero martingale, and for $a \leq t$,
		$$
		\textup{Var}\Big(Z_{im}^{(a,b]}(t)\Big)=E\bigg[\int_a^t e^{-2\int_a^u G_i(v)dv}\Big(M_i^{(a,b]}(u)+\Big(B_{im}^{(a,b]}(u)+D_{im}^{(a,b]}(u)\Big)X_{im}^{(a,b]}(u)\Big) du\bigg].
		$$
	Also, For $i=0,1,2,3$ the process $(Z_{ir}^{(a,b]}(t),t\geq a)$ is a mean-zero martingale, and for $a \leq t$, 
		$$
		\textup{Var}\Big(Z_{ir}^{(a,b]}(t)\Big)=E\bigg[\int_a^t e^{-2\int_a^u G_{ir}^{(a,b]}(v)dv}\Big( R_i^{(a,b]}(u)+\Big(B_{ir}^{(a,b]}(u)+D_{ir}^{(a,b]}(u)\Big)X_{ir}^{(a,b]}(u)\Big) du\bigg].
		$$
	Moreover, if $T$ is a stopping time and $T\geq a$, then for $i=1, 2,3$, the process $(Z_{im}^{(a,b]}(t\wedge T),t\geq a)$ is a mean-zero martingale, and for $a \leq t$, 
		$$
		\textup{Var}\Big(Z_{im}^{(a,b]}(t\wedge T)\Big)=E\bigg[\int_a^{t\wedge T} e^{-2\int_a^u G_i(v)dv}\Big(M_i^{(a,b]}(u)+\Big(B_{im}^{(a,b]}(u)+D_{im}^{(a,b]}(u)\Big)X_{im}^{(a,b]}(u)\Big) du\bigg].
		$$
	Also, for $i=0,1,2,3$, the process $(Z_{ir}^{(a,b]}(t\wedge T),t\geq a)$ is a mean-zero martingale, and for $a \leq t$, 
		$$
		\textup{Var}\Big(Z_{ir}^{(a,b]}(t\wedge T)\Big)=E\bigg[\int_a^{t\wedge T} e^{-2\int_a^u G_{ir}^{(a,b]}(v)dv}\Big(R_i^{(a,b]}(u)+\Big(B_{ir}^{(a,b]}(u)+D_{ir}^{(a,b]}(u)\Big)X_{ir}^{(a,b]}(u)\Big) du\bigg].
		$$
\end{proposition}

\begin{proof}
	The technique used in this proof was previously used in section 5.1 of \cite{Sch}. We will prove the result for the process $(Z^{(a,b]}_{1m}(t),t\geq 0)$. The results for the other processes can be proved in the same manner.

	For $t\geq a$, let $U(t)$ be the number of times in $[a,t]$ that the number of type $1m(a,b]$ individuals increases, and let $V(t)$ be the number of times in $[a,t]$ that the number of type $1m(a,b]$ individuals decreases. Then, $X^{(a,b]}_{1m}(t)=U(t)-V(t)$. Next, we define
		\begin{align}
		W_+(t)&=U(t)-\int_a^t\Big(M_1^{(a,b]}(u)+B^{(a,b]}_{1m}(u)X^{(a,b]}_{1m}(u)\Big)du,\label{W+}\\
		W_-(t)&=V(t)-\int_a^tD^{(a,b]}_{1m}(u)X^{(a,b]}_{1m}(u)du, \label{W-}
		\end{align}
	and $W(t)=W_+(t)-W_-(t)$, for all $t\geq a$. Because $M_1^{(a,b]}(u)+B^{(a,b]}_{1m}(u)X^{(a,b]}_{1m}(u)$ and $D^{(a,b]}_{1m}(u)X^{(a,b]}_{1m}(u)$ are exactly the rates that the process $(X_{1m}^{(a,b]}(t),t\geq a)$ increases and decreases by 1 at time $u$, and both $U(a)$ and $V(a)$ are 0, both the process  $(W_+(t),t\geq a)$ and the process $(W_-(t),t\geq a)$ are mean-zero martingales. It follows that the processes $(W(t),t\geq a)$ and $(W_+(t)+W_-(t),t\geq a)$ are also mean-zero martingales. Since $W$ is locally of bounded variation, its   quadratic variation is
		$$
		[W](t)=\sum_{u\in[a,t]}\big(W(u)-W(u-)\big)^2=U(t)+V(t).
		$$
	Now, consider the process $(\langle W \rangle(t), t\geq a)$. The process $([W](t)-\langle W \rangle (t),t\geq a)$ is mean-zero martingale, by the definition of the sharp bracket. From equations (\ref{W+}), (\ref{W-}), and the fact that $\big(W_+(t)+W_-(t),t\geq a\big)$ is a mean-zero martingale, we have that
		$$
		\langle W \rangle (t)=\int_a^t\Big(M_i^{(a,b]}(u)+\Big(B_{im}^{(a,b]}(u)+D_{im}^{(a,b]}(u)\Big)X_{im}^{(a,b]}(u)\Big)du.
		$$

	Now, for $t\geq a$, we define 
  		$$
  		I(t)=e^{-\int_a^tG_{1}(v)dv}.
  		$$
	Because both $(X_{1m}^{(a,b]}(t),t\geq a)$ and $(I(t),t\geq a)$ are semimartingales, such that $(I(t),t\geq 0)$ has continuous paths and the process $(X_{1m}^{(a,b]}(t),t\geq a)$ is locally of bounded variation, 
  		$$
  		[X^{(a,b]}_{1m},I](t)=0 \hspace{.1 in}\mbox{for all $t$ a.s.}
  		$$   
	Also, because
  		$$
  		X^{(a,b]}_{1m}(t)=U(t)-V(t)=W(t)+\int_a^t\Big(M^{(a,b]}_1(u)+G_1(u)X^{(a,b]}_{1m}(u)\Big)du
  		$$
for all $t\geq a$, we have
  		$$
  		\int_a^tI(u)dX^{(a,b]}_{1m}(u)=\int_a^tI(u)dW(u)+\int_a^tI(u)\Big(M^{(a,b]}_1(u)+G_1(u)X^{(a,b]}_{1m}(u)\Big)du.
  		$$
	Using the Integration by Parts formula, we have
		\begin{align}
		I(t)X^{(a,b]}_{1m}(t) 
		&=I(a)X^{(a,b]}_{1m}(a)+\int_a^tX^{(a,b]}_{1m}(u-)dI(u)+\int_a^tI(u-)dX^{(a,b]}_{1m}(u)+[X^{(a,b]}_{1m},I](t)\nonumber\\           
		& =0-\int_a^tX^{(a,b]}_{1m}(u)G_1(u)I(u)du+\int_a^tI(u)dX^{(a,b]}_{1m}(u)+0 \nonumber\\           
		&=\int_a^tM^{(a,b]}_1(u)I(u)du+\int_a^tI(u)dW(u). \label{-4.2}
		\end{align}
	Therefore, from (\ref{Zm}) and (\ref{-4.2}),
		\begin{equation}
		Z^{(a,b]}_{1m}(t)=I(t)X^{(a,b]}_{1m}(t)-\int_a^tM^{(a,b]}_1(u)I(u)du=\int_a^tI(u)dW(u). \label{Z=intIdW}
		\end{equation}
	From (\ref{B1mab}) and (\ref{D1mab}), we have $B_1(t)\in [0,1]$ and $D_1(t)\in [0,1+\mu]$ for all $t\geq a$. So, $G_1(t)\in[-1-\mu,1]$ for all $t\geq a$. Thus,
		\begin{align}
		\int_a^t I^2(u)d\langle W\rangle(u)
		&=\int_a^te^{-2\int_a^uG_1(v)dv}\Big(M^{(a,b]}_1(u)+\Big(B^{(a,b]}_{1m}(u)+D^{(a,b]}_{1m}(u)\Big)X^{(a,b]}_{1m}(u)\Big)du \label{int1}\\
		&\leq \int_a^te^{2(1+\mu)(u-a)}\cdot\big(\mu N+(2+\mu)N\big)du\nonumber\\
		&=\big(e^{2(1+\mu)(t-a)}-1\big)N, \nonumber
		\end{align}
	for all $t\geq a$. Hence, for each $t\geq a$, we have $E[\int_0^tI^2(u)d\langle W \rangle(u)] < \infty$.
Therefore, from (\ref{Z=intIdW}), the process $\big(Z^{(a,b]}_{1m}(t),t\geq 0\big)$ is a square integrable martingale with 
		\begin{equation} \label{<Z>}
		\big\langle Z^{(a,b]}_{1m}\big\rangle(t)= \int_a^t I^2(u)d\langle W\rangle(u).
		\end{equation}
	This process has mean zero, because $Z^{(a,b]}_{1m}(a)=0$. By Corollary 8.25 of \cite{Kle}, $\textup{Var}\big(Z^{(a,b]}_{1m}(t)\big)=E\big[\big(Z^{(a,b]}_{1m}(t)\big)^2\big]=E\big[\big\langle Z^{(a,b]}_{1m}\big\rangle(t)\big]$, and this proves the variance formula by using (\ref{int1}) and (\ref{<Z>}).
	Lastly, because a stopped martingale is a martingale, the process $(Z_{1m}^{(a,b]}(t\wedge T),t\geq a)$ is a mean-zero martingale, and by the same argument above, we can get the variance formula for the process $(Z_{1m}^{(a,b]}(t\wedge T),t\geq a)$.
\end{proof}

	Since the process $((X_0(t),X_1(t)-X^{(a,b]}_{1m}(t),X^{(a,b]}_{1m}(t),X_2(t),X_3(t)),t\geq 0)$ is a continuous-time Markov chain, combining Proposition \ref{ZMar} and Markov property yields the following result.

\begin{corollary} \label{ZMar*}
	If $T$ is a stopping time and $T\geq a$, then for $i=1,2,3$ and $a \leq t$, 
		$$
		\textup{Var}\Big(Z_{im}^{(a,b]}(t\wedge T)\Big|\mathcal{F}_a\Big)=E\bigg[\int_a^{t\wedge T} e^{-2\int_a^u G_i(v)dv}\Big(M_i^{(a,b]}(u)+\Big(B_{im}^{(a,b]}(u)+D_{im}^{(a,b]}(u)\Big)X_{im}^{(a,b]}(u)\Big) du\bigg|\mathcal{F}_a\bigg],
		$$
	and for $i=0,1,2,3$ and for $a \leq t$, 
		$$
		\textup{Var}\Big(Z_{ir}^{(a,b]}(t\wedge T)\Big|\mathcal{F}_a\Big)=E\bigg[\int_a^{t\wedge T} e^{-2\int_a^u G_{ir}^{(a,b]}(v)dv}\Big( R_i^{(a,b]}(u)+\Big(B_{ir}^{(a,b]}(u)+D_{ir}^{(a,b]}(u)\Big)X_{ir}^{(a,b]}(u)\Big) du\bigg|\mathcal{F}_a\bigg].
		$$
\end{corollary}

	Now, for $i=0,1,2,3$ and $0\leq a \leq t$,  we define
		\begin{equation} \label{Zi}
		Z_i^{[a]}(t)=e^{-\int_a^tG_i(v)dv}X_i^{[a]}(t).
		\end{equation}
	By a similar argument to the one used in proving Proposition \ref{ZMar} and Corollary \ref{ZMar*}, we get the following result.

\begin{proposition} \label{ZiMar}
	If $T$ is a stopping time with $T\geq a$, then for $i=0, 1, 2,3$, the process $(Z_i^{[a]}(t),t\geq a)$ is a  martingale, and for all $a\leq t$,
	$$
	\textup{Var}\Big(Z_i^{[a]}(t\wedge T)\Big|\mathcal{F}_a\Big)=E\bigg[\int_a^{t\wedge T} e^{-2\int_a^u G_i(v)dv}\Big(B_i^{[a]}(u)+D_i^{[a]}(u)\Big)X_i^{[a]}(u) du\bigg|\mathcal{F}_a\bigg].
	$$
\end{proposition}

	Lastly, for $i=1,2,3$, for $0\leq a <b$ and $a\leq t$, let us define
		\begin{equation} \label{Wm}
		W^{(a,b]}_{im}(t)=e^{-\int_a^t G_i(v)dv}X_{im}^{(a,b]}(t),
		\end{equation}
and for $i=0,1,2,3$, for $0\leq a <b$ and $a\leq t$, we define
		$$
		W^{(a,b]}_{ir}(t)=e^{-\int_a^t G^{(a,b]}_{ir}(v)dv}X_{ir}^{(a,b]}(t).
		$$

\begin{proposition}\label{W}
	If $T$ is a stopping time and $T\geq a$, for $i=1, 2,3$, the process $(W_{im}^{(a,b]}(t\wedge T),t\geq a)$ is a submartingale, and for $a\leq t$,
		$$
		E\Big[W_{im}^{(a,b]}(t\wedge T)\Big|\mathcal{F}_a\Big]=E\bigg[\int_a^{t\wedge T} M_i^{(a,b]}(u)e^{-\int_a^uG_i(v)dv} du\bigg|\mathcal{F}_a\bigg].
		$$
	For $i=0,1,2,3$ the process $(W_{ir}^{(a,b]}(t\wedge T),t\geq a)$ is a submartingale, and for $a\leq t$,
		$$
		E\Big[W_{ir}^{(a,b]}(t\wedge T)\Big|\mathcal{F}_a\Big]=E\bigg[\int_a^{t\wedge T} R_i^{(a,b]}(u)e^{-\int_a^uG_{ir}^{(a,b]}(v)dv} du\bigg|\mathcal{F}_a\bigg].
		$$
\end{proposition}

\begin{proof} 
	Consider the process $\big(W_{im}^{(a,b]}(t\wedge T),t\geq a\big)$. Because $B_{im}^{(a,b]}(t)\in[0,1]$ and $D_{im}^{(a,b]}(t)\in[0,1+\mu]$, we have that $G_{im}^{(a,b]}(t)\in[-1-\mu,1]$. Thus,
		$$
		W^{(a,b]}_{im}(t\wedge T)=e^{-\int_a^{t\wedge T} G_i(v)dv}X_{im}^{(a,b]}(t\wedge T)\in\big[0,e^{(1+\mu)(t-a)}N\big],
		$$
	for all $t\geq a$. So, $E\big[W^{(a,b]}_{im}(t\wedge T)\big]<\infty$ for all $t\geq a$.

	From (\ref{Zm}) and (\ref{Wm}), for all $t\geq a$,
		$$
		W_{im}^{(a,b]}(t\wedge T)=Z_{im}^{(a,b]}(t\wedge T)+\int_a^{t\wedge T} M_i^{(a,b]}(u)e^{-\int_a^uG_i(v)dv} du.
		$$
	For $a\leq t' <t$, by Proposition \ref{ZMar}, we have
		\begin{align}
		E\Big[W^{(a,b]}_{im}(t\wedge T)\Big| \mathcal{F}_{t'}\Big]
		&=E\Big[Z^{(a,b]}_{im}(t'\wedge T)\Big| \mathcal{F}_{t'}\Big]+E\Big[\int_a^{t\wedge T} M_i^{(a,b]}(u)e^{-\int_a^uG_i(v)dv} du\Big| \mathcal{F}_{t'}\Big]\label{***4.1}\\
		&=Z^{(a,b]}_{im}(t'\wedge T)+\int_a^{t'\wedge T} M_i^{(a,b]}(u)e^{-\int_a^uG_i(v)dv} du\nonumber\\
		&\hspace{1 cm}+E\Big[\int_{t'\wedge T}^{t\wedge T} M_i^{(a,b]}(u)e^{-\int_a^uG_i(v)dv} du\Big| \mathcal{F}_{t'}\Big]\nonumber\\
		&=W^{(a,b]}_{im}(t'\wedge T)+E\Big[\int_{t'\wedge T}^{t\wedge T} M_i^{(a,b]}(u)e^{-\int_a^uG_i(v)dv} du\Big| \mathcal{F}_{t'}\Big]\nonumber\\
		&\geq W^{(a,b]}_{im}(t'\wedge T).\nonumber
		\end{align}
	Thus, the process $(W_{im}^{(a,b]}(t\wedge T),t\geq a)$ is a submartingale. From (\ref{***4.1}) and from the fact that the process $(Z_{im}^{(a,b]}(t\wedge T),t\geq a)$ is a mean-zero martingale by Proposition \ref{ZMar}, 
		\begin{align*}
		E\Big[W^{(a,b]}_{im}(t\wedge T)\Big| \mathcal{F}_a\Big]
		&=E\Big[Z^{(a,b]}_{im}(t\wedge T)\Big| \mathcal{F}_a\Big]+E\Big[\int_a^{t\wedge T} M_i^{(a,b]}(u)e^{-\int_a^uG_i(v)dv} du\Big| \mathcal{F}_a\Big]\\
		&=Z^{(a,b]}_{im}(a)+E\Big[\int_a^{t\wedge T} M_i^{(a,b]}(u)e^{-\int_a^uG_i(v)dv} du\Big| \mathcal{F}_a\Big]\\
		&=E\Big[\int_a^{t\wedge T} M_i^{(a,b]}(u)e^{-\int_a^uG_i(v)dv} du\Big| \mathcal{F}_a\Big].
		\end{align*}
	The proof for the process $W^{(a,b]}_{ir}$ can be done by a similar argument. 
\end{proof}

\section{Phase 1 and the proof of Proposition \ref{@t1}}\label{phase1}
\subsection{Notations}
	First, note that to prove Propositions \ref{@t1}, \ref{@t2}, \ref{@t3} and \ref{@t4}, it is enough to prove that they hold for all small values of $\epsilon$ and $\delta$. We choose $\epsilon$ and $\delta$ as follow:
		\begin{equation}\label{eps}
		\epsilon \in \Big(0,\frac{1}{16}\Big),
		\end{equation}  	
  	and
  		\begin{equation}\label{delta}
  		\delta \in \Big(0,\frac{1}{4}\Big).
  		\end{equation}
  	We will now define several constants, fixed times, and stopping times. In both the recombination dominating case and the mutation dominating case, we pick the following  constants:
		\begin{align}
    		K&> \frac{6}{\epsilon}\label{K}\\
    		C_1&>\ln\Big(\frac{5K}{\epsilon}\Big)\vee \ln		\Big(\frac{8}{\delta^2}\Big), \label{C1}\\
    		C_{0,m}&>2\ln\Big(\frac{2K}{\epsilon}\Big), \label{C0m}\\
    		C_{0,m}^+&>C_{0,m} \vee \bigg(14e^{-C_1}+\ln\bigg(\frac{48K}{\epsilon (1-\delta^2)^2}\bigg)\bigg),\label{C0m+}\\ 
    		C_{0,r}&>\ln\Big(\frac{K^2}{\epsilon}\Big)\vee (C_1+\ln4),\label{C0r}\\
    		\eta&=2Ke^{-C_1}.\label{n}
   		\end{align}
	Next, we define several fixed times as follows:
    		\begin{equation}\label{t0r}
    		t_{0,r}=
    		\begin{cases}
    		\displaystyle{\frac{1}{s}\ln\Big(\frac{s}{\mu\sqrt{Nr}}\Big)-\frac{C_{0,r}}{s}} & \text{in the recombination dominating case}\\
    		\displaystyle{\frac{1}{s}\ln\Big(\frac{s}{\mu\sqrt{Nr}}\Big)-\frac{C_{0,r}}{s}} & \text{in the mutation dominating case and when $Nr\geq e$}\\
    		\displaystyle{\frac{1}{s}\ln\Big(\frac{s}{\mu}\Big)-\frac{C_{0,r}}{s}} &\text{in the mutation dominating case and when $Nr< e$},
    		\end{cases}
    		\end{equation}
    	and in both cases, we define
    	
    		\begin{align}
    		t_{0,m}&=\frac{1}{s}\ln\Big(\frac{s}{N\mu^2}\Big)-\frac{C_{0,m}}{s},\label{t0m}\\
    		t_{0,m}^+&=\frac{1}{s}\ln\Big(\frac{s}{N\mu^2}\Big)+\frac{C_{0,m}^+}{s},\label{t0m+}\\
    		t_1&=\frac{1}{s}\ln\Big(\frac{s}{\mu}\Big)-\frac{C_1}{s}.\label{t1}
   		\end{align}
   		
	It follows from these definitions, the fact that $1\ll \mu$, and the fact that $1 \ll Nr$ in the recombination dominating case that for sufficiently large $N$, we have $0<t_{0,m}<t_{0,m}^+<t_1$ and $0<t_{0,r}<t_1$. Now, in both cases, we define the following stopping times:
    		\begin{align}
     	T_1&=\inf\Big\{t\geq 0:X_1(t)\geq\frac{KN\mu}{s}e^{st}\Big\} ,\label{T1}\\
     	T_2&=\inf\Big\{t\geq 0:X_2(t)\geq\frac{KN\mu}{s}e^{st}\Big\},\label{T2}\\
     	T_3&=\inf\Big\{t \geq 0: X_3(t)\geq \frac{N\mu}{s}e^{st}\Big\} ,\label{T3}\\
     	T_{(1)}&=T_1 \wedge T_2 \wedge T_3. \label{T(1)}
    		\end{align}
  	Lastly, we define the following events:
   		\begin{align}
		A_1&=\{T_{(1)}> t_1\}\label{A1}.\\
		A_2&=\bigg\{\sup_{t\in[0,t_1]}\Big|Z_{1m}^{(0,t_1]}(t\wedge T_{(1)})\Big|\leq\sqrt{\frac{48}{\epsilon}\cdot\frac{N\mu}{s^2}}\bigg\} \label{A2}.\\
		A_3&=\bigg\{\sup_{t\in[0,t_1]}\Big|Z_{2m}^{(0,t_1]}(t\wedge T_{(1)})\Big|\leq\sqrt{\frac{48}{\epsilon}\cdot\frac{N\mu}{s^2}}\bigg\}\nonumber.\\
		A_4&=\bigg\{\sup_{t\in[0,t_1]}\Big|Z_{1r}^{(0,t_1]}(t\wedge T_{(1)})\Big|\leq\sqrt{\frac{48}{\epsilon}\cdot\frac{N\mu r}{s^3}\ln\Big(\frac{s}{\mu}\Big)}\bigg\}\label{A4}.\\
		A_5&=\bigg\{\sup_{t\in[0,t_1]}\Big|Z_{2r}^{(0,t_1]}(t\wedge T_{(1)})\Big|\leq\sqrt{\frac{48}{\epsilon}\cdot\frac{N\mu r}{s^3}\ln\Big(\frac{s}{\mu}\Big)}\bigg\}\nonumber.\\
		A_6&=\bigg \{\sup_{t\in[t_{0,m}^+,t_1]}\Big|Z_{3m}^{(t_{0,m}^+,t_1]}(t\wedge T_{(1)})\Big|\leq\sqrt{\frac{48Ke^{C_{0,m}^+}}{\epsilon}\cdot\frac{1}{s^2}}\bigg\}\label{A6}.\\
		A_7&=\bigg\{\sup_{t\in[t_{0,r},t_1]}\Big|Z_{3r}^{(t_{0,r},t_1]}(t\wedge T_{(1)})\Big|\leq\sqrt{\frac{16K^2e^{-2C_{0,r}}}{\epsilon}\cdot\frac{\ln_+(Nr)}{s^2}}\bigg\}\label{A7}.\\
		A_8&=\Big\{X_{3m}^{(0,t_{0,m}]}(t_1\wedge T_{(1)})=0\Big\}\label{A8}.\\
		A_9&=\Big\{X_{3r}^{(0,t_{0,r}]}(t_1\wedge T_{(1)})=0\Big\}\label{A9}.\\
		A_{10}&=\bigg\{X_{3m}^{(t_{0,m},t_1]}(t_1\wedge T_{(1)})\leq \bigg(\frac{2Ke^{-2C_1+C_{0,m}}}{\epsilon}\bigg)\frac{N^2\mu^2}{s}\bigg\}.\label{A10} \\
		A_{11}&=\bigg\{X_{3r}^{(t_{0,r},t_1]}(t_1\wedge T_{(1)})\leq \bigg(\frac{K^2e^{-2C_1}(2(C_{0,r}-C_1)+1)}{2\epsilon}\bigg)\frac{(1\vee Nr\ln_+(Nr))}{s}\bigg\}.\label{A11} 
		\end{align}
Also, we define
	\begin{equation}\label{A(1)}
	A_{(1)}=
		\begin{cases}
		\displaystyle{\bigcap_{1\leq i \leq 11, i\neq 6}}A_i &\text{in the recombination dominating case}\\
    		\displaystyle{\bigcap_{1\leq i \leq 11, i\neq 7}}A_i &\text{in the mutation dominating case}.
    		\end{cases}
		\end{equation}
	We will show that these events occur with high probability. Here, we will prove some inequalities involving $G_1(t), G_2(t)$ and $G_3(t)$, which will be used quite often in this section.

  	\begin{lemma}\label{G<}
   	For sufficiently large $N$, and $t\in [0,t_1\wedge T_{(1)})$, the following statements hold:
   		\begin{enumerate}
   		\item $X_i(t) \leq \eta N$,  for $i=1,2,3$.
    		\item $G_1(t)\leq s$, $G_2(t)\leq s$, and $G_3(t) \leq 2s$.
    		\item $G_1(t) \geq s-4\eta s-r-\mu$, $G_2(t) \geq s-4\eta s-r-\mu$, and $G_3(t) \geq 2s-4\eta s -r$.
    		\item For $0<a<b$, we have $G_{1r}^{(a,b]}(t)\leq s+r1_{(a,b]}(t)$, $G_{2r}^{(a,b]}(t)\leq s+r1_{(a,b]}(t)$, and $G_{3r}^{(a,b]}(t) \leq 2s+r1_{(a,b]}(t)$. 
    		\item For $0<a<b$, we have $G_{1r}^{(a,b]}(t) \geq s-4\eta s-r-\mu$, $G_{2r}^{(a,b]}(t)\geq s-4\eta s-r-\mu$, and $G_{3r}^{(a,b]}(t)\geq 2s-4\eta s -r$.
   		\end{enumerate}
  	\end{lemma}

  	\begin{proof}
  	By the definition of $\eta$, $t_1$  and $T_{(1)}$ in (\ref{n}), (\ref{t1}) and (\ref{T(1)}), for every $t\in[0,t_1\wedge T_{(1)})$, and for $i=1,2,3$,
  		$$		
  		X_i(t)<\frac{KN\mu}{s}e^{st}\leq \frac{KN\mu}{s}e^{st_1}=Ke^{-C_1}N<\eta N.
  		$$
	For statement 2, since $0\leq\tilde X_1(t)+\tilde X_2(t)+\tilde X_3(t)\leq 1$ for all $t\geq 0$, and $s\ll 1$, it follows that for sufficiently large $N$, we have $0<1-2s \leq 1-s\tilde X_1(t)-s\tilde X_2(t)-2s\tilde X_3(t) \leq 1$ for all $t\geq 0$. Thus, by the definition of $G_1(t)$ in (\ref{G1}), for sufficiently large $N$, we have $G_1(t)\leq s$ for all $t\in [0,t_1\wedge T_{(1)})$. Also, by part 1, if $t\in[0,t_1\wedge T_{(1)})$, then $1-\tilde X_1(t)-\tilde X_2(t)-2\tilde X_3(t)\geq 1-4\eta$. Again, by using the definition of $G_1(t)$ in (\ref{G1}), we get the lower bound of $G_1(t)$ in statement 3. Both the upper and lower bounds for $G_3(t)$ can be shown by similar arguments. Lastly, we can prove statements 4 and 5 by using (\ref{G1r}), (\ref{G2r}) and (\ref{G3r}) along with statements 1, 2 and 3 of this lemma.
  	\end{proof}

\subsection{Upper bounds for expectations}
	In this section, we are going to prove some results on the upper  bounds for the expectations of $X_{im}^{(a,b]}(t\wedge T_{(1)})$ and $X_{ir}^{(a,b]}(t\wedge T_{(1)})$.

  	\begin{lemma} \label{EX1m}
   	For sufficiently large $N$, for $i=1,2$ and $t\in[0,t_1]$, we have
    		\begin{equation} \label{EX1m.1}
    		E\Big[e^{-s(t\wedge T_{(1)})}X_{im}(t\wedge T_{(1)})\Big]\leq \frac{N\mu}{s},
    		\end{equation}
	and
		$$
		E\Big[X_{im}(t\wedge T_{(1)})\Big]\leq \frac{N\mu}{s}e^{st}.
		$$
  	\end{lemma}

	\begin{proof}
  	The proof is similar to Lemma 5.1 in \cite{Sch}. We will show the proof for $i=1$, since the argument is similar for $i=2$. We will first show that for sufficiently large $N$, for $0\leq a <b\leq t_1$, and for $t \in [0,t_1]$, we have
    		\begin{equation}\label{lemEX1m}
    		E\Big[e^{-s(t\wedge T_{(1)})}X_{1m}^{(a,b]}(t\wedge T_{(1)})\Big]\leq e^{(4\eta s+r+\mu)(b-a)}\cdot N\mu \int_a^be^{-su}du.
    		\end{equation}
	If $t\in[0,a)$, this inequality is trivial, since by the definition of $X_{1m}^{(a,b]}(t)$, we have $X_{1m}^{(a,b]}(t)=0$. Assume that $t\in[a,t_1]$. By Proposition  \ref{ZMar} and (\ref{Zm}), we have $E\big[Z_{1m}^{(a,b]}(t\wedge T_{(1)})\big]=0$, and
    		\begin{equation} \label{1m.0}
    		E\Big[e^{-\int_a^{t\wedge T_{(1)}}G_1(v)dv}X_{1m}^{(a,b]}(t\wedge T_{(1)})\Big]=E\bigg[\int_a^{t\wedge T_{(1)}} M_1^{(a,b]}(u)e^{-\int_a^uG_1(v)dv}du\bigg].
    		\end{equation}
	Note that in the event that $T_{(1)}<a$, we interpret the integral from $a$ to $t\wedge T_{(1)}$ as 0. Also, from the definition of $X_{1m}^{(a,b]}(t)$, in the event $T_{(1)}<a$, we have $X_{1m}^{(a,b]}(t\wedge T_{(1)})=0$. Now, using the upper bound for $G_1(t)$ in Lemma \ref{G<}, we know that for sufficiently large $N$, for $0\leq a <b\leq t_1$, and $t\in [a,t_1]$,
    		\begin{align} 
     	E\Big[e^{-\int_a^{t\wedge T_{(1)}}G_1(v)dv}X_{1m}^{(a,b]}(t\wedge T_{(1)})\Big]
   		&=  E\Big[e^{-\int_a^{t\wedge T_{(1)}}G_1(v)dv}X_{1m}^{(a,b]}(t\wedge T_{(1)})1_{\{T_{(1)}\geq a\}}\Big] \nonumber \\	
     	&\geq E\Big[e^{-\int_a^{t\wedge T_{(1)}} sdv}X_{1m}^{(a,b]}(t\wedge T_{(1)})1_{\{T_{(1)}\geq a\}}\Big] \nonumber \\
     	&=e^{sa}E\Big[e^{-s(t\wedge T_{(1)})}X_{1m}^{(a,b]}(t\wedge T_{(1)})1_{\{T_{(1)}\geq a\}}\Big] \nonumber \\
     	&=e^{sa}E\Big[e^{-s(t\wedge T_{(1)})}X_{1m}^{(a,b]}(t\wedge T_{(1)})\Big]. \label{1m.1}
    		\end{align}
	Next, we use the lower bound for $G_1(t)$ in Lemma \ref{G<}. From (\ref{M1ab}), for sufficiently   large $N$, for $0\leq a <b\leq t_1$, and $t\in [a,t_1]$,
    		\begin{align}
     	E\bigg[\int_a^{t\wedge T_{(1)}} M_1^{(a,b]}(u)e^{-\int_a^uG_1(v)dv}du\bigg]
     	&=E\bigg[\int_a^{t\wedge T_{(1)}} \mu X_0(u)1_{(a,b]}(u)e^{-\int_a^uG_1(v)dv}du\bigg] \nonumber\\
     	&\leq \int_a^b \mu N e^{-(s -4\eta s-r-\mu)(u-a)}du \nonumber \\
     	&\leq e^{(4\eta s+r+\mu)(b-a)}\cdot N\mu e^{sa}\int_a^b e^{-su}du. \label{1m.2}
    		\end{align}
	From (\ref{1m.0}), (\ref{1m.1}) and (\ref{1m.2}), we have the inequality (\ref{lemEX1m}).
	
   	In the second part of the proof, for each $n \in \mathbb{N}$, let $t'_j=(b-a)j/n+a$, for $j=0,1,...,n$. It follows from
    		\begin{align*}
     	E\Big[e^{-s(t\wedge T_{(1)})}X_{1m}^{(a,b]}(t\wedge T_{(1)})\Big] 
     	&=E\Bigg[\sum_{j=0}^{n-1}e^{-s(t\wedge T_{(1)})}X_{1m}^{(t'_j,t'_{j+1}]}(t\wedge T_{(1)})\Bigg] \\
     	&\leq \sum _{j=0}^{m-1}e^{(4\eta s+r+\mu)(t'_{j+1}-t'_j)}\cdot N\mu\int_{t'_j}^{t'_{j+1}} e^{-su}du \\
     	&=e^{(4\eta s+r+\mu)(\frac{b-a}{n})}\cdot N\mu\int_a^b e^{-su}du \\
     	&\leq e^{(4\eta s+r+\mu)(\frac{b-a}{n})}\cdot \frac{N\mu}{s} e^{-sa}.
    		\end{align*}
	By letting $n\rightarrow \infty$, we have that for sufficiently large $N$, $0\leq a<b\leq t_1$, and $t\in[a,t_1]$,
 		\begin{equation} \label{1m.3}
    		E\Big[e^{-s(t\wedge T_{(1)})}X_{1m}^{(a,b]}(t\wedge T_{(1)})\Big]\leq \frac{N\mu}{s}e^{-sa}.
    		\end{equation}
The inequality (\ref{EX1m.1}) follows from the fact that $X_{1m}(t\wedge T_{(1)})=X^{(0,t]}_{1m}(t\wedge T_{(1)})$. From (\ref{1m.3}), it follows that
		$$
		E\Big[X_{1m}(t\wedge T_{(1)})\Big]=e^{st}E\Big[e^{-st}X_{1m}(t\wedge T_{(1)})\Big]\leq e^{st}E\Big[e^{-s(t\wedge T_{(1)})}X_{1m}^{(0,t]}(t\wedge T_{(1)})\Big]\leq \frac{N\mu}{s}e^{st},
		$$ 
	and the proof is completed.
  	\end{proof}

  	\begin{lemma} \label{EX1r}
   	For sufficiently large $N$, for $i=1,2$, and $t \in [0, t_1]$, we have
    		\begin{equation} \label{EX1r.1}
    		E\Big[e^{-s(t\wedge T_{(1)})}X_{ir}(t\wedge T_{(1)})\Big]\leq\bigg(\frac{N\mu r}{s}\bigg) t,
    		\end{equation}
   	and
    		\begin{equation} \label{EX1r.2}
    		E\Big[X_{ir}(t\wedge T_{(1)})\Big]\leq \bigg(\frac{N\mu r}{s}\bigg)e^{st}t.
    		\end{equation}
  	\end{lemma}

  	\begin{proof}
   	The proof is similar to the proof of Lemma \ref{EX1m}. We will show the proof for $i=1$, and the same argument can be used when $i=2$. In the first part of this proof, we will show that for   sufficiently large $N$, for $0\leq a <b\leq t_1$, and for $t \in [0,t_1]$, we have
    		\begin{equation}\label{lemEX1r}
    		E\Big[e^{-s(t\wedge T_{(1)})}X_{1r}^{(a,b]}(t\wedge T_{(1)})\Big]\leq e^{(4\eta s+2r+\mu)(b-a)}\cdot \frac{N\mu r}{s}\cdot(b-a).
    		\end{equation}
	If $t\in[0,a)$, this inequality is trivial, since by the definition of $X_{1r}^{(a,b]}(t)$, we have $X_{1r}^{(a,b]}(t)=0$. Assume that $t\in[a,t_1]$. By Proposition  \ref{ZMar} and (\ref{Zr}), we have $E\big[Z_{1r}^{(a,b]}(t\wedge T_{(1)})\big]=0$, and
    		\begin{equation} \label{1r.0}
    		E\Big[e^{-\int_a^{t\wedge T_{(1)}}G^{(a,b]}_{1r}(v)dv}X_{1r}^{(a,b]}(t\wedge T_{(1)})\Big]=E\bigg[\int_a^{t\wedge T_{(1)}} R_1^{(a,b]}(u)e^{-\int_a^uG^{(a,b]}_{1r}(v)dv}du\bigg].
    		\end{equation}
	Using the upper bound for $G^{(a,b]}_{1r}(t)$ in Lemma \ref{G<}, we know that for sufficiently large $N$, for $0\leq a <b\leq t_1$, and $t\in [a,t_1]$,
    		\begin{align} 
     	E\Big[e^{-\int_a^{t\wedge T_{(1)}}G^{(a,b]}_{1r}(v)dv}X_{1r}^{(a,b]}(t\wedge T_{(1)})\Big]
     	&\geq E\Big[e^{-\int_a^{t\wedge T_{(1)}} (s+r1_{(a,b]}(v))dv}X_{1r}^{(a,b]}(t\wedge T_{(1)})1_{\{T_{(1)}\geq a\}}\Big] \nonumber \\
     	&\geq e^{sa-r(b-a)}E\Big[e^{-s(t\wedge T_{(1)})}X_{1r}^{(a,b]}(t\wedge T_{(1)})1_{\{T_{(1)}\geq a\}}\Big] \nonumber \\
     	&=e^{sa-r(b-a)}E\Big[e^{-s(t\wedge T_{(1)})}X_{1r}^{(a,b]}(t\wedge T_{(1)})\Big]. \label{1r.1}
    		\end{align}
	Then, using the lower bound for $G_{1r}^{(a,b]}(t)$ in Lemma \ref{G<}, along with the upper bound for $R_1^{(a,b]}(t)$ in (\ref{R1<}) and the definition of $T_3$ in (\ref{T3}), we have that for sufficiently   large $N$, for $0\leq a <b\leq t_1$, and $t\in [a,t_1]$,
    		\begin{align}
     	E\bigg[\int_a^{t\wedge T_{(1)}} R_1^{(a,b]}(u)e^{-\int_a^uG_{1r}^{(a,b]}(v)dv}du\bigg]
     	&\leq E\bigg[\int_a^{t\wedge T_{(1)}} Nr\tilde X_0(u)\tilde X_3(u)1_{(a,b]}(u)e^{-\int_a^uG_{1r}^{(a,b]}(v)dv}du\bigg] \nonumber\\
     	&\leq \int_a^b Nr\cdot \frac{\mu}{s}e^{su}\cdot e^{-(s -4\eta s-r-\mu)(u-a)}du \nonumber \\
     	&\leq e^{(4\eta s+r+\mu)(b-a)}\cdot \frac{N\mu r}{s} e^{sa}(b-a). \label{1r.2}
    		\end{align}
	From (\ref{1r.0}), (\ref{1r.1}) and (\ref{1r.2}), we have the inequality (\ref{lemEX1r}). Lastly, by using (\ref{lemEX1r}) and following the argument in the second part of the proof of Lemma \ref{EX1m}, we can prove (\ref{EX1r.1}) and (\ref{EX1r.2}). 
  	\end{proof}     

 	\begin{lemma} \label{EX3m}
  	For sufficiently large $N$ and for $t\in[0,t_1]$, we have
   		\begin{equation} \label{EX3m.1}
    		E\Big[e^{-s(t\wedge T_{(1)})}X_{3m}(t\wedge T_{(1)})\Big]\leq \frac{2KN\mu^2}{s^2} e^{st},
    		\end{equation}
   		\begin{equation} \label{EX3m.2}
   		E\Big[X_{3m}^{(t_{0,m},t_1]}(t\wedge T_{(1)})\Big]\leq\frac{2Ke^{C_{0,m}}N^2\mu^4}{s^3} e^{2st},
   		\end{equation}
   	and
   		\begin{equation} \label{EX3m.3}
   		E\Big[X_{3m}^{(t_{0,m}^+,t_1]}(t\wedge T_{(1)})\Big]\leq\frac{2Ke^{-C_{0,m}^+}N^2\mu^4}{s^3} e^{2st},
   		\end{equation}
 	\end{lemma}

 	\begin{proof}
  	The argument in this proof is similar to that of Lemma \ref{EX1m}. We will first show that for sufficiently large $N$, for $0\leq a <b\leq t_1$, and for $t \in [0,t_1]$, we have
    		\begin{equation}\label{lemEX3m}
    		E\Big[e^{-s(t\wedge T_{(1)})}X_{3m}^{(a,b]}(t\wedge T_{(1)})\Big]\leq e^{(4\eta s+r)(b-a)}\cdot \frac{2KN\mu^2}{s}e^{st} \int_a^be^{-su}du.
    		\end{equation}
	If $t\in[0,a)$, this inequality is trivial, since by the definition of $X_{3m}^{(a,b]}(t)$, we have $X_{3m}^{(a,b]}(t)=0$. Let assume that $t\in[a,t_1]$. By Proposition  \ref{ZMar} and (\ref{Zm}), we have $E\big[Z_{3m}^{(a,b]}(t\wedge T_{(1)})\big]=0$, and
    		\begin{equation} \label{3m.0}
    		E\Big[e^{-\int_a^{t\wedge T_{(1)}}G_{3}(v)dv}X_{3m}^{(a,b]}(t\wedge T_{(1)})\Big]=E\bigg[\int_a^{t\wedge T_{(1)}} M_3^{(a,b]}(u)e^{-\int_a^uG_3(v)dv}du\bigg].
    		\end{equation}
	 Using the upper bound for $G_3(t)$ in Lemma \ref{G<}, we know that for sufficiently large $N$, for $0\leq a <b\leq t_1$, and $t\in [a,t_1]$,
    		\begin{align} 
     	E\Big[e^{-\int_a^{t\wedge T_{(1)}}G_3(v)dv}X_{3m}^{(a,b]}(t\wedge T_{(1)})\Big]
     	&\geq E\Big[e^{-\int_a^{t\wedge T_{(1)}} 2sdv}X_{3m}^{(a,b]}(t\wedge T_{(1)})1_{\{T_{(1)}\geq a\}}\Big] \nonumber \\
     	&\geq e^{-st+2sa}E\Big[e^{-s(t\wedge T_{(1)})}X_{3m}^{(a,b]}(t\wedge T_{(1)})1_{\{T_{(1)}\geq a\}}\Big] \nonumber \\
     	&=e^{-st+2sa}E\Big[e^{-s(t\wedge T_{(1)})}X_{3m}^{(a,b]}(t\wedge T_{(1)})\Big]. \label{3m.1}
    		\end{align}
	Now, we use the formula for $M_3^{(a,b]}(t)$ in (\ref{M3ab}), the lower bound for $G_3(t)$ in Lemma \ref{G<}, and the definition of $T_1$ and $T_2$ in (\ref{T1}) and (\ref{T2}). It follows that for sufficiently   large $N$, $0\leq a <b\leq t_1$, and $t\in [a,t_1]$,
    		\begin{align}
     	E\bigg[\int_a^{t\wedge T_{(1)}} M_3^{(a,b]}(u)e^{-\int_a^uG_3(v)dv}du\bigg]
     	&=E\bigg[\int_a^{t\wedge T_{(1)}} \mu (X_1(u)+X_2(u))1_{(a,b]}(u)e^{-\int_a^uG_3(v)dv}du\bigg] \nonumber\\
     	&\leq E\bigg[\int_a^{b\wedge T_{(1)}} \frac{2KN\mu^2}{s}e^{su}\cdot e^{-(2s-4\eta s-r)(u-a)}du\bigg] \nonumber\\
     	&\leq e^{(4\eta s+r)(b-a)}\cdot \frac{2KN\mu^2}{s} e^{2sa}\int_a^b e^{-su}du. \label{3m.2}
    		\end{align}
	From (\ref{3m.0}), (\ref{3m.1}) and (\ref{3m.2}), we have the inequality (\ref{lemEX3m}). By following the argument in the second part of the proof of Lemma \ref{EX1m}, it follows that for sufficiently large $N$, when $0\leq a< t_1$ and $t\in[a,t_1]$,
   		\begin{equation} \label{3m.3}
    		E\Big[e^{-s(t\wedge T_{(1)})}X_{3m}^{(a,b]}(t\wedge T_{(1)})\Big]\leq \frac{2KN\mu^2}{s^2}e^{s(t-a)},
    		\end{equation}
    	and
    		\begin{equation} \label{3m.4}
    		E\Big[X_{3m}^{(a,b]}(t\wedge T_{(1)})\Big]\leq \frac{2KN\mu^2}{s^2}e^{s(2t-a)}.
    		\end{equation}
	The inequality (\ref{EX3m.1}) follows from (\ref{3m.3}) and the fact that $X_{3m}(t\wedge T_{(1)})=X^{(0,t]}_{3m}(t\wedge T_{(1)})$, and the inequalities (\ref{EX3m.2}) and (\ref{EX3m.3}) follow from (\ref{3m.4}) and the definitions of $t_{0,m}$ and $t_{0,m}^+$ in (\ref{t0m}) and (\ref{t0m+}).
  	\end{proof}

  \begin{lemma} \label{EX3r}
   For sufficiently large $N$ and $0\leq a < t_1$, if $t \in [0, t_1]$, we have
    		\begin{equation} \label{EX3r.1}
    		E\Big[e^{-s(t\wedge T_{(1)})}X_{3r}(t\wedge T_{(1)})\Big]\leq\bigg(\frac{K^2N\mu^2 r}{s^2}\bigg)e^{st} t,
    		\end{equation}
   	and if $t\in[a,t_1]$, 
    		\begin{equation} \label{EX3r.2}
    		E\Big[X_{3r}^{(a,t_1]}(t\wedge T_{(1)})\Big]\leq \bigg(\frac{K^2N\mu^2 r}{s^2}\bigg)e^{2st}(t-a).
    		\end{equation}
  	\end{lemma}

  	\begin{proof}
   	The proof is similar to the proof of Lemma \ref{EX1m}. We will first show that for   sufficiently large $N$, for $0\leq a < b \leq t_1$ and $t\in[0,t_1]$, we have
    		\begin{equation}\label{lemEX3r}
    		E\Big[e^{-s(t\wedge T_{(1)})}X_{3r}^{(a,b]}(t\wedge T_{(1)})\Big]\leq e^{(4\eta s+2r)(b-a)}\cdot \frac{K^2N\mu^2 r}{s^2}e^{st}\cdot\int_{t\wedge a}^{t\wedge b}1du.
    		\end{equation}
	If $t\in[0,a)$, this inequality is trivial, since by the definition of $X_{3r}^{(a,b]}(t)$, we have $X_{3r}^{(a,b]}(t)=0$. Assume that $t\in[a,t_1]$. By Proposition  \ref{ZMar} and (\ref{Zr}), we have $E\big[Z_{3r}^{(a,b]}(t\wedge T_{(1)})\big]=0$, and
    		\begin{equation} \label{3r.0}
    		E\Big[e^{-\int_a^{t\wedge T_{(1)}}G^{(a,b]}_{3r}(v)dv}X_{3r}^{(a,b]}(t\wedge T_{(1)})\Big]=E\bigg[\int_a^{t\wedge T_{(1)}} R_3^{(a,b]}(u)e^{-\int_a^uG^{(a,b]}_{3r}(v)dv}du\bigg].
    		\end{equation}
	Using the upper bound for $G^{(a,b]}_{3r}(t)$ in Lemma \ref{G<}, we know that for sufficiently large $N$, for $0\leq a <b\leq t_1$, and $t\in [a,t_1]$,
    		\begin{align} 
     	E\Big[e^{-\int_a^{t\wedge T_{(1)}}G^{(a,b]}_{3r}(v)dv}X_{3r}^{(a,b]}(t\wedge T_{(1)})\Big]
     	&\geq E\Big[e^{-\int_a^{t\wedge T_{(1)}} (2s+r1_{(a,b]}(v))dv}X_{3r}^{(a,b]}(t\wedge T_{(1)})1_{\{T_{(1)}\geq a\}}\Big] \nonumber \\
     	&\geq e^{2sa-st-r(b-a)}E\Big[e^{-s(t\wedge T_{(1)})}X_{3r}^{(a,b]}(t\wedge T_{(1)})1_{\{T_{(1)}\geq a\}}\Big] \nonumber \\
     	&=e^{2sa-st-r(b-a)}E\Big[e^{-s(t\wedge T_{(1)})}X_{3r}^{(a,b]}(t\wedge T_{(1)})\Big]. \label{3r.1}
    		\end{align}
	Then, we use the lower bound for $G_{3r}^{(a,b]}(t)$ in Lemma \ref{G<}, along with the upper bound for $R_3^{(a,b]}(t)$ in (\ref{R1<}) and the definitions of $T_1$ and $T_2$ in (\ref{T1}) and (\ref{T2}), we have that for sufficiently large $N$, for $0\leq a <b\leq t_1$, and $t\in [a,t_1]$,
    		\begin{align}
     	E\bigg[\int_a^{t\wedge T_{(1)}} R_3^{(a,b]}(u)e^{-\int_a^uG_{3r}^{(a,b]}(v)dv}du\bigg]
     	&\leq E\bigg[\int_a^{t\wedge T_{(1)}} Nr\tilde X_1(u)\tilde X_2(u)1_{(a,b]}(u)e^{-\int_a^uG_{3r}^{(a,b]}(v)dv}du\bigg] \nonumber\\
     	&\leq \int_a^{t\wedge b} \frac{K^2N\mu^2r}{s^2}e^{2su}\cdot e^{-(2s -4\eta s-r)(u-a)}du \nonumber \\
     	&\leq e^{(4\eta s+r)(b-a)}\cdot \frac{K^2N\mu^2 r}{s^2} \cdot e^{2sa}\cdot \int_a^{t\wedge b}1du. \label{3r.2}
    		\end{align}
	From (\ref{3r.0}), (\ref{3r.1}) and (\ref{3r.2}), we have the inequality (\ref{lemEX3r}). By similar argument to the second part of the proof of Lemma \ref{EX1m}, we can show that for sufficiently large $N$, for $0\leq a< b\leq t_1$, and $t\in[a,t_1]$,
		\begin{equation}	\label{3r.3}	
		E\Big[e^{-s(t\wedge T_{(1)})}X_{3r}^{(a,b]}(t\wedge T_{(1)})\Big]\leq \frac{K^2N\mu^2 r}{s^2}e^{st}\cdot(t\wedge b-a),
		\end{equation}
	and
		\begin{equation}	\label{3r.4}	
		E\Big[X_{3r}^{(a,b]}(t\wedge T_{(1)})\Big]\leq \frac{K^2N\mu^2 r}{s^2}e^{2st}\cdot(t\wedge b-a).
		\end{equation}
	The inequality (\ref{EX3r.1}) follows from the inequality (\ref{3r.3}) and the fact that $X_{3r}^{(0,t]}(t\wedge T_{(1)})=X_{3r}(t\wedge T_{(1)})$. The inequality (\ref{EX3r.2}) is a special case of the inequality (\ref{3r.4}) when $b=t_1$.
  	\end{proof}
  	   
  	Using these upper bounds on expectations, we can prove that when $N$ is sufficiently large, the event $T_{(1)}>t_1$ occurs with probability close to 1, and the proof is shown below.
  	
	\begin{lemma}\label{T<t1}
   	For sufficiently large $N$, we have $P(A_1^c)\leq 2\epsilon$.
	\end{lemma}

  	\begin{proof}
   	Recall the definition of $A_1$ in (\ref{A1}). First, note that 
    		\begin{equation} \label{T<t1.1}
    		P(A_1^c)=P(T_{(1)}\leq t_1)= P(t_1\wedge T_{(1)} = T_{(1)}) \leq \sum_{i=1}^3 P(t_1\wedge T_{(1)} =T_i).
    		\end{equation}
	Now, consider the term $P(t_1\wedge T_{(1)}=T_i)$, for $i=1,2$. Using Markov's inequality, Lemmas \ref{EX1m} and \ref{EX1r}, the definition of $t_1$ in (\ref{t1}), and (\ref{rslogsu}), for sufficiently large $N$, 
    \begin{align}
  		P(t_1\wedge T_{(1)}=T_i)
     	&\leq P\Big(X_i(t_1\wedge T_{(1)})\geq \frac{KN\mu}{s}e^{s(t_1\wedge T_{(1)})}\Big)\nonumber\\
     	&\leq P\Big(X_{im}(t_1\wedge T_{(1)})\geq \frac{KN\mu}{2s}e^{s(t_1\wedge T_{(1)})}\Big)+P\Big(X_{ir}(t_1\wedge T_{(1)})	\geq \frac{KN\mu}{2s}e^{s(t_1\wedge T_{(1)})}\Big)\nonumber\\
     	&\leq P\Big(e^{-s(t_1\wedge T_{(1)})}X_{im}(t_1\wedge T_{(1)})\geq \frac{KN\mu}{2s}\Big)+P\Big(e^{-s(t_1\wedge T_{(1)})}X_{ir}(t_1\wedge T_{(1)})\geq \frac{KN\mu}{2s}\Big)\nonumber\\
     	&\leq \frac{E[e^{-s(t_1\wedge T_{(1)})}X_{im}(t_1\wedge T_{(1)})]}{KN\mu/2s}+\frac{E[e^{-s(t_1\wedge T_{(1)})}X_{ir}(t_1\wedge T_{(1)})]}{KN\mu/2s}\nonumber\\
     	&\leq \frac{2}{K}+\frac{2rt_1}{K}\nonumber\\
     	&\leq \frac{2}{K}+\frac{2}{K}\cdot \frac{r}{s}\ln\Big(\frac{s}{\mu}\Big)\nonumber\\
		&\leq 3K^{-1}.\label{T<t1.2}
    	\end{align}
   	Next, consider the term $P(t_1\wedge T_{(1)}=T_3)$. By Markov's inequality, Lemma \ref{EX3m}, Lemma \ref{EX3r}, and using  (\ref{rslogsu}), for sufficiently large $N$, we have
    		\begin{align}
     	P(t_1\wedge T_{(1)}=T_3)
     	&\leq P\Big(X_3(t_1\wedge T_{(1)})\geq \frac{N\mu}{s} e^{s(t_1\wedge T_{(1)})}\Big)\nonumber\\
     	&\leq P\Big(X_{3m}(t_1\wedge T_{(1)})\geq \frac{N\mu}{2s}e^{s(t_1\wedge T_{(1)})}\Big)+P\Big(X_{3r}(t_1\wedge T_{(1)})\geq \frac{N\mu}{2s}e^{s(t_1\wedge T_{(1)})}\Big)\nonumber\\
     	&\leq P\Big(e^{-s(t_1\wedge T_{(1)})}X_{3m}(t_1\wedge T_{(1)})\geq \frac{N\mu}{2s}\Big)+P\Big(e^{-s(t_1\wedge T_{(1)})}X_{3r}(t_1\wedge T_{(1)})\geq \frac{N\mu}{2s}\Big)\nonumber\\
     	&\leq \frac{E[e^{-s(t_1\wedge T_{(1)})}X_{3m}(t_1\wedge T_{(1)})]}{N\mu/2s}+\frac{E[e^{-s(t_1\wedge T_{(1)})}X_{3r}(t_1\wedge T_{(1)})]}{N\mu/2s}\nonumber\\
     	&\leq \frac{4K\mu e^{st_1}}{s}+\frac{2K^2\mu r e^{st_1}t_1}{s}\nonumber\\
     	&\leq 4Ke^{-C_1}+2K^2e^{-C_1}\cdot \frac{r}{s}\ln\Big(\frac{s}{\mu}\Big)\nonumber\\
     	& \leq 5Ke^{-C_1}. \label{T<t1.3}
    		\end{align}
   	Thus, from (\ref{T<t1.1}), (\ref{T<t1.2}), (\ref{T<t1.3}) and the way we choose $K$ and $C_1$ in (\ref{K}) and (\ref{C1}), for sufficiently large $N$, we have $P(T\leq t_1)\leq 6K^{-1}+5Ke^{-C_1}\leq 2\epsilon$. 
  	\end{proof}

\subsection{The variance bounds}
	By using the upper bounds for expectations, the variance formulas in Proposition \ref{ZMar}, and the $L^2$-maximal inequality, we can show that the probability that each of the events $A_2, A_3, A_4, A_5, A_6, A_7$ occurs is at least $1-\epsilon$.
  	\begin{lemma} \label{|Z|}
   	The following statements hold:
    		\begin{enumerate}
     	\item For sufficiently large $N$, and for $i=2,3,4,5,6$, we have $P(A_i^c)\leq\epsilon$.
     	\item In the recombination dominating case, for sufficiently large $N$, we have $P(A_7^c)\leq\epsilon$.
    		\end{enumerate}
   	\end{lemma}

   	\begin{proof} Recall the definitions of the events $A_2, A_3, A_4, A_5, A_6, A_7$ in (\ref{A2}) - (\ref{A7}). We will first prove that $P(A_2^c)\leq\epsilon$, when $N$ is sufficiently large. From (\ref{B1mab}), (\ref{D1mab}) and the facts that $\mu \ll s, r\ll s$, and $s\ll 1$,  for sufficiently large $N$ and for $t\geq 0$,
    		$$
    		B_{1m}^{(0,t_1]}(t)
    		\leq \tilde X_0(t)+\tilde X_1(t)+\tilde X_2(t)+\tilde X_3(t)
    		=1,
    		$$
   	and
    		$$
    		D_{1m}^{(0,t_1]}(t)
    		\leq (1-s)+\mu
    		\leq 1.
    		$$
   	From Proposition \ref{ZMar}, Lemma \ref{G<}, and Lemma \ref{EX1m}, for sufficiently large $N$,
    		\begin{align}
     	\textup{Var}\Big(Z_{1m}^{(0,t_1]}(t_1\wedge T_{(1)})\Big) 
     	&=E\bigg[\int_0^{t_1\wedge T_{(1)}} e^{-2\int_0^u G_1(v)dv}\Big(\mu X_0(u)+\big(B_{1m}^{(0,t_1]}(u)+D_{1m}^{(0,t_1]}(u)\big)X_{1m}^{(0,t_1]}(u)\Big)du\bigg] \nonumber\\
     	&\leq E\bigg[\int_0^{t_1}e^{-2(s-4\eta s-r-\mu )u}(N\mu +2X_{1m}(u\wedge T_{(1)}))du\bigg] \nonumber\\
     	&\leq e^{2(r+\mu) t_1}E\bigg[\int_0^{t_1}e^{-2s(1-4\eta)u}(N\mu+2X_{1m}(u\wedge T_{(1)}))du\bigg] \nonumber\\
     	&=e^{2(r+\mu) t_1}\int_0^{t_{1}}e^{-2s(1-4\eta)u}(N\mu+2 E[X_{1m}(u\wedge T_{(1)})])du \nonumber\\ 
     	& \leq e^{2(r+\mu) t_1}\int_0^{t_1}e^{-2s(1-4\eta)u}\bigg(N\mu +\frac{2 N\mu}{s}e^{su}\bigg)du \nonumber\\
     	& = e^{2(r+\mu) t_1}\cdot\frac{N\mu}{s}\int_0^{t_1}\big(e^{-2s(1-4\eta)u}s+2 e^{-s(1-8\eta )u}\big)du \nonumber\\
     	& \leq e^{2(r+\mu) t_1}\cdot\frac{N\mu}{s}\cdot \bigg(\frac{1}{2(1-4\eta)}+\frac{2}{s(1-8\eta)}\bigg). \label{Z1m.1}
    		\end{align}
   	From the definition of $t_1$ in (\ref{t1}) along with (\ref{rslogsu}),  and the facts that $\mu \ll s$ and $r\ll s$, we have that 
   		\begin{equation}\label{(r+u)t1}
   		(r+\mu)t_1
   		=\frac{r}{s}\ln\Big(\frac{s}{\mu}\Big)+\frac{\mu}{s}\ln\Big(\frac{s}{\mu}\Big)-\frac{C_1(r+\mu)}{s} 
   		\ll 1.
   		\end{equation}
   	By the way we choose $\epsilon, K$ and $\eta$ in (\ref{eps}), (\ref{K}) and (\ref{n}), 
   		\begin{equation} \label{n<1/16}		
   		\eta=2Ke^{-C_1}<2\epsilon /5 <\epsilon \leq 1/16.
   		\end{equation} 
   	By (\ref{Z1m.1}), (\ref{(r+u)t1}) and the fact that $s\ll 1$, for sufficiently large $N$, 
    		$$
    		\textup{Var}\Big(Z_{1m}^{(0,t_1]}(t_1\wedge T_{(1)})\Big) 
    		\leq 2\Big(\frac{N\mu}{s}\Big)\Big(\frac{3}{s(1-8\eta)}\Big)
    		=\Big(\frac{6}{1-8\eta}\Big)\Big(\frac{N\mu}{s^2}\Big)
    		\leq\frac{12N\mu}{s^2}.
    		$$
   	By the $L^2$-maximal inequality, for sufficiently large $N$,
    		$$
    		P\bigg(\sup_{t\in[0,t_1]}\Big|Z_{1m}^{(0,t_1]}(t\wedge T_{(1)})\Big|\geq\sqrt{\frac{48}{\epsilon}\cdot\frac{N\mu}{s^2}}\bigg)
    		\leq \frac{4\textup{Var}(Z_{1m}^{(0,t_1]}(t_1\wedge T_0))}{\frac{48}{\epsilon}\cdot\frac{N\mu}{s^2}}
    		\leq\epsilon.
    		$$
  	Hence, we have shown that $P(A_2^c)\leq \epsilon$. The proof for $P(A_3^c)\leq \epsilon$ is in fact the same as that for $P(A_2^c)\leq \epsilon$.
   	
	Now, we will prove that $P(A_4^c)\leq \epsilon$. From (\ref{B1rab}), (\ref{D1rab}) and the facts that $\mu \ll s, r \ll s$, and $s\ll 1$, for sufficiently large $N$, for all $t\geq 0$, we have $B_{1r}^{(0,t_1]}(t)\leq 1$ and $D_{1r}^{(0,t_1]}(t)\leq  1$. From Proposition \ref{ZMar}, Lemma \ref{G<}, and inequality (\ref{R1<}), for sufficiently large $N$,
      	\begin{align*}
       	\textup{Var}\Big(Z_{1r}^{(0,t_1]}(t_1\wedge T_{(1)})\Big) 
       	&=E\bigg[\int_0^{t_1\wedge T_{(1)}} e^{-2\int_0^u G_{1r}^{(0,t_1]}(v)dv}\Big(R_1^{(0,t_1]}(u)+\big(B_{1r}^{(0,t_1]}(u)+D_{1r}^{(0,t_1]}(u)\big)X_{1r}^{(0,t_1]}(u)\Big)du\bigg] \\
       	&\leq E\bigg[\int_0^{t_1}e^{-2(s-4\eta s-r-\mu )u}(Nr \tilde X_3(u\wedge T_{(1)}) +2X_{1r}(u\wedge T_{(1)}))du\bigg] \\
       	&\leq e^{2(r+\mu) t_1}\int_0^{t_1}e^{-2s(1-4\eta)u}\Big(Nr\cdot \frac{\mu}{s}e^{su}+2E[X_{1r}(u\wedge T_{(1)})]\Big)du.
       	\end{align*}
   	From Lemma \ref{EX1r} and the definition of $t_1$ in (\ref{t1}), for sufficiently large $N$,
       	\begin{align*}
       	\textup{Var}\Big(Z_{1r}^{(0,t_1]}(t_1\wedge T_{(1)})\Big)
       	& \leq e^{2(r+\mu) t_1}\int_0^{t_1}e^{-2s(1-4\eta)u}\bigg(\frac{N\mu r}{s}e^{su}+\frac{2 N\mu r}{s^2}\ln\Big(\frac{s}{\mu}\Big)e^{su}\bigg)du\\
       	& \leq e^{2(r+\mu) t_1}\frac{N\mu r}{s}\Big(1+\frac{2}{s}\ln\Big(\frac{s}{\mu}\Big)\Big)\int_0^{t_1}e^{-s(1-8\eta)u}du\\
       	& \leq e^{2(r+\mu) t_1}\frac{N\mu r}{s^2(1-8\eta)}\Big(1+\frac{2}{s}\ln\Big(\frac{s}{\mu}\Big)\Big).
     	\end{align*}
    	Therefore, from (\ref{(r+u)t1}), (\ref{n<1/16}), and the fact that $\mu \ll s \ll 1$, for sufficiently large $N$, 
     	$$
     	\textup{Var}\Big(Z_{1r}^{(0,t_1]}(t_1\wedge T_{(1)})\Big)    
     	\leq 2\cdot\frac{N\mu r}{s^2(1-8\eta)}\cdot\frac{3}{s}\ln\Big(\frac{s}{\mu}\Big)
     	\leq \frac{12N\mu r}{s^3}\ln\Big(\frac{s}{\mu}\Big).
     	$$
    	By the $L^2$-maximal inequality, for sufficiently large $N$,
     	$$
     	P\bigg(\sup_{t\in[0,t_1]}\big|Z_{1r}^{(0,t_1]}(t\wedge T_{(1)})\big|\geq\sqrt{\frac{48}{\epsilon}\cdot\frac{N\mu r}{s^3}\ln\Big(\frac{s}{\mu}\Big)}\bigg)
     	\leq\frac{4\textup{Var}(Z_{1r}^{(0,t_1]}(t_1\wedge T_{(1)}))}{\frac{48}{\epsilon}\cdot\frac{N\mu r}{s^3}\ln\Big(\frac{s}{\mu}\Big)}
     	\leq\epsilon.
     	$$ 
    We have proved that $P(A_4^c)\leq \epsilon$. The proof for $P(A_5^c)\leq \epsilon$ is the same as the proof for $P(A_4^c)\leq \epsilon$. 	
   	
   	Next, We will give a proof that $P(A_6^c)\leq \epsilon$. From (\ref{B3mab}), (\ref{D3mab}) and the facts that $\mu \ll s, r \ll  s$, and $s\ll 1$. for sufficiently large $N$, for all $t\geq t_{0,m}^+$, we have $B_{3m}^{(t_{0,m}^+,t_1]}(t) \leq 1$ and $ D_{3m}^{(t_{0,m}^+,t_1]}(t)\leq 1$. From Proposition \ref{ZMar}, Lemma \ref{G<}, and the definitions of $T_1$ and $T_2$ in (\ref{T1}) and (\ref{T2}), for sufficiently large $N$,
    		\begin{align*}
     	&\textup{Var}\Big(Z_{3m}^{(t_{0,m}^+,t_1]}(t_1\wedge T_{(1)})\Big)\nonumber\\ 
     	&\hspace{0.5cm}=E\bigg[\int_{t_{0,m}^+}^{t_1\wedge T_{(1)}}e^{-2\int_{t_{0,m}^+}^u G_3(v)dv}\Big(\mu (X_1(u)+X_2(u))+\big(B_{3m}^{(t_{0,m}^+,t_1]}(u)+D_{3m}^{(t_{0,m}^+,t_1]}(u)\big)X_{3m}^{(t_{0,m}^+,t_1]}(u)\Big)du\bigg] \nonumber\\
     	&\hspace{0.5cm}\leq E\bigg[\int_{t_{0,m}^+}^{t_1\wedge T_{(1)}}e^{-2\int_{t_{0,m}^+}^u (2s-4\eta s)dv}\Big(\mu\cdot \frac{2KN\mu}{s}e^{su} +2X_{3m}^{(t_{0,m}^+,t_1]}(u\wedge T_{(1)})\Big)du\bigg] \nonumber\\
     	&\hspace{0.5cm}\leq \int_{t_{0,m}^+}^{t_1}e^{-2(2s-4\eta s-r)(u-t_{0,m}^+)}\Big(\mu \cdot \frac{2KN\mu}{s}e^{su} +2E\big[X_{3m}^{(t_{0,m}^+,t_1]}(u\wedge T_{(1)})\big]\Big)du.
     	\end{align*}
     By Lemma \ref{EX3m}, and the definition of $t_{0,m}^+$ in (\ref{t0m+}), for sufficiently large $N$,
     	\begin{align*}
     	&\textup{Var}\Big(Z_{3m}^{(t_{0,m}^+,t_1]}(t_1\wedge T_{(1)})\Big)\\
     	&\hspace{0.5cm}\leq e^{2rt_1}\cdot e^{2(2s-4\eta s)t_{0,m}^+}\int_{t_{0,m}^+}^{t_1}e^{-2(2s-4\eta s)u}\Big(\frac{2KN\mu^2}{s}e^{su} +\frac{4Ke^{-C_{0,m}^+}N^2\mu^4}{s^3}e^{2su}\Big)du \nonumber\\
     	&\hspace{0.5cm}\leq e^{2rt_1}\cdot e^{2(2s-4\eta s)t_{0,m}^+}\cdot\frac{2KN\mu^2}{s}\int_{t_{0,m}^+}^{t_1}\Big(e^{-(3s-8\eta s)u}+\frac{2e^{-C_{0,m}^+}N\mu^2}{s^2}e^{-(2s-8\eta s)u}\Big)du \nonumber\\
    	 	&\hspace{0.5cm}\leq e^{2rt_1}\cdot e^{2(2s-4\eta s)t_{0,m}^+}\cdot \frac{2KN\mu^2}{s}\bigg(\frac{e^{-(3s-8\eta s)t_{0,m}^+}}{s(3-8\eta)}+\frac{2e^{-C_{0,m}^+}N\mu^2}{s^2}\cdot \frac{e^{-(2s-8\eta s)t_{0,m}^+}}{s(2-8\eta)}\bigg)\nonumber \\
     	&\hspace{0.5cm}=e^{2rt_1}\cdot \frac{2K}{s^2}\bigg(\frac{e^{C_{0,m}^+}s}{3-8\eta}+\frac{2e^{C_{0,m}^+}}{2-8\eta}\bigg).
    		\end{align*}
     From (\ref{(r+u)t1}) and (\ref{n<1/16}) along with the fact that $s\ll 1$, for sufficiently large $N$,  
      	$$
      	\textup{Var}\Big(Z_{3m}^{(t_{0,m}^+,t_1]}(t_1\wedge T_{(1)})\Big) 
      	\leq 2 \cdot\frac{2K}{s^2}\cdot \frac{3e^{C_{0,m}^+}}{2-8\eta}
      	< \frac{12Ke^{C_{0,m}^+}}{s^2}.
      	$$
     By the $L^2$-maximal inequality, we have that for sufficiently large $N$,
      	$$
      	P\bigg(\sup_{t\in[t_{0,m}^+,t_1]}\big|Z_{3m}^{(t_{0,m}^+,t_1]}(t\wedge T_{(1)})\big|\geq\sqrt{\frac{48Ke^{C_{0,m}^+}}{\epsilon}\cdot\frac{1}{s^2}}\bigg)
      	\leq\frac{4\textup{Var}(Z_{3m}^{(t_{0,m}^+,t_1]}(t_1\wedge T_{(1)}))}{\frac{48Ke^{C_{0,m}^+}}{\epsilon}\cdot\frac{1}{s^2}}
      	\leq\epsilon.
      	$$
    	
    	Lastly, we will prove part 2. From (\ref{B3rab}), (\ref{D3rab}) and the fact that $\mu \ll s, r\ll s $, and $s\ll 1$, for sufficiently large $N$, for all $t\geq 0$, we have $B_{1r}^{(0,t_1]}(t)\leq 1$ and $D_{1r}^{(0,t_1]}(t)\leq 1$. From Proposition \ref{ZMar}, Lemma \ref{G<}, inequality (\ref{R3<}), and the definition of $T_1$ and $T_2$ in (\ref{T1}) and (\ref{T2}), for sufficiently large $N$,
     	\begin{align*}
      	&\textup{Var}\Big(Z_{3r}^{(t_{0,r},t_1]}(t_1\wedge T_{(1)})\Big)\\ 
      	&\hspace{0.5cm}=E\bigg[\int_{t_{0,r}}^{t_1\wedge T_{(1)}} e^{-2\int_{t_{0,r}}^u G_{3r}^{(t_{0,r},t_1]}(v)dv}  \Big(R_3^{(t_{0,r},t_1]}(u)+\big(B_{3r}^{(t_{0,r},t_1]}(u)+D_{3r}^{(t_{0,r},t_1]}(u)\big)X_{3r}^{(t_{0,r},t_1]}(u)\Big)du\bigg] \\
      	&\hspace{0.5cm}\leq E\bigg[ \int_{t_{0,r}}^{t_1\wedge T_{(1)}}e^{-2(2s-4\eta s-r)(u-t_{0,r})}\Big(Nr\tilde X_1(u) \tilde X_2(u) +2X_{3r}^{(t_{0,r},t_1]}(u\wedge T_{(1)})\Big)du\bigg] \\
      	&\hspace{0.5cm}\leq \int_{t_{0,r}}^{t_1}e^{-2(2s-4\eta s-r)(u-t_{0,r})}	\Big(\frac{K^2N\mu^2 r}{s^2}e^{2su}+2E[X_{3r}^{(t_{0,r},t_1]}(u\wedge T_{(1)})]\Big)du.
      	\end{align*}
   	By Lemma \ref{EX3r} and the definitions of $t_{0,r}$ and $t_1$ in (\ref{t0r}) and (\ref{t1}), for sufficiently large $N$, 
      	\begin{align*}
      	&\textup{Var}\Big(Z_{3r}^{(t_{0,r},t_1]}(t_1\wedge T_{(1)})\Big)\\ 
      	&\hspace{0.5cm}\leq e^{2rt_1}\cdot e^{2(2s-4\eta s)t_{0,r}}\int_{t_{0,r}}^{t_1}e^{-2(2s-4\eta s)u}\Big(\frac{K^2N\mu^2r}{s^2}e^{2su} +\frac{2K^2N\mu^2r}{s^2}e^{2su}(u-t_{0,r})\Big)du \\
      	&\hspace{0.5cm}\leq e^{2rt_1}\cdot e^{2(2s-4\eta s)t_{0,r}}\cdot\frac{K^2N\mu^2r}{s^2}\Big(1+2(t_1-t_{0,r})\Big)\int_{t_{0,r}}^{t_1}e^{-(2s-8\eta s)u}du \\
      	&\hspace{0.5cm}\leq e^{2rt_1}\cdot e^{2(2s-4\eta s)t_{0,r}}\cdot\frac{K^2N\mu^2r}{s^2}\Big(1+2(t_1-t_{0,r})\Big)\cdot \frac{e^{-(2s-8\eta s)t_{0,r}}}{s(2-8\eta)} \\
      	&\hspace{0.5cm}= e^{2rt_1}\cdot \frac{K^2e^{-2C_{0,r}}}{s(2-8\eta)}\Big(1+\frac{2}{s}\ln(Nr)+\frac{2(C_{0,r}-C_1)}{s}\Big)
     	\end{align*}
    	Because in the recombination dominating case, $1\ll Nr$, by using the fact that $s\ll 1$ along with (\ref{(r+u)t1}) and (\ref{n<1/16}), we have that for sufficiently large $N$, 
     	$$
     	\textup{Var}\Big(Z_{3r}^{(t_{0,r},t_1]}(t_1\wedge T_{(1)})\Big) 
     	\leq 2 \cdot \frac{K^2e^{-2C_{0,r}}}{s(2-8\eta)}\cdot \frac{3}{s}\ln(Nr)
     	\leq \frac{4K^2e^{-2C_{0,r}}\ln(Nr)}{s^2}.
     	$$
    	By the $L^2$-maximal inequality, for sufficiently large $N$, 
     	$$
     	P\bigg(\sup_{t\in[t_{0,r},t_1]}|Z_{3r}^{(t_{0,r},t_1]}(t\wedge T_{(1)})|\geq\sqrt{\frac{16K^2e^{-2C_{0,r}}}{\epsilon}\cdot\frac{\ln(Nr)}{s^2}}\bigg)
     	\leq\frac{4\textup{Var}(Z_{3r}^{(t_{0,r},t_1]}(t_1\wedge T_{(1)}))}{\frac{16K^2e^{-2C_{0,r}}}{\epsilon}\cdot\frac{\ln(Nr)}{s^2}}
     	\leq\epsilon,
     	$$
     which proves part 2.
   	\end{proof}

\subsection{Results on type 3 individuals}
  	In this section, we will show that the events $A_8$ and $A_9$ as defined in (\ref{A8}) and (\ref{A9}) occur with high probability. That is with probability close to 1, there are no type 3m (or 3r) individuals at time $t_1$ that are descended from type 3m (or 3r) ancestors that appear before time $t_{0,m}$ (or $t_{0,r}$). The proof consists of two main ideas.
  	\begin{enumerate}
  	\item With probability close to 1, the number of type 3m (or 3r) ancestors that appear before time $t_{0,m}$ (or $t_{0,r}$) is small.
  	\item With probability close to 1, each of these early ancestors will not have alive descendant by time $t_1$. 
  	\end{enumerate}
  	At the end of this subsection, we will show that the events $A_{10}$ and $A_{11}$ as defined in (\ref{A10}) and (\ref{A11}) also occur with high probability.
	
   	\begin{lemma} \label{m&rho}
    	Define $m(t)$ and $\rho(t)$ to be the number of type 3m ancestors and 3r ancestors respectively that appear in the time interval $(0,t]$. For sufficiently large $N$, the following statements hold:  
     	\begin{enumerate}
      	\item $P\bigg(m(t_{0,m}\wedge T_{(1)})\geq \frac{e^{-C_{0,m}/2}}{s}\bigg)\leq \epsilon.$
      	\item $P\bigg(\rho(t_{0,r}\wedge T_{(1)})\geq \frac{e^{-C_{0,r}+1}}{s}\bigg)\leq \epsilon.$
     	\end{enumerate}
    	\end{lemma}

    	\begin{proof}
    	The process $(m(t), t\geq 0)$ is a pure birth process with total birth rate $M_3^{(0,t]}(t)$ as defined in (\ref{M3ab}). Then, there is a mean-zero martingale $(W'(t),t\geq 0)$ such that for all $t\geq 0$,
     	$$
     	m(t)=W'(t)+\int_0^tM_3^{(0,u]}(u)du.
     	$$ 
    	By Doob's stopping theorem, $(W'(t\wedge T_{(1)}),t\geq 0)$ is a mean-zero martingale. Thus,
     	\begin{align*}
	  	E[m(t_{0,m} \wedge T_{(1)})]
	  	&=E\bigg[\int_0^{t_{0,m}\wedge T_{(1)}}\mu(X_1(u)+X_2(u))du\bigg]\\
	  	&\leq \int_0^{t_{0,m}}\mu\cdot\frac{2KN\mu}{s}e^{su}du\\
	  	&=\frac{2KN\mu^2}{s^2}(e^{st_{0,m}}-1)\\
	  	&\leq\frac{2Ke^{-C_{0,m}}}{s}.
     	\end{align*}
    	So, by Markov's inequality and by the way we choose $C_{0,m}$ in (\ref{C0m}),
	 	$$
	 	P\bigg(m(t_{0,m}\wedge T_{(1)})\geq \frac{e^{C_{0,m}/2}}{s}\bigg)
	 	\leq\frac{E[m(t_{0,m} \wedge T_{(1)})]}{e^{-C_{0,m}/2}/s}
	 	\leq 2Ke^{-C_{0,m}/2}
	 	\leq \epsilon.
	 	$$

	Now, consider the process $(\rho(t),t\geq 0)$. By (\ref{R3ab}), the process is a pure birth process, and the birth rate at time $t$ is given by  $R_3^{(0,t]}(t)$ as defined in (\ref{R3ab}). Then, there is a mean-zero martingale $(W''(t),t\geq 0)$ such that for all $t\geq 0$,
     	$$
     	\rho(t)=W''(t)+\int_0^tR_3^{(0,u]}(u)du,
     	$$ 
    	By Doob's stopping theorem, $(W''(t\wedge T_{(1)}),t\geq 0)$ is a mean-zero martingale. Thus,
     	\begin{align}
	  	E[\rho(t_{0,r} \wedge T_{(1)})]
	  	&=E\bigg[\int_0^{t_{0,r}\wedge T_{(1)}}R_3^{(0,u]}(u)du\bigg]\nonumber\\
	  	&\leq E\bigg[\int_0^{t_{0,r}\wedge T_{(1)}}Nr \tilde X_1(u)\tilde X_2(u)du\bigg]\nonumber\\
      	&\leq \int_0^{t_{0,r}}Nr\cdot\frac{K^2\mu^2}{s^2}e^{2su}du\nonumber\\
	  	&\leq\frac{K^2N\mu^2r}{s^3}\cdot e^{2st_{0,r}}\label{m&rho.1}
     	\end{align}
   	From the definition of $t_{0,r}$ in (\ref{t0r}), if we are in the recombination dominating case or in the mutation dominating case with $Nr\geq e$,
   		$$
   		\frac{K^2N\mu^2r}{s^3}\cdot e^{2st_{0,r}}
   		=\frac{K^2e^{-2C_{0,r}}}{s},
   		$$
   	and in the mutation dominating case when $Nr< e$, we have
   		$$
   		\frac{K^2N\mu^2r}{s^3}\cdot e^{2st_{0,r}}
   		=\frac{K^2e^{-2C_{0,r}}Nr}{s}
   		\leq\frac{K^2e^{-2C_{0,r}+1}}{s}.
   		$$
   	Hence, from (\ref{m&rho.1}), 
   		$$
   		E[\rho(t_{0,r} \wedge T_{(1)})]\leq \frac{K^2e^{-2C_{0,r}+1}}{s}.
   		$$
    	Lastly, by Markov's inequality and the definition of $C_{0,r}$ in (\ref{C0r}),
    		$$
     	P\bigg(\rho(t_{0,r}\wedge T_{(1)})\geq \frac{e^{-C_{0,r}+1}}{s}\bigg)
     	\leq\frac{E[\rho(t_{0,r} \wedge T_{(1)})]}{e^{-C_{0,r}+1}/s}
     	\leq K^2e^{-C_{0,r}}
     	\leq \epsilon.	  
     	$$
   	\end{proof}

	\begin{lemma} \label{3mlemma}
	For $i\in \mathbb{N}$, define $\tau_{i,m}$ to be the time that the ith type 3m ancestor appears, where we set $\tau_{i,m}=\infty$ if the ith type 3m ancestor never appears. Let $Y_{i,m}(t)$ be the number of descendants of the ith type 3m ancestor alive at time $t$. Then, for sufficiently large N, for all $i\in\mathbb{N}$,
	$$
	P\Big(\{Y_{i,m}(t_1)>0\}\cap\{t_1< T_{(1)}\}\Big|\tau_{i,m}\leq t_{0,m}\wedge T_{(1)}\Big)\leq 3s.
	$$
	\end{lemma} 

	\begin{proof} 
	First, define $\tilde Y_{i,m}(t)=Y_{i,m}(t)/N$ for all $t\geq 0$ and $i\in\mathbb{N}$. By following the same reasoning that led us to get the rates in (\ref{B3mab}) and (\ref{D3mab}), we have that on the event $\tau_{i,m}\leq t_{0,m}\wedge T_{(1)}$, the process $(Y_{i,m}(t+\tau_{i,m}),t\geq 0)$ is a birth-death process with $Y_{i,m}(\tau_{i,m})=1$, where each individual gives birth at rate
		$$
		b(t)=\big(\tilde X_0(t+\tau_{i,m})+(1-s)(\tilde X_1(t+\tau_{i,m})+\tilde X_2(t+\tau_{i,m}))+(1-2s)(\tilde X_3(t+\tau_{i,m})-\tilde Y_{i,m}(t+\tau_{i,m}))\big)(1-r\tilde X_0(t+\tau_{i,m})),
		$$
	and dies at rate 
		$$
		d(t)=(1-2s)\big(1-\tilde Y_{i,m}(t+\tau_{i,m})+r\tilde X_0(t+\tau_{i,m})\tilde Y_{i,m}(t+\tau_{i,m})\big).
		$$
	Note that for $t\geq 0$,
		\begin{align*}
		b(t)&\leq\tilde X_0(t+\tau_{i,m})+\tilde X_1(t+\tau_{i,m})+\tilde X_2(t+\tau_{i,m})+(\tilde X_3(t+\tau_{i,m})-\tilde Y_{i,m}(t+\tau_{i,m}))\\
		&= 1-\tilde Y_{i,m}(t+\tau_{i,m}),
		\end{align*}
	and
		$$
		d(t)\geq(1-2s)\big(1-\tilde Y_{i,m}(t+\tau_{i,m})).
		$$

	For $t\geq 0$, define $\lambda(t)=\int_{\tau_{i,m}}^{t+\tau_{i,m}}1-\tilde Y_{i,m}(v)dv$. Define $Y_{i,m}^*(t)=Y_{i,m}(\lambda^{-1}(t)+\tau_{i,m}) $ for $t\in[0,\lambda((t_1\wedge T_{(1)}) -\tau_{i,m})]$. The process $(Y^*_{i,m}(t),0\leq t<\lambda((t_1\wedge T_{(1)}) -\tau_{i,m}))$ is a birth-death process with $Y^*_{i,m}(0)=1$, where each individual gives birth at rate 
		$$
		b^*(t)=b(\lambda^{-1}(t))\cdot(\lambda^{-1})'(t)=\frac{b(\lambda^{-1}(t))}{1-\tilde Y_{i,m}(\lambda^{-1}(t)+\tau_{i,m})}\leq 1,
		$$
	and dies at rate
		$$
		d^*(t)=d(\lambda^{-1}(t))\cdot(\lambda^{-1})'(t)=\frac{d(\lambda^{-1}(t))}{1-\tilde Y_{i,m}(\lambda^{-1}(t)+\tau_{i,m})}\geq 1-2s.
		$$

	Let $(Y^\#(t),t\geq 0)$ be a birth-death process where $Y^\#(0)=1$, where each individual gives birth at rate 1 and dies at rate $1-2s$. From the generating function of birth and death process (in the section 5 of Chapter III of \cite{A-N}), for $t\geq 0$,
		\begin{equation}\label{-5.2.1}
		P(Y^\#(t)>0)
		= \frac{1-(1-2s)}{1-(1-2s)e^{-(1-(1-2s))t}}
		\leq\frac{2s}{1-e^{-2st}}.
		\end{equation}
	Since $1\ll N\mu$, we have that for sufficiently large $N$, 
		\begin{equation}\label{-5.3}
		P\bigg(Y^\#\Big(\frac{t_1-t_{0,m}}{2}\Big)>0\bigg)
		\leq \frac{2s}{1-e^{-s(t_1-t_{0,m})}}
		=\frac{2s}{1-\frac{1}{N\mu}e^{C_1-C_{0,m}}}\leq 3s.
		\end{equation}
	By Lemma \ref{G<} and (\ref{n<1/16}), on the event $t_1<T_{(1)}$ , we have $Y_{i,m}(t)\leq X_3(t)\leq \eta N\leq \frac{N}{2}$ for all $t\in[0,t_1]$, which implies that 
		\begin{equation}\label{lem3m.1}
		\lambda(t_1-t_{0,m})
		=\int_{\tau_{i,m}}^{t_1-t_{0,m}+\tau_{i,m}}1-\tilde Y_{i,m}(v)dv
		\geq\int_{\tau_{i,m}}^{t_1-t_{0,m}+\tau_{i,m}} \frac{1}{2} dv
		\geq \frac{t_1-t_{0,m}}{2}.
		\end{equation}
	It is possible to couple the process $(Y^\#(t),t\geq 0)$ with the population process, such that 1) on the event $t_1<T_{(1)} $, for any time $t$, if $Y^*_{i,m}(t)>0$, then $Y^\#(t)>0$, and 2) the process $(Y^\#(t),t\geq 0)$ is independent of $\mathcal{F}_{\tau_{i,m}}$. It follows that
		\begin{align*}
		&P\Big(\{Y_{i,m}(t_1)>0\}\cap\{t_1\leq T_{(1)}\}\Big|\tau_{i,m}\leq t_0\wedge T_{(1)}\Big)\nonumber\\
		&\hspace{0.5 cm}=P\Big(\{Y_{i,m}(t_1)>0\}\cap\{t_1<T_{(1)}\}\Big|\tau_{i,m}\leq t_{0,m}\wedge T_{(1)}\Big)\nonumber\\
		&\hspace{0.5 cm}\leq P\Big(\{Y_{i,m}(t_1-t_{0,m}+\tau_{i,m})>0\}\cap\{t_1<T_{(1)}\}\Big|\tau_{i,m}\leq t_{0,m}\wedge T_{(1)}\Big)\nonumber\\
		&\hspace{0.5 cm}=P\Big(\{Y^*_{i,m}(\lambda(t_1-t_{0,m}))>0\}\cap\{t_1< T_{(1)}\}\Big|\tau_{i,m}\leq t_{0,m}\wedge T_{(1)}\Big)\nonumber\\
		&\hspace{0.5 cm}\leq P\Big(\{Y^\#(\lambda(t_1-t_{0,m}))>0\}\cap\{t_1< T_{(1)}\}\Big|\tau_{i,m}\leq t_{0,m}\wedge T_{(1)}\Big).
		\end{align*}
	Lastly, using (\ref{lem3m.1}) and (\ref{-5.3}), we have
		\begin{align}
		&P\Big(\{Y_{i,m}(t_1)>0\}\cap\{t_1\leq T_{(1)}\}\Big|\tau_{i,m}\leq t_0\wedge T_{(1)}\Big)\nonumber\\
		&\hspace{0.5 cm}\leq P\Big(\Big\{Y^\#\Big(\frac{t_1-t_{0,m}}{2}\Big)>0\Big\}\cap\{t_1< T_{(1)}\}\Big|\tau_{i,m}\leq t_{0,m}\wedge T_{(1)}\Big)\nonumber\\
		&\hspace{0.5 cm}\leq P\Big(Y^\#\Big(\frac{t_1-t_{0,m}}{2}\Big)>0\Big|\tau_{i,m}\leq t_{0,m}\wedge T_{(1)}\Big)\nonumber\\
		&\hspace{0.5 cm}= P\Big(Y^\#\Big(\frac{t_1-t_0}{2}\Big)>0\Big)\nonumber\\
		&\hspace{0.5 cm}\leq 3s. \label{lem3m.2}
		\end{align}
	\end{proof} 

	\begin{lemma} \label{3rlemma}
	For $i\in \mathbb{N}$, define $\tau_{i,r}$ to be the time that the ith type 3r ancestor appears, where we set $\tau_{i,r}=\infty$, if the ith type 3r ancestor never appears. Let $Y_{i,r}(t)$ be the number of  descendants of the ith type 3r ancestor alive at time $t$. Then, for sufficiently large N,  for all $i\in\mathbb{N}$,
		$$
		P\Big(\{Y_{i,r}(t_1)>0\}\cap\{t_1<T_{(1)}\}\Big|\tau_{i,r}\leq t_{0,r}\wedge T_{(1)}\Big)\leq 4s.
		$$
	\end{lemma} 

	\begin{proof}
	The proof is similar to that of Lemma \ref{3mlemma}. First, define $\tilde Y_{i,r}(t)=Y_{i,r}(t)/N$ for all $t\geq 0$ and $i\in\mathbb{N}$. By following the same reasoning that led us to get the rates in (\ref{B3rab}) and (\ref{D3rab}), we have that on the event $\tau_{i,r}\leq t_{0,r}\wedge T_{(1)}$, the process $(Y_{i,r}(t+\tau_{i,r}),t\geq 0)$ is a birth-death process with $Y_{i,r}(\tau_{i,r})=1$, where each individual gives birth at rate
		$$
		b(t)=\big(\tilde X_0(t+\tau_{i,r})+(1-s)(\tilde X_1(t+\tau_{i,r})+\tilde X_2(t+\tau_{i,r}))+(1-2s)(\tilde X_3(t+\tau_{i,r})-\tilde Y_{i,r}(t+\tau_{i,r}))\big)(1-r\tilde X_0(t+\tau_{i,r})),
		$$
and dies at rate 
		$$
		d(t)=(1-2s)\big(1-\tilde Y_{i,r}(t+\tau_{i,r})+r\tilde X_0(t+\tau_{i,r})\tilde Y_{i,r}(t+\tau_{i,r})-r\tilde X_1(t+\tau_{i,r})\tilde X_2(t+\tau_{i,r})\big).
		$$
Note that when $t\geq 0$, we have $b(t)\leq 1-\tilde Y_{i,r}(t+\tau_{i,r})$. 

	For $t\geq 0$, let $\lambda(t)=\int_{\tau_{i,r}}^{t+\tau_{i,r}}1-\tilde Y_{i,r}(v)dv$. Define $Y_{i,r}^*(t)=Y_{i,r}(\lambda^{-1}(t)+\tau_{i,r}) $ for $t\in[0,\lambda((t_1\wedge T_{(1)}) -\tau_{i,r})]$. The process $(Y^*_{i,m}(t),0\leq t<\lambda((t_1\wedge T_{(1)}) -\tau_{i,r}))$ is a birth-death process with $Y^*_{i,r}(0)=1$, where each individual gives birth at rate 
		$$
		b^*(t)=b(\lambda^{-1}(t))\cdot(\lambda^{-1})'(t)
		=\frac{b(\lambda^{-1}(t))}{1-\tilde Y_{i,r}(\lambda^{-1}(t)+\tau_{i,r})}
		\leq 1,
		$$
	and dies at rate
		\begin{align}
		d^*(t)
		&=d(\lambda^{-1}(t))\cdot(\lambda^{-1})'(t)\nonumber\\
		&=\frac{d(\lambda^{-1}(t))}{1-\tilde Y_{i,r}(\lambda^{-1}(t)+\tau_{i,r})}\nonumber\\
		&\geq (1-2s)\bigg(1-\frac{r\tilde X_1(\lambda^{-1}(t)+\tau_{i,r})\tilde X_2(\lambda^{-1}(t)+\tau_{i,r})}{1-\tilde Y_{i,r}(\lambda^{-1}(t)+\tau_{i,r})}\bigg). \label{d*}
		\end{align}
	Since the function $\lambda$ is strictly increasing on the interval $[0,(t_1\wedge T_{(1)}) -\tau_{i,r})$, we have that if $t\in [0,\lambda((t_1\wedge T_{(1)}) -\tau_{i,r}))$, then $\lambda^{-1}(t)+\tau_{i,r}(t)\leq t_1\wedge T_{(1)}.$ Hence, from Lemma \ref{G<}, for every $t\in [0,\lambda((t_1\wedge T_{(1)}) -\tau_{i,r}))$ and $j=1,2$ and 3, we have $\tilde X_j(\lambda^{-1}(t)+\tau_{i,r})\leq \eta$, and $\tilde Y_{i,r}(\lambda^{-1}(t)+\tau_{i,r})\leq \tilde X_3(\lambda^{-1}(t)+\tau_{i,r})\leq \eta$. Now, because $r \ll s$, by (\ref{d*}), for sufficiently large $N$, for $t\in[0,\lambda((t_1\wedge T_{(1)}) -\tau_{i,r}))$,
		$$
		d^*(t)\geq(1-2s)\bigg(1-\Big(\frac{\eta^2}{1-\eta}\Big)r\bigg)\geq (1-2s)(1-s) >1-3s.
		$$

	Let $(Y^\#(t),t\geq 0)$ be a birth-death process where $Y^\#(0)=1$, where each individual gives birth at rate 1 and dies at rate $1-3s$. By the same argument we used to get (\ref{-5.2.1}), for $t\geq 0$,
		\begin{equation}\label{-5.4}
		P(Y^\#(t)>0)= \frac{1-(1-3s)}{1-(1-3s)e^{-(1-(1-3s))t}}\leq\frac{3s}{1-e^{-3st}}.
		\end{equation}
	We claim that for sufficiently large $N$,
		$$
		P\bigg(Y^\#\Big(\frac{t_1-t_{0,r}}{3}\Big)>0\bigg) \leq 4s.
		$$
	From (\ref{-5.4}) and the definition of $C_{0,r}$ in (\ref{C0r}), in the recombination dominating case and the  mutation dominating case with $Nr\geq e$, we have that for sufficiently large $N$, 
		$$
		P\bigg(Y^\#\Big(\frac{t_1-t_{0,r}}{3}\Big)>0\bigg)
		\leq \frac{3s}{1-e^{-s(t_1-t_{0,r})}}
		=\frac{3s}{1-\frac{1}{\sqrt{Nr}}e^{-(C_{0,r}-C_1)}}
		< \frac{3s}{1-e^{-(C_{0,r}-C_1)}}
		\leq 4s,
		$$
	and in the mutation dominating case with $Nr\leq e$, we also have
		$$
		P\bigg(Y^\#\Big(\frac{t_1-t_{0,r}}{3}\Big)>0\bigg)
		\leq \frac{3s}{1-e^{-s(t_1-t_{0,r})}}
		=\frac{3s}{1-e^{-(C_{0,r}-C_1)}}
		\leq 4s.
		$$
	On the event $t_1<T_{(1)}$ , using (\ref{n<1/16}), we have $Y_{i,r}(t)\leq X_3(t)\leq \eta N\leq \frac{N}{3}$ for all $t\in[0,t_1]$. By following the same reasoning in (\ref{lem3m.1}), 
		$$
		\lambda(t_1-t_{0,r})\geq \frac{t_1-t_{0,r}}{3}.
		$$
	It is possible to couple the process $(Y^\#(t),t\geq 0)$ with the population process, such that 1) on the event $t_1<T_{(1)}$, for any time $t$, if $Y^*_{i,m}(t)>0$, then $Y^\#(t)>0$, and 2) the process $(Y^\#(t),t\geq 0)$ is independent of $\mathcal{F}_{\tau_{i,r}}$.
By the same reasoning we used to get (\ref{lem3m.2}), it follows that for sufficiently large $N$,
		$$
		P\Big(\{Y_{i,r}(t_1)>0\}\cap\{t_1< T_{(1)}\}\Big|\tau_{i,r}\leq t_{0,r}\wedge T_{(1)}\Big)
		\leq P\Big(Y^\#\Big(\frac{t_1-t_{0,r}}{3}\Big)>0\Big)			
		\leq 4s.
		$$
	\end{proof}

Now, we are ready to show that the events $A_8$ and $A_9$ occur with probability close to 1.

	\begin{lemma} \label{PX3t1}
	For sufficiently large $N$, we have $P(A_8^c)\leq 4\epsilon$, and $P(A_9^c)\leq 4\epsilon$.
	\end{lemma}

	\begin{proof}
	Recall the definitions of $A_8$ and $A_9$ in (\ref{A8}) and (\ref{A9}). We will only show that $P(A_8^c)\leq 4\epsilon$. The same reasoning can be used to prove that $P(A_9^c)\leq 4\epsilon$. 
	
	Let $J=\lfloor e^{-C_{0,m}/2}/s\rfloor$. By Lemma \ref{3mlemma}, we have that for sufficiently large $N$,
		\begin{align*}
		&P\big(\{X_{3m}^{(0,t_{0,m}]}(t_1 )>0\}\cap\{m(t_{0,m}\wedge T_{(1)})<e^{-C_{0,m}/2}/s\}\cap\{t_1< T_{(1)}\}\big)\\
		&\leq \sum_{i=1}^{J}P\Big(\{Y_{i,m}(t_1)>0\}\cap\{\tau_{i,m}\leq t_{0,m}\wedge T_{(1)}\}\cap\{t_1<T_{(1)}\}\Big)\\
		&\leq \sum_{i=1}^J P\Big(\{Y_{i,m}(t_1)>0\}\cap\{t_1<T_{(1)}\}\Big|\tau_{i,m}\leq t_{0,m}\wedge T_{(1)}\Big)\\
		&\leq 3sJ\\
		&\leq 3e^{-C_{0,m}/2}.
		\end{align*}
	Hence, by  Lemma \ref{T<t1} and Lemma \ref{m&rho} along with the way we choose $\epsilon, K$ and $C_{0,m}$ in (\ref{eps}), (\ref{K}) and (\ref{C0m}), for sufficiently large $N$,
		\begin{align*}
		&P(\{X_{3m}^{(0,t_{0,m}]}(t_1\wedge T_{(1)})>0\})\\
		&\leq P\big(\{X_{3m}^{(0,t_{0,m}]}(t_1\wedge T_{(1)})>0\}\cap\{m(t_{0,m}\wedge T_{(1)})<e^{-C_{0,m}/2}/s\}\cap\{t_1<T_{(1)}\}\big)\\
		&\hspace{0.5 cm}+P(m(t_{0,m}\wedge T_{(1)})\geq e^{-C_{0,m}/2}/s)+ P(T_{(1)}\leq t_1)\\
		&\leq 3e^{-C_{0,m}/2}+3\epsilon\\
		&\leq 4\epsilon.
		\end{align*}
	So, this prove that $P(A_9^c)\leq 4\epsilon$.
	\end{proof}
	
	\begin{lemma} \label{X3t0-t1}
	For sufficiently large $N$, we have $P(A_{10}^c)\leq \epsilon$, and $P(A_{11}^c)\leq \epsilon$.
	\end{lemma}
	
	\begin{proof}
	Recall the definition of $A_{10}$ in (\ref{A10}). From Lemma \ref{EX3m} and the definition of $t_1$ in (\ref{t1}), for sufficiently large $N$,
		$$
		E\bigg[X_{3m}^{(t_{0,m},t_1]}(t_1\wedge T_{(1)})\bigg]\leq \frac{2Ke^{C_{0,m}}N^2\mu^4}{s^3}e^{2st_1}=\frac{2Ke^{-2C_1+C_{0,m}}N^2\mu^2}{s},
		$$
	and from the Markov's inequality, we get that $P(A_{10}^c)\leq \epsilon$. 
	
	Now, recall the definition of $A_{11}$ in (\ref{A11}). We will first consider the recombination dominating case and the mutation dominating case with $Nr\geq e$. Recall that $Nr\gg 1$ in the recombination dominating case. From Lemma \ref{EX3r} and the definition of $t_{0,r}$ in (\ref{t0r}), for sufficiently large $N$,
		\begin{align*}
		E\bigg[X_{3r}^{(t_{0,r},t_1]}(t_1\wedge T_{(1)})\bigg]
		&\leq \frac{K^2N\mu^2r}{s^2}e^{2st_1}(t_1-t_{0,r})\\
		&=K^2e^{-2C_1}Nr\bigg(\frac{\ln(Nr)}{2s}+\frac{C_{0,r}-C_1}{s}\bigg)\\
		&\leq \frac{K^2e^{-2C_1}(2(C_{0,r}-C_1)+1)Nr\ln(Nr)}{2s}.
		\end{align*}
	In the mutation dominating case with $Nr<e$, from Lemma \ref{EX3r} and the definition of $t_{0,r}$ in (\ref{t0r}), for sufficiently large $N$,
		$$
		E\bigg[X_{3r}^{(t_{0,r},t_1]}(t_1\wedge T_{(1)})\bigg]\leq \frac{K^2N\mu^2r}{s^2}e^{2st_1}(t_1-t_{0,r})=K^2e^{-2C_1}Nr\bigg(\frac{C_{0,r}-C_1}{s}\bigg)\leq \frac{K^2e^{-2C_1}(2(C_{0,r}-C_1)+1)}{2s}.
		$$ 
	Thus, in both cases, for sufficiently large $N$,
		$$
		E\bigg[X_{3r}^{(t_{0,r},t_1]}(t_1\wedge T_{(1)})\bigg]\leq\frac{K^2e^{-2C_1+1}(2(C_{0,r}-C_1)+1)\big(1\vee Nr\ln_+(Nr)\big)}{2s},
		$$ 
	and $P(A_{11}^c)\leq \epsilon$ is followed from the Markov's inequality.
	\end{proof}
	
	Before we prove Proposition \ref{@t1}, we will give both upper and lower bounds of the numbers of type 1 and 2 individuals on the event $A_{(1)}$.

	\begin{lemma} \label{lem@t1}
	The following statements hold:
		\begin{enumerate}
		\item  On the event $A_{(1)}$, for $i=1,2$, for sufficiently large $N$ and for $t\in[0,t_1]$, 
		$$\displaystyle X_i(t) \leq (1+\delta^2)\frac{N\mu}{s}e^{st}.$$
		\item  In the recombination dominating case, on the event $A_{(1)}$, for $i=1,2$, for sufficiently large $N$, and for every $t\in[t_{0,r},t_1]$, we have 
		$$\displaystyle X_i(t) \geq (1-\delta^2)\frac{N\mu}{s}e^{st}.$$
		\item In the mutation dominating case, on the event $A_{(1)}$, for $i=1,2$, for sufficiently large $N$, and for $t\in[t_{0,m},t_1]$, we have 
		$$\displaystyle X_i(t) \geq (1-\delta^2)\frac{N\mu}{s}e^{st}.$$
		\end{enumerate}
	\end{lemma}
	
	\begin{proof}
	In this proof, we assume that we are on the event $A_{(1)}$. From (\ref{Xm}), we have that for all $t\in(0,t_1]$,
		\begin{equation}\label{X1m}
		X_{1m}(t)=X_{1m}^{(0,t_1]}(t)=\int_0^t M_1^{(0,t_1]}(u)e^{\int_u^t G_1(v)dv}du+Z_{1m}^{(0,t_1]}(t)e^{\int_0^t G_1(v)dv}.
		\end{equation}
 	From Lemma \ref{G<}, definitions of $A_1$ and $A_2$ in (\ref{A1}) and (\ref{A2}), and the fact that $1\ll N\mu$, for sufficiently large $N$ and for $t\in(0,t_1]$,
		\begin{align*}
		X_{1m}(t)&\leq \int_0^t N\mu e^{\int_u^t s dv}du + \sqrt{\frac{48}{\epsilon}\cdot\frac{N\mu}{s^2}}\cdot e^{\int_0^t s dv}\\
		&=\frac{N\mu}{s}(e^{st}-1)+ \sqrt{\frac{48}{\epsilon}\cdot\frac{N\mu}{s^2}}\cdot e^{st}\\
		&\leq \bigg(1+ \sqrt{\frac{48}{\epsilon}\cdot\frac{1}{N\mu}}\bigg)\frac{N\mu}{s}e^{st},\\
		&\leq \Big(1+\frac{\delta^2}{2}\Big)\frac{N\mu}{s}e^{st}.
		\end{align*}
	Next, from (\ref{Xr}), we have that for all $t\in(0,t_1]$,
		$$
		X_{1r}(t)=X_{1r}^{(0,t_1]}(t)=\int_0^t R_1^{(0,t_1]}(u)e^{\int_u^t G^{(0,t_1]}_{1r}(v)dv}du+Z_{1r}^{(0,t_1]}(t)e^{\int_0^t G^{(0,t_1]}_{1r}(v)dv}.
		$$
	From (\ref{R1<}), Lemma \ref{G<}, and definitions of $A_1$ and $A_4$ in (\ref{A1}) and (\ref{A4}), for sufficiently large $N$ and for $t\in(0,t_1]$,
		$$
		X_{1r}(t)
		\leq \int_0^t Nr\tilde X_0(u)\tilde X_3(u) e^{\int_u^t (s+r)dv}du + \sqrt{\frac{48}{\epsilon}\cdot\frac{N\mu r}{s^3}\ln\Big(\frac{s}{\mu}\Big)}e^{\int_0^t (s+r) dv}.
		$$
	By the definition of $T_3$ in (\ref{T3}), inequalities (\ref{rslogsu}), (\ref{(r+u)t1}) and the fact that $1\ll N\mu$, it follows that  for sufficiently large $N$ and for $t\in(0,t_1]$,
		\begin{align*}
		X_{1r}(t)
		&\leq Nr\int_0^t \frac{\mu}{s}e^{su}\cdot e^{s(t-u)+rt}du+ \sqrt{\frac{48}{\epsilon}\cdot\frac{N\mu r}{s^3}\ln\Big(\frac{s}{\mu}\Big)}e^{st+rt}\\
		&=e^{rt}\bigg(\frac{N\mu r}{s}e^{st}t+\sqrt{\frac{48}{\epsilon}\cdot\frac{N\mu r}{s^3}\ln\Big(\frac{s}{\mu}\Big)}e^{st}\bigg)\\
		&\leq \frac{N\mu}{s}e^{st}\cdot e^{rt_1}\bigg(rt_1+\sqrt{\frac{48}{\epsilon}\cdot\frac{1}{N\mu}\cdot\frac{r}{s}\ln\Big(\frac{s}{\mu}\Big)}\bigg)\\
		&\leq \frac{\delta^2}{2}\cdot\frac{N\mu}{s}e^{st}.
		\end{align*}
 	Therefore, for sufficiently large $N$, for all $t\in[0,t_1]$, we have
		$$
		X_1(t)=X_{1m}(t)+X_{1r}(t)\leq (1+\delta^2)\frac{N\mu}{s}e^{st},
		$$
Note that by similar argument, we can also prove the upper bound for $X_2(t)$.

	To prove the lower bound for $X_1(t)$ in the recombination dominating case, we first need to consider the term $\int_u^t G_1(v)dv$. By using (\ref{G1}), part 1 of this lemma and the definition of $T_3$ in (\ref{T3}), we have that when $N$ is sufficiently large, for $0\leq u<t\leq t_1$,
		\begin{align} 
		\int_u^tG_1(v)dv
		&\geq\int_u^t(s-s\tilde{X}_1(v)- s\tilde{X}_2(v)-2s\tilde{X}_3(v)-r-\mu) dv \nonumber\\
		&\geq\int_u^t\bigg(s-s\cdot (1+\delta^2)\frac{\mu}{s}e^{su}- s\cdot (1+\delta^2)\frac{\mu}{s}e^{su}-2s\cdot\frac{\mu}{s}e^{su}-r-\mu\bigg) dv \nonumber\\
		&= s(t-u)-\frac{(4+2\delta^2)\mu}{s}(e^{st}-e^{su})-(r+\mu)(t-u)\nonumber\\
		&= s(t-u)-\frac{(4+2\delta^2)\mu}{s}e^{st_1}-(r+\mu)t_1.\nonumber
		\end{align}
	Now, using the fact that $\delta<1$, the definition of $t_1$ in (\ref{t1}) along with (\ref{(r+u)t1}), we have that when $N$ is sufficiently large, for $0\leq u<t\leq t_1$,
		\begin{align} 
		\int_u^tG_1(v)dv 
		&\geq s(t-u)-\frac{6\mu}{s}e^{st_1}-(r+\mu)t_1\nonumber\\
		&=s(t-u)-6e^{-C_1}-(r+\mu)t_1\nonumber\\
		&\geq s(t-u)-7e^{-C_1}.\label{intG1}
		\end{align}
	Also, using part 1 of this lemma, the definition of $T_3$ in (\ref{T3}) and the fact that $\delta <1$, for sufficiently large $N$, and for $u\in[0,t_1]$, 
		\begin{equation} \label{X0}
		X_0(u)=N-X_1(u)-X_2(u)-X_3(u)\geq N-2(1+\delta^2)\frac{N\mu}{s}e^{su}-\frac{N\mu}{s}e^{su}\geq N-\frac{5N\mu}{s}e^{su}.
		\end{equation} 
Thus, from (\ref{X1m}), (\ref{intG1}), (\ref{X0}), along with the definition of $A_4$ in (\ref{A4})  for sufficiently large $N$, for all $t\in[t_{0,r},t_1]$,
		\begin{align}
		X_1(t)
		&\geq X_{1m}^{(0,t_1]}(t)\nonumber\\
		&\geq\int_0^t\mu\bigg(N-\frac{5N\mu}{s}e^{su}\bigg)e^{s(t-u)-7e^{-C_1}}du-\sqrt{\frac{48}{\epsilon}\cdot\frac{N\mu}{s^2}}e^{st}\nonumber\\
		&=\bigg(e^{-7e^{-C_1}}\int_0^t(se^{-su}-5\mu)du-\sqrt{\frac{48}{\epsilon}\cdot\frac{1}{N\mu}}\bigg)\bigg(\frac{N\mu}{s}e^{st}\bigg)\nonumber\\
		&\geq\bigg((1-7e^{-C_1})(1-e^{-st}-5\mu t)-\sqrt{\frac{48}{\epsilon}\cdot\frac{1}{N\mu}}\bigg)\bigg(\frac{N\mu}{s}e^{st}\bigg)\nonumber\\
		&\geq\bigg((1-7e^{-C_1})(1-e^{-st_{0,r}}-5\mu t_1)-\sqrt{\frac{48}{\epsilon}\cdot\frac{1}{N\mu}}\bigg)\bigg(\frac{N\mu}{s}e^{st}\bigg). \label{X1>}
		\end{align}
	In the recombination dominating case, we have $N\mu^2 \ll s$ and $r\ll s$. So, by using the definition of $t_{0,r}$ in (\ref{t0r}), we have that
		$$
		st_{0,r}=\frac{1}{2}\ln\Big(\frac{s^2}{N\mu^2r}\Big)-C_{0,r}\gg 1. 
		$$
	Thus, from (\ref{X1>}), (\ref{(r+u)t1}), and the way we choose $C_1$ as in (\ref{C1}), for sufficiently large $N$, and for all $t\in[t_{0,r},t_1]$,
		$$
		X_1(t)\geq (1-8e^{-C_1})\frac{N\mu}{s}e^{st}\geq (1-\delta^2)\frac{N\mu}{s}e^{st}.
		$$

	The proof for the mutation dominating case is almost exactly the same as that of the recombination dominating case by replacing $t_{0,r}$ by $t_{0,m}$, and using that because $N\mu^2 \ll s$, we have
		$$
		st_{0,m}=\ln\Big(\frac{s}{N\mu^2}\Big)-C_{0,m}\gg 1,
		$$
	which completes the proof.
	\end{proof}	
	
\subsection{The proof of Proposition \ref{@t1}}

	\begin{proof}
	By the definition of $A_{(1)}$ in (\ref{A(1)}) and Lemmas \ref{T<t1}, \ref{|Z|}, \ref{PX3t1}, and \ref{X3t0-t1},  for sufficiently large $N$, we have that $P(A_{(1)})\geq 1-17\epsilon$. From now on, we will assume that we are working on the event $A_{(1)}$.  The statement 1 follows from Lemma \ref{lem@t1} by inserting $t=t_1$. 
	
	Now consider $X_3(t_1)$. From the definitions of $A_8, A_9, A_{10}$ and $A_{11}$, in (\ref{A8}), (\ref{A9}), (\ref{A10}) and (\ref{A11}), it follows that 
		\begin{equation}\label{X3mt1<}
		X_{3m}(t_1)=X^{(0,t_{0,m}]}_{3m}(t_1)+X^{(t_{0,m},t_1]}_{3m}(t_1)
		\leq \bigg(\frac{2Ke^{-2C_1+C_{0,m}}}{\epsilon}\bigg)\frac{N^2\mu^2}{s},
		\end{equation}
	and
		\begin{equation}\label{X3rt1<}
		X_{3r}(t_1)=X^{(0,t_{0,r}]}_{3r}(t_1)+X^{(t_{0,r},t_1]}_{3r}(t_1)
		\leq \bigg(\frac{K^2e^{-2C_1}(2(C_{0,r}-C_1)+1)}{2\epsilon}\bigg)\frac{(1\vee Nr\ln_+(Nr))}{s}.
		\end{equation}
	In the recombination dominating case, $Nr\gg 1$ and $r$ satisfy (\ref{Con2.1}). It follows from (\ref{X3mt1<}) and (\ref{X3rt1<}) that if $N$ is sufficiently large, then
		$$
		X_3(t_1)=X_{3m}(t_1)+X_{3r}(t_1)
		\leq \bigg(\frac{K^2e^{-2C_1}(2(C_{0,r}-C_1)+1)}{\epsilon}\bigg)\frac{Nr\ln(Nr)}{s}.
		$$
	So, we choose the positive constant 
		$$
		K_{1r}^+=\frac{K^2e^{-2C_1}(2(C_{0,r}-C_1)+1)}{\epsilon}.
		$$ 
	Next, consider the mutation dominating case. In this case, $r$ satisfies (\ref{Con2.2}), and we also have that $1\ll N\mu$. It follows from (\ref{X3mt1<}) and (\ref{X3rt1<}) that if $N$ is sufficiently large, then
		$$
		X_{3r}(t_1)
		\leq \bigg(\frac{K^2e^{-2C_1}(2(C_{0,r}-C_1)+1)}{2\epsilon}\bigg)\frac{CN^2\mu^2}{s},
		$$
	and
		$$
		X_3(t_1)=X_{3m}(t_1)+X_{3r}(t_1)
		\leq \bigg(\frac{4Ke^{-2C_1+C_{0,m}}+K^2e^{-2C_1}(2(C_{0,r}-C_1)+1)C}{2\epsilon}\bigg)\frac{N^2\mu^2}{s}.
		$$
	Thus, we choose the positive constant 
		$$
		K_{1m}^+=\frac{4Ke^{-2C_1+C_{0,m}}+K^2e^{-2C_1}(2(C_{0,r}-C_1)+1)C}{2\epsilon}.
		$$
		
	Now, we will show the lower bound of $X_3(t_1)$. First, consider the recombination dominating case. To prove the lower bound, we will need to consider the term $\int_u^{t_1}G^{(t_{0,r},t_1]}_{3r}(v)dv$. Similar to the way we get (\ref{intG1}) by using (\ref{G3r}) instead of (\ref{G1r}), for $t_{0,r}\leq u\leq t_1$, 
		\begin{align*} 
		\int_u^{t_1}G^{(t_{0,r},t_1]}_{3r}(v)dv
		&\geq\int_u^{t_1}G_3(v)dv\\
		&\geq\int_u^{t_1}\big(2s-s\tilde{X}_1(v)- s\tilde{X}_2(v)-2s\tilde{X}_3(v)-r\big)dv \\
		&\geq\int_u^{t_1}\Big(2s-s\cdot (1+\delta^2)\frac{\mu}{s}e^{su}- s\cdot (1+\delta^2)\frac{\mu}{s}e^{su}-2s\cdot\frac{\mu}{s}e^{su}-r\Big)dv \\
		&\geq 2s(t_1-u)-\frac{6\mu}{s}e^{st_1}-rt_1\\
		&= 2s(t_1-u)-6e^{-C_1}-rt_1.
		\end{align*}
	By (\ref{(r+u)t1}), when $N$ is sufficiently large, for $t_{0,r}\leq u\leq t_1$
		\begin{equation}\label{intG3r}
		\int_u^{t_1}G^{(t_{0,r},t_1]}_{3r}(v)dv \geq 2s(t_1-u)-7e^{-C_1}.
		\end{equation}
	By (\ref{R3ab}) and Lemma \ref{lem@t1}, for sufficiently large $N$, and for $t\in[t_{0,r},t_1]$,
		\begin{align}
		R_3^{(t_{0,r},t_1]}(t)
		&\geq X_0(t)\cdot r\tilde X_1(t)\tilde X_2(t)\nonumber\\
		&=(N-X_1(t)-X_2(t)-X_3(t))\cdot r\tilde X_1(t)\tilde X_2(t)\nonumber\\
		&\geq\Big(N-2(1+\delta^2)\frac{N\mu}{s}e^{st}-\frac{N\mu}{s}e^{st}\Big)\Big((1-\delta^2)^2\frac{\mu^2r}{s^2}e^{2st}\Big)\nonumber\\
		&\geq\Big(N-\frac{5N\mu}{s}e^{st_1}\Big)\Big((1-\delta^2)^2\frac{\mu^2r}{s^2}e^{2st}\Big)\nonumber\\
		&=(1-5e^{-C_1})(1-\delta^2)^2\cdot\frac{N\mu^2r}{s^2}e^{2st}.\label{R3>}
		\end{align}
	 Using (\ref{Xr}), (\ref{intG3r}), (\ref{R3>}), Lemma \ref{G<} and the definitions of $A_7$ in (\ref{A7}), for sufficiently large $N$,
		\begin{align*}
		X_3(t_1)
		&\geq X_{3r}^{(t_{0,r},t_1]}(t_1)\\
		&\geq \int_{t_{0,r}}^{t_1}(1-5e^{-C_1})(1-\delta^2)^2\cdot\frac{N\mu^2r}{s^2}e^{2su}\cdot e^{2s(t_1-u)-7e^{-C_1}}du-\sqrt{\frac{16K^2e^{-2C_{0,r}}}{\epsilon}\cdot \frac{\ln(Nr)}{s^2}}\cdot e^{\int_{t_{0,r}}^{t_1} (2s+r) dv}\\
		&= e^{-7e^{-C_1}}(1-5e^{-C_1})(1-\delta^2)^2\cdot\frac{N\mu^2r}{s^2}e^{2st_1}(t_1-t_{0,r})-\sqrt{\frac{16K^2e^{-2C_{0,r}}}{\epsilon}\cdot \frac{\ln(Nr)}{s^2}}\cdot e^{(2s+r)(t_1-t_{0,r})}.
		\end{align*}
	It follows from the definitions of $t_1$ and $t_{0,r}$ in (\ref{t1}) and (\ref{t0r}) that for sufficiently large $N$,
		\begin{align*}
		X_3(t_1)&\geq e^{-7e^{-C_1}}(1-5e^{-C_1})(1-\delta^2)^2\cdot\frac{e^{-2C_1}Nr}{2s}\big(\ln(Nr)+2(C_{0,r}-C_1)\big)\\
		&\hspace{1 cm}-\frac{4Ke^{-2C_1+C_{0,r}+rt_1}}{\sqrt{\epsilon}}\cdot \frac{Nr\sqrt{\ln(Nr)}}{s}\\
		&=\frac{Nr\ln(Nr)}{s}\cdot\bigg(e^{-7e^{-C_1}}(1-5e^{-C_1})(1-\delta^2)^2e^{-2C_1}\Big(\frac{1}{2}+\frac{(C_{0,r}-C_1)}{\ln(Nr)}\Big)\\
		&\hspace{1 cm}-\frac{4Ke^{-2C_1+C_{0,r}+rt_1}}{\sqrt{\epsilon}}\cdot \frac{1}{\sqrt{\ln(Nr)}}\bigg).
		\end{align*}
	By (\ref{(r+u)t1}) and the fact that $1\ll Nr$, we have that for sufficiently large $N$
		$$
		X_3(t_1)\geq \bigg(\frac{e^{-7e^{-C_1}}(1-5e^{-C_1})(1-\delta^2)^2e^{-2C_1}}{3}\bigg)\frac{Nr\ln(Nr)}{s},
		$$
	and we choose the positive constant
		$$
		K_{1r}^-=\frac{e^{-7e^{-C_1}}(1-5e^{-C_1})(1-\delta^2)^2e^{-2C_1}}{3}.
		$$

	Lastly, consider the mutation dominating case. By the same argument we used to obtain (\ref{intG3r}), we have that for sufficiently large $N$ and for $t_{0,m}\leq u\leq t_1$,  
		$$
		\int_u^{t_1}G_3(v)dv \geq 2s(t_1-u)-7e^{-C_1}.
		$$
	From (\ref{Xm}),  Lemma \ref{G<}, Lemma \ref{lem@t1}, and the definition of $A_6$ in (\ref{A6}), for sufficiently large $N$,
		\begin{align*}
		X_3(t_1)
		&\geq X_{3m}^{(t_{0,m}^+,t_1]}(t_1)\\
		&=\int_{t_{0,m}^+}^{t_1} \mu(X_1(u)+X_2(u))e^{\int_u^{t_1}G_3(v)dv}du+Z^{(t_{0,m}^+,t_1]}_{3m}(t_1)e^{\int_{t_{0,m}^+}^{t_1}G_3(v)dv}\\
		&\geq \int_{t_{0,m}^+}^{t_1}2(1-\delta^2)\cdot\frac{N\mu^2}{s}e^{su}\cdot e^{2s(t_1-u)-7e^{-C_1}}du-\sqrt{\frac{48Ke^{C_{0,m}^+}}{\epsilon}\cdot \frac{1}{s^2}}\cdot e^{\int_{t_{0,m}^+}^{t_1} 2s dv}\\
		&= 2(1-\delta^2)e^{-7e^{-C_1}}\cdot\frac{N\mu^2}{s^2}e^{2st_1}(e^{-st_{0,m}^+}-e^{-st_1})-\sqrt{\frac{48Ke^{C_{0,m}^+}}{\epsilon}\cdot \frac{1}{s^2}}\cdot e^{2s(t_1-t_{0,m}^+)}.
		\end{align*}
	Using the definitions of $t_1$ and $t_{0,m}$ in (\ref{t1}) and (\ref{t0m}), and the fact that $1\ll N\mu$, for sufficiently large $N$,
		\begin{align*}
		X_3(t_1)
		&\geq 2(1-\delta^2)e^{-7e^{-C_1}}\cdot Ne^{-2C_1}\Big(e^{-C_{0,m}^+}\cdot \frac{N\mu^2}{s}-e^{C_1}\cdot \frac{\mu}{s}\Big)-\sqrt{\frac{48Ke^{C_{0,m}^+}}{\epsilon}\cdot \frac{1}{s^2}}\cdot e^{-2C_1-2C_{0,m}^+}N^2\mu^2\\
		&=\frac{N^2\mu^2}{s}\Bigg(2(1-\delta^2)e^{-7e^{-C_1}-2C_1-C_{0,m}^+}\Big(1-e^{C_1+C_{0,m}^+}\cdot \frac{1}{N\mu}\Big)-\sqrt{\frac{48Ke^{C_{0,m}^+}}{\epsilon}}\cdot e^{-2C_1-2C_{0,m}^+}\Bigg)\\
		&\geq \frac{N^2\mu^2}{s}\Bigg((1-\delta^2)e^{-7e^{-C_1}-2C_1-C_{0,m}^+}-\sqrt{\frac{48Ke^{C_{0,m}^+}}{\epsilon}}\cdot e^{-2C_1-2C_{0,m}^+}\Bigg)\\
		&= \frac{N^2\mu^2}{s}\cdot e^{-2C_1-2C_{0,m}^+}\Bigg((1-\delta^2)e^{-7e^{-C_1}}-\sqrt{\frac{48Ke^{-C_{0,m}^+}}{\epsilon}}\Bigg).
		\end{align*}
	Note that the way we define $C_{0,m}^+$ in (\ref{C0m+}) is precisely to make
		$$
		(1-\delta^2)e^{-7e^{-C_1}}-\sqrt{\frac{48Ke^{-C_{0,m}^+}}{\epsilon}}>0.
		$$
	Hence, we choose the positive constant
		$$
		K_{1m}^-=(1-\delta^2)e^{-7e^{-C_1}-2C_1-C_{0,m}^+}-\sqrt{\frac{48Ke^{C_{0,m}^+}}{\epsilon}}\cdot e^{-2C_1-2C_{0,m}^+}.		
		$$
	This completes the proof.
	\end{proof}

\section{Phase 2 and the proof of Proposition \ref{@t2}}\label{phase2}

\subsection{Comparing the Markov chain with a differential equation} 
	Theorem \ref{DNthm} below is a special case of Theorem 4.1 of \cite{Dar}. Let $(\mathbf{X}(t),t\geq 0)$ be a continuous time Markov chain with finite state space $S\subset\mathbb{R}^3$. Let $q(\xi,\xi')$ be the jump rate from the state $\xi$ to the state $\xi'$. For each state $\xi \in S$, define the function $\alpha:S\rightarrow \mathbb{R}$ by
		\begin{equation}\label{alpha}
		\alpha(\xi)=\sum_{\xi'\neq\xi}|\xi'-\xi|^2q(\xi,\xi'),
		\end{equation}
	where $|\cdot|$ is the Euclidean norm, and define the function $\beta:S\rightarrow \mathbb{R}^3$ by
		\begin{equation}\label{beta}
		\beta(\xi)=\sum_{\xi'\neq\xi}(\xi'-\xi)q(\xi,\xi'). 
		\end{equation}
	It follows that 
		$$
		\mathbf{X}(t)=\mathbf{X}(0)+M(t)+\int_0^t\beta(\mathbf{X}(s))ds, \hspace{0.5cm}\mbox{for $t\geq 0$},
		$$
	for some martingale $(M(t),t\geq 0)$.

	Let $b:[0,1]^3\rightarrow\mathbb{R}^3$ be a Lipschitz function with Lipschitz constant $K$. Let $x:[0,\infty)\rightarrow \mathbb{R}^3$ be the function that satisfies
		$$
		x(t)=x(0)+\int_0^tb(x(s))ds, \hspace{0.5 cm} \mbox{for $t\geq 0$}.
		$$  
	The goal is to compare $\mathbf{X}(t)$ with $x(t)$.

	Fix $T>0$, $\epsilon_0>0$, $L>0$, and let $\Delta=\epsilon_0 e^{-KT}/3$. Define the events
		\begin{align}
		\Omega_0&=\{|\mathbf{X}(0)-x(0)|\leq\Delta\},\nonumber\\
		\Omega_1&=\bigg\{\int_0^{T}|\beta(\mathbf{X}(t))-b(\mathbf{X}(t))|dt\leq \Delta\bigg\},\nonumber\\
		\Omega_2&=\bigg\{\int_0^{T}\alpha(\mathbf{X}(t))dt\leq LT\bigg\}.\nonumber
		\end{align}

	\begin{theorem} \label{DNthm}
	Under all the assumptions above,
		$$ 
		P\bigg(\sup_{0\leq t \leq T}|\mathbf{X}(t)-x(t)|>\epsilon\bigg)\leq\frac{4LT}{\Delta^2}+P\big(\Omega_0^c\cup\Omega_1^c\cup\Omega_2^c\big).
		$$
	\end{theorem}

	Now, we will apply this theorem to our process $((X_0(t),X_1(t),X_2(t),X_3(t)),t\geq 0)$. First, for $t\geq 0$, we define
		\begin{equation} \label{X(t)}
		\mathbf{X}(t)=(\tilde{X_1}(t),\tilde{X_2}(t),\tilde{X_3}(t)),
		\end{equation}
	and $S=\{(\xi_1,\xi_2,\xi_3)\in\{0,\frac{1}{N},...,\frac{N-1}{N},1\}^3:\xi_1+\xi_2+\xi_3\leq 1\}$. We are thinking of $\xi_1,\xi_2$ and $\xi_3$ as the fractions of type 1, 2 and 3 individuals in the population. For better understanding in the following formulas, we will define $\xi_0=1-\xi_1-\xi_2-\xi_3$, which represents the fraction of type 0 individuals in the population. Now, for each $\xi =(\xi_1, \xi_2, \xi_3)\in S$, we define 
		\begin{align*}
		f_0(\xi)&=(1-r)\xi_0+r(\xi_0+\xi_1)(\xi_0+\xi_2),\\
		f_1(\xi)&=(1-r)\xi_1+r(\xi_1+\xi_3)(\xi_0+\xi_1),\\
		f_2(\xi)&=(1-r)\xi_2+r(\xi_0+\xi_2)(\xi_2+\xi_3),\\
		f_3(\xi)&=(1-r)\xi_3+r(\xi_1+\xi_3)(\xi_2+\xi_3).
		\end{align*}
	Note that for each $i=0, 1, 2, 3$, the quantity $f_i(\xi)$ represents the probability that a new individual born is of type $i$. Next, for $\xi=(\xi_1,\xi_2,\xi_3)\in S$ and $\xi'=(\xi'_1,\xi'_2,\xi'_3)\in S$, the transition rate $q(\xi,\xi')$ is given by
		\begin{equation}\label{q}
  		q(\xi,\xi')=
			\begin{cases}
  			N\xi_0f_1(\xi)+\mu N\xi_0, &\mbox{ if $(\xi'_1,\xi'_2,\xi'_3)=(\xi_1+\frac{1}{N},\xi_2,\xi_3)$}\\
  			N\xi_0f_2(\xi)+\mu N\xi_0, &\mbox{ if $(\xi'_1,\xi'_2,\xi'_3)=(\xi_1,\xi_2+\frac{1}{N},\xi_3)$}\\
  			N\xi_0f_3(\xi), &\mbox{ if $(\xi'_1,\xi'_2,\xi'_3)=(\xi_1,\xi_2,\xi_3+\frac{1}{N})$}\\
  			N(1-s)\xi_1f_0(\xi), &\mbox{ if $(\xi'_1,\xi'_2,\xi'_3)=(\xi_1-\frac{1}{N},\xi_2,\xi_3)$}\\
  			N(1-s)\xi_1f_2(\xi), &\mbox{ if $(\xi'_1,\xi'_2,\xi'_3)=(\xi_1-\frac{1}{N},\xi_2+\frac{1}{N},\xi_3)$}\\
  			N(1-s)\xi_1f_3(\xi)+\mu N\xi_1, &\mbox{ if $(\xi'_1,\xi'_2,\xi'_3)=(\xi_1-\frac{1}{N},\xi_2,\xi_3+\frac{1}{N})$}\\
 			N(1-s)\xi_2f_0(\xi), &\mbox{ if $(\xi'_1,\xi'_2,\xi'_3)=(\xi_1,\xi_2-\frac{1}{N},\xi_3)$}\\
  			N(1-s)\xi_2f_1(\xi), &\mbox{ if $(\xi'_1,\xi'_2,\xi'_3)=(\xi_1+\frac{1}{N},\xi_2-\frac{1}{N},\xi_3)$}\\
  			N(1-s)\xi_2f_3(\xi)+\mu N\xi_2, &\mbox{ if $(\xi'_1,\xi'_2,\xi'_3)=(\xi_1,\xi_2-\frac{1}{N},\xi_3+\frac{1}{N})$}\\
  			N(1-2s)\xi_3f_0(\xi), &\mbox{ if $(\xi'_1,\xi'_2,\xi'_3)=(\xi_1,\xi_2,\xi_3-\frac{1}{N})$}\\
  			N(1-2s)\xi_3f_1(\xi), &\mbox{ if $(\xi'_1,\xi'_2,\xi'_3)=(\xi_1+\frac{1}{N},\xi_2,\xi_3-\frac{1}{N})$}\\
  			N(1-2s)\xi_3f_2(\xi), &\mbox{ if $(\xi'_1,\xi'_2,\xi'_3)=(\xi_1,\xi_2+\frac{1}{N},\xi_3-\frac{1}{N})$}\\
  			0, &\mbox{ otherwise}.
			\end{cases}
		\end{equation}
	The reasons behind the formulas for these rates are similar to the ones we used to obtain the birth and death rates in section \ref{rate}. Let us consider the first rate. It is the rate that the number of type 0 individuals decreases by 1 and the number of type 1 individuals increases by 1. There are two ways for this to occur: 1) a type 0 individual mutates to type 1, which occurs at total rate of $\mu N\xi_0$, and 2) a type 0 individual dies and is replaced by a type 1 individual. The total rate that a type 0 individual dies is $N\xi_0$, and the probability that the replacement is of type 1 is $f_1(\xi)$.

	We define the functions $\alpha$ and $\beta$ as in (\ref{alpha}) and (\ref{beta}). For $\xi,\xi'\in S$ such that $q(\xi,\xi')\neq 0$, we have $|\xi-\xi'|^2\leq2/N^2$, since it is equal to $1/N^2$ or  $2/N^2$. Because for each $i= 0,1, 2,3$ and $\xi \in S$,  we have $0\leq f_i(\xi)\leq 1$, and because $\mu \ll s \ll 1$, it follows that for sufficiently large $N$, for all $\xi, \xi' \in S$, we have $q(\xi, \xi')\leq 2N$. By the definition of $\alpha$ in (\ref{alpha}), for sufficiently large $N$, 
		\begin{equation}\label{alpha<}
		\alpha(\xi)\leq\frac{48}{N}.
		\end{equation}
	
	For each $\xi \in S$, we define
		$$
		\gamma_s(\xi)=\big(\xi_0\xi_3-\xi_1\xi_2\big) (1-s\xi_1-s\xi_2-2s\xi_3).
		$$  
	A tedious calculation gives
  		\begin{align} \label{beta=}
		\beta(\xi)&
			=s\begin{pmatrix}
        			(1-\xi_1-\xi_2-2\xi_3)\xi_1\\
           		(1-\xi_1-\xi_2-2\xi_3)\xi_2\\
           		(2-\xi_1-\xi_2-2\xi_3)\xi_3
             	\end{pmatrix}
           	+r\gamma_s(\xi)
            		\begin{pmatrix}
            		1\\
            		1\\
            		-1
           		 \end{pmatrix}  
            	+\mu\begin{pmatrix}
             	\xi_0-\xi_1\\
             	\xi_0-\xi_2\\
             	\xi_1+\xi_2
             	\end{pmatrix}.
		\end{align}
	Note that for $i=1,2,3$, the $i^{th}$ row  of $N\beta(\xi)$ is exactly the rate at which the number of type $i$ individuals increases by 1 minus the rate at which the number of type $i$ individuals decreases by 1.  

	Here, we define the functions $b:[0,1]^3\rightarrow \mathbb{R}^3$ and $\tilde b:[0,1]^3\rightarrow \mathbb{R}^3$ by
		\begin{equation}\label{b}
		b(x_1,x_2,x_3)=s\big((1-x_1-x_2-2x_3)x_1,(1-x_1-x_2-2x_3)x_2,(2-x_1-x_2-2x_3)x_3\big),
		\end{equation}
	and 
		\begin{equation} \label{b-}
		\tilde b(x_1,x_2,x_3)=b(x_1,x_2,x_3)/s.
		\end{equation}
		 Since all first partial derivatives of $\tilde{b}$ are bounded, the function $\tilde{b}$ is Lipschitz. Hence, $b$ is also Lipschitz with Lipschitz constant $ks$, where $k>0$ and $k$ does not depend on $N$.

	Now, we define a random variable $B$ such that on the event that $\tilde X_1(t_1)+\tilde X_2(t_1)>0$, we have
		\begin{equation}\label{B}
		B =\big(\tilde X_1(t_1)+\tilde X_2(t_1)\big)^{-1}-1. 
		\end{equation}
	The value of $B$ on the event that $\tilde X_1(t_1)+\tilde X_2(t_1)=0$ is not of interest, as we will work only on the event $A_{(1)}$ when $N$ is sufficiently large. By Proposition \ref{@t1}, we know that $\tilde X_1(t_1)+\tilde X_2(t_1)>0$ on the event $A_{(1)}$. Next, for $t\geq t_1$, we define	
		\begin{equation} \label{f}
			f(t)=\frac{1}{1+Be^{-s(t-t_1)}},
		\end{equation}
	and for $t\geq t_1$, we let
		\begin{equation} \label{x}
		x(t)=\big(x_1(t),x_2(t),x_3(t)\big)=\bigg(\bigg(\frac{\tilde X_1(t_1)}{\tilde X_1(t_1)+\tilde X_2(t_1)}\bigg)f(t),\bigg(\frac{\tilde X_2(t_1)}{\tilde X_1(t_1)+\tilde X_2(t_1)}\bigg)f(t),0\bigg).
		\end{equation}
	Note that for $i=1,2$, we have $x_i(t_1)=\tilde X_i(t_1)$, and for all $t\geq t_1$, we have $x_1(t)+x_2(t)=f(t)$. From (\ref{f}), for $t\geq t_1$,
		$$
		\frac{d}{dt}f(t)=\frac{sBe^{-s(t-t_1)}}{(1+Be^{-s(t-t_1)})^2}=sBe^{-s(t-t_1)}(f(t))^2,
		$$
	and it follows that
		$$
		\frac{d}{dt}x(t)=sBe^{-s(t-t_1)}(f(t))^2\bigg(\frac{\tilde X_1(t_1)}{\tilde X_1(t_1)+\tilde X_2(t_1)},\frac{\tilde X_2(t_1)}{\tilde X_1(t_1)+\tilde X_2(t_1)},0\bigg).
		$$
	From (\ref{b}), (\ref{x}), (\ref{f}), and the fact that $x_1(t)+x_2(t)=f(t)$ for all $t\geq t_1$, we have that for $t\geq t_1$,
		\begin{align*}
		b(x(t))&=s\big((1-f(t))x_1(t),(1-f(t))x_2(t),0\big)\\
		&=sBe^{-s(t-t_1)}(f(t))^2\bigg(\frac{\tilde X_1(t_1)}{\tilde X_1(t_1)+\tilde X_2(t_1)},\frac{\tilde X_2(t_1)}{\tilde X_1(t_1)+\tilde X_2(t_1)},0\bigg).
		\end{align*}
	Therefore, for $t\geq t_1$, we have $\frac{d}{dt}x(t)=b(x(t))$, and
		$$
		x(t)=x(t_1)+\int_{t_1}^tb(x(s))ds.
		$$

	We pick the constant 
		\begin{equation}\label{C2}
		C_2=-C_1+\ln\bigg(\frac{e^{C_1}}{2(1+\delta^2)}-1\bigg)+\ln\bigg(\frac{1}{\delta^2}-1\bigg),
		\end{equation}
	and we define
		\begin{equation}
		t_2=\frac{1}{s}\ln\Big(\frac{s}{\mu}\Big)+\frac{C_2}{s}.\label{t2}
		\end{equation} 

	We will use Theorem \ref{DNthm} to show that with probability almost 1, both $X_1(t)$ and $X_2(t)$ are close to $x_1(t)N$ and $x_2(t)N$ for $t\in[t_1,t_2]$. We define the event
	\begin{equation}\label{A12}
	A_{12}=\bigg\{\sup_{t \in [t_1,t_2]}|X_i(t)-x_i(t)N|\leq \Big(\frac{\delta^4}{4}\Big)N \hspace{0.1 cm}\textup{ for $i=1,2$}\bigg\}.
	\end{equation}
	
	\begin{lemma} \label{DNt12}
	For sufficiently large $N$, we have $P(A_{12}^c|\hspace{0.1 cm}\mathcal{F}_{t_1})\leq \epsilon$ on the event $A_{(1)}$.
	\end{lemma}
	
	\begin{proof}
	Let $\Delta=\delta^4e^{-k(C_2+C_1)}/12$. We will first prove that for sufficiently large $N$, on the event $A_{(1)}$, 
		$$
		P\bigg(\sup_{t \in [t_1,t_2]}|\mathbf{X}(t)-x(t)|>\frac{\delta^4}{4}\hspace{0.1 cm} \bigg|\hspace{0.1 cm} \mathcal{F}_{t_1}\bigg)
		\leq\epsilon.
		$$
	By (\ref{beta=}) and (\ref{b}), we have
		\begin{align*}
		\beta(\mathbf{X}(t))-b(\mathbf{X}(t))
			&= r\gamma_s\big((\tilde X_1(t),\tilde X_2(t),\tilde X_3(t))\big)
            		\begin{pmatrix}
            		1\\
            		1\\
            		-1
            		\end{pmatrix}
		+\mu\begin{pmatrix}
  				1-2\tilde{X_1}(t)-\tilde{X_2}(t)-\tilde{X_3}(t)\\
  				1-\tilde{X_1}(t)-2\tilde{X_2}(t)-\tilde{X_3}(t)\\
  				\tilde{X_1}(t)+\tilde{X_2}(t)\\
  			\end{pmatrix}.
		\end{align*}
	Because $\tilde{X_i}(t)\in[0,1]$ for all $i=1,2,3$,  and $t\geq 0$, we have
		$$
		|\beta(\mathbf{X}(t))-b(\mathbf{X}(t))|\leq Dr+D'\mu,
		$$
	for some positive constants $D$ and $D'$. Thus, 
		\begin{equation} \label{I1}
		\int_{t_1}^{t_2}|\beta(\mathbf{X}(t)-b(\mathbf{X}(t))|dt \leq(Dr+D'\mu)(t_2-t_1)=(C_2+C_1)\Big(D\Big(\frac{r}{s}\Big)+D'\Big(\frac{\mu}{s}\Big)\Big).
		\end{equation}
	
	In the recombination dominating case, since $r\ll s$, $\mu \ll s$, $1\ll N\mu$, and $r\ln_+(Nr)\ll s$, if $N$ is sufficiently large, then 
		\begin{equation} \label{-6.0}
		\bigg(\frac{192(C_2+C_1)}{\Delta^2}\bigg)\bigg(\frac{1}{Ns}\bigg)\leq \epsilon,
		\end{equation}
		\begin{equation}\label{*3}
		(C_2+C_1)\Big(D\Big(\frac{r}{s}\Big)+D'\Big(\frac{\mu}{s}\Big)\Big)\leq \Delta,
		\end{equation}
	and 
		\begin{equation}\label{-6.1}
		\frac{K_{1r}^+r\ln(Nr)}{s}\leq \Delta.
		\end{equation}
	In the mutation dominating case, since $r\ll s$, $\mu \ll s$, an $N\mu^2\ll s$ , if $N$ is sufficiently large, then (\ref{*3}) holds and
		\begin{equation}\label{-6.2}
		\frac{K_{1m}^+N\mu^2}{s}\leq \Delta.
		\end{equation}
	In this proof, we assume that $N$ is large enough so that in the recombination dominating case, (\ref{-6.0}), (\ref{*3}) and (\ref{-6.1}) hold, and  in the mutation dominating case, (\ref{-6.0}), (\ref{*3}) and (\ref{-6.2}) hold.
	
	Now, let us consider the process $(\mathbf{X}(t),t\geq 0)$. By Markov property of the process, if we condition on $\mathcal{F}_{t_1}$, the process after time $t_1$ behaves as if we start the whole process again with $\mathbf{X}(t_1)$ as the initial condition. Now, let us fix the value of $\mathbf{X}(t_1)=(\xi_1,\xi_2,\xi_3)$, and consider the process starting at time $t_1$ with this initial condition. Note that by starting the process from this fixed start point, the function $f$ and $x$ defined in (\ref{f}) and (\ref{x}) are no longer random, which allows us to use Theorem \ref{DNthm}.
	
	We define $T=t_2-t_1$, and note that $\Delta=\delta^4 e^{-k(C_2+C_1)}/12=(\delta^4/4)\cdot e^{-(ks)T}/3$, which is in the form required in order to use Theorem \ref{DNthm}. We let $L=48/N$ and define the events
		\begin{align*}
		\Omega_0&=\{|\mathbf{X}(t_1)-x(t_1)|\leq\Delta\}\\
		\Omega_1&=\bigg\{\int_{t_1}^{t_2}|\beta(\mathbf{X}(t))-b(\mathbf{X}(t))|dt\leq \Delta\bigg\}\\
		\Omega_2&=\bigg\{\int_{t_1}^{t_2}\alpha(\mathbf{X}(t))dt\leq LT\bigg\}.
		\end{align*}
 	
 	First, we consider $\Omega_0$. In the recombination dominating case, if $\mathbf{X}(t_1)=(\xi_1,\xi_2,\xi_3)$ satisfies (\ref{Xit1}) and (\ref{X3t1rec}), then by (\ref{-6.1}), we have 
		$$
		|\mathbf{X}(t_1)-x(t_1)|
		\leq |\tilde{X_1}(t_1)-x_1(t_1)|+|\tilde{X_2}(t_1)-x_2(t_1)|+|\tilde{X_3}(t_1)-x_3(t_1)| 
		\leq 0+0+\frac{K_{1r}^+r\ln(Nr)}{s}
		\leq \Delta.
		$$
	Similarly, in the mutation dominating case, if $\mathbf{X}(t_1)=(\xi_1,\xi_2,\xi_3)$ satisfies (\ref{Xit1}) and (\ref{X3t1mut}), then by (\ref{-6.2}), we have
		$$
		|\mathbf{X}(t_1)-x(t_1)|
		\leq |\tilde{X_1}(t_1)-x_1(t_1)|+|\tilde{X_2}(t_1)-x_2(t_1)|+|\tilde{X_3}(t_1)-x_3(t_1)| 
		\leq 0+0+\frac{K_{1m}^+N^2\mu^2}{s}
		\leq \Delta.
		$$ 
	
	Next, because of (\ref{I1}) and (\ref{*3}), we have that $\Omega_1^c=\emptyset$. Lastly, by (\ref{alpha<}), it follows that
		$$
		\int_{t_1}^{t_2}\alpha(\mathbf{X}(t))dt\leq\Big(\frac{48}{N}\Big)(t_2-t_1)=LT.
		$$
	So, $\Omega_2^c=\emptyset$.

	Therefore, if $\mathbf{X}(t_1)=(\xi_1,\xi_2,\xi_3)$ satisfies (\ref{Xit1}) and (\ref{X3t1rec}) in the recombination dominating case, or satisfies (\ref{Xit1}) and (\ref{X3t1mut}) in the mutation dominating case, by Theorem \ref{DNthm} and (\ref{-6.0}), we have that
		$$
		P\bigg(\sup_{t \in [t_1,t_2]}|\mathbf{X}(t)-x(t)|>\frac{\delta^4}{4}\hspace{0.1 cm} \bigg|\hspace{0.1 cm} \mathbf{X}(t_1)=(\xi_1,\xi_2,\xi_3)\bigg)
		\leq\frac{4AT}{\Delta^2}+0=\bigg(\frac{192(C_2+C_1)}{\Delta^2}\bigg)\bigg(\frac{1}{Ns}\bigg)\leq \epsilon.
		$$
	Note that the upper bound does not depend on the value of $(\xi_1,\xi_2,\xi_3)$. By Proposition \ref{@t1}, on the event $A_{(1)}$, we know that $\mathbf{X}(t_1)=(\xi_1,\xi_2,\xi_3)$ satisfies (\ref{Xit1}) and (\ref{X3t1rec}) in the recombination dominating case, and satisfies (\ref{Xit1}) and (\ref{X3t1mut}) in the mutation dominating case for sufficiently large $N$. Using the Markov property of the process, we have that on the event $A_{(1)}$,
		$$
		P\bigg(\sup_{t \in [t_1,t_2]}|\mathbf{X}(t)-x(t)|>\frac{\delta^4}{4}\hspace{0.1 cm}\bigg |\hspace{0.1 cm}\mathcal{F}_{t_1}\bigg)
		\leq\epsilon.
		$$
	Thus, from the definition of the event $A_{12}$ in (\ref{A12}), on the event $A_{(1)}$,
		$$
		P(A_{12}^c|\hspace{0.1 cm}\mathcal{F}_{t_1})
		\leq P\bigg(\sup_{t \in [t_1,t_2]}|\mathbf{X}(t)-x(t)|>\frac{\delta^4}{4}\hspace{0.1 cm}\bigg |\hspace{0.1 cm}\mathcal{F}_{t_1}\bigg)
		\leq \epsilon,
		$$
	which completes the proof.
	\end{proof}

\subsection{Results on type 3 individuals}

	We will now show that for sufficiently large $N$, with probability close to 1, $X_3(t_2)$ has the same order as $(Nr\ln(Nr))/s$ in the recombination dominating case, and has the same order as $(N^2\mu^2)/s$ in the mutation dominating case. The proof mainly has two parts. In the first part, we will show that $X_3^{[t_1]}(t_2)$, which was defined to be the number of type 3 individuals at time $t$ that descend from the type 3 individuals at time $t_1$, has order $(Nr\ln(Nr))/s$ in the recombination dominating case, and $(N^2\mu^2)/s$ in the mutation dominating case. In the second part, we show that $X_{3m}^{(t_1,t_2]}(t_2)$ and $X_{3r}^{(t_1,t_2]}(t_2)$ are much smaller than $X_3^{[t_1]}(t_2)$.

	\begin{lemma} \label{EX3[t1]}
	For sufficiently large $N$, for all $t\geq t_1$
		$$
		E\Big[X_3^{[t_1]}(t)\hspace{0.1 cm}\Big |\hspace{0.1 cm}\mathcal{F}_{t_1}\Big]\leq e^{2s(t-t_1)}X_3(t_1).
		$$
	\end{lemma} 
	
	\begin{proof}
	From (\ref{Zi}) and Proposition \ref{ZiMar}, we have that for $t\geq t_1$,
		$$
		X_3(t_1)=Z_3^{[t_1]}(t_1)=E\Big[Z_3^{[t_1]}(t)\hspace{0.1 cm}\Big |\hspace{0.1 cm}\mathcal{F}_{t_1}\Big]=E\Big[e^{-\int_{t_1}^tG_3(v)dv}X_3^{[t_1]}(t)\hspace{0.1 cm}\Big |\hspace{0.1 cm}\mathcal{F}_{t_1}\Big].
		$$
	Because of Lemma \ref{G<}, for sufficiently large $N$,
		$$
		E\Big[e^{-\int_{t_1}^tG_3(v)dv}X_3^{[t_1]}(t)\hspace{0.1 cm}\Big |\hspace{0.1 cm}\mathcal{F}_{t_1}\Big]
		\geq E\Big[e^{-\int_{t_1}^t2sdv}X_3^{[t_1]}(t)\hspace{0.1 cm}\Big |\hspace{0.1 cm}\mathcal{F}_{t_1}\Big]
		=e^{-2s(t-t_1)}E\Big[X_3^{[t_1]}(t)\hspace{0.1 cm}\Big |\hspace{0.1 cm}\mathcal{F}_{t_1}\Big].
		$$
	Thus, for sufficiently large $N$,
		$$
		E\Big[X_3^{[t_1]}(t)\hspace{0.1 cm}\Big |\hspace{0.1 cm}\mathcal{F}_{t_1}\Big]\leq e^{2s(t-t_1)}X_3(t_1),
		$$
	which proves this lemma.
	\end{proof}

	\begin{lemma}\label{Z3[t1]}
	The following statements hold:
		\begin{enumerate}
		\item In the recombination dominating case, there is a positive constant $K_{0r}$,  such that for sufficiently large $N$, on the event $A_{(1)}$, we have
			$$
			P\bigg(\Big| Z_3^{[t_1]}(t_2)-X_3(t_1) \Big| \geq \sqrt{\frac{K_{0r}}{\epsilon}\cdot\frac{Nr\ln(Nr)}{s^2}}\hspace{0.1 cm}\bigg |\hspace{0.1 cm}\mathcal{F}_{t_1}\bigg)\leq \epsilon.
			$$
		\item In the mutation dominating case, there is a positive constant $K_{0m}$,  such that for sufficiently large $N$, on the event $A_{(1)}$, we have
			$$
			P\bigg(\Big| Z_3^{[t_1]}(t_2)-X_3(t_1) \Big| \geq \sqrt{\frac{K_{0m}}{\epsilon}}\cdot\frac{N\mu}{s}\hspace{0.1 cm}\bigg |\hspace{0.1 cm}\mathcal{F}_{t_1}\bigg)\leq \epsilon.
			$$
		\end{enumerate}
	\end{lemma}

	\begin{proof}
	First, consider the recombination dominating case. From (\ref{B3a}) and (\ref{D3a}), for all $t\geq 0$, we have that $B_3^{[t_1]}(t)\leq 1$ and $D_3^{[t_1]}(t)\leq 1$. Also, from (\ref{G3}) and the fact that $s\ll 1$, for sufficiently large $N$,  for all $t\geq 0$,
		\begin{equation}\label{G3>-r}
		G_3(t)\geq -r.
		\end{equation}
	
	By Proposition \ref{ZiMar}, (\ref{G3>-r}), and Lemma \ref{EX3[t1]},  for sufficiently large $N$, 
		\begin{align}
		\textup{Var}\Big(Z_3^{[t_1]}(t_2)\hspace{0.1 cm}\Big |\hspace{0.1 cm}\mathcal{F}_{t_1}\Big)
		&\leq E\bigg[\int_{t_1}^{t_2}e^{2r(u-t_1)}\cdot 2X_3^{[t_1]}(u)du\hspace{0.1 cm}\bigg |\hspace{0.1 cm}\mathcal{F}_{t_1}\bigg]\nonumber\\
		&\leq 2e^{2r(t_2-t_1)}\int_{t_1}^{t_2}E\Big[X_3^{[t_1]}(u)\hspace{0.1 cm}\Big |\hspace{0.1 cm}\mathcal{F}_{t_1}\Big]du\nonumber\\
		&\leq 2e^{2r(t_2-t_1)}\int_{t_1}^{t_2}e^{2s(u-t_1)}X_3(t_1)du\nonumber\\
		&=e^{2r(t_2-t_1)}\bigg(\frac{e^{2s(t_2-t_1)}-1}{s}\bigg)X_3(t_1).\nonumber
		\end{align}
	By Proposition \ref{@t1} and the definitions of $t_1$ and $t_2$ in (\ref{t1}) and (\ref{t2}), for sufficiently large $N$, on the event $A_{(1)}$,
		\begin{align}
		\textup{Var}\Big(Z_3^{[t_1]}(t_2)\hspace{0.1 cm}\Big |\hspace{0.1 cm}\mathcal{F}_{t_1}\Big)
		&\leq e^{2(C_2+C_1)\cdot\frac{r}{s}}\bigg(\frac{e^{2(C_2+C_1)}K_{1r}^+Nr\ln(Nr)}{s^2}\bigg)\label{*4}\\
		&\leq\frac{2e^{2(C_2+C_1)}K_{1r}^+Nr\ln(Nr)}{s^2}.\nonumber
		\end{align}
	We define 
	 	\begin{equation}\label{K0r}
	 	K_{0r}=2e^{2(C_2+C_1)}K_{1r}^+.
	 	\end{equation}
	 Since the process $\big(Z_3^{[t_1]}(t),t\geq 0)$ is a martingale, we have that $E[Z_3^{[t_1]}(t_2)|\mathcal{F}_{t_1}]=Z_3^{[t_1]}(t_1)=X_3(t_1)$. Hence, by Chebyshev's inequality, we have that for sufficiently large $N$, on the event $A_{(1)}$,  
		$$
		P\bigg(\Big| Z_3^{[t_1]}(t_2)-X_3(t_1) \Big| \geq \sqrt{\frac{K_{0r}}{\epsilon}\cdot\frac{Nr\ln(Nr)}{s^2}}\hspace{0.1 cm}\bigg |\hspace{0.1 cm}\mathcal{F}_{t_1}\bigg)\leq \epsilon.
		$$
	
	For the mutation dominating case, the proof is almost exactly the same. The only difference is the inequality (\ref{*4}), for which Proposition \ref{@t1} gives that
		$$
		\textup{Var}\Big(Z_3^{[t_1]}(t_2)\hspace{0.1 cm}\Big |\hspace{0.1 cm}\mathcal{F}_{t_1}\Big)
		\leq e^{2(C_2+C_1)\cdot\frac{r}{s}}\bigg(\frac{e^{2(C_2+C_1)}K_{1m}^+N^2\mu^2}{s^2}\bigg)\leq\frac{2e^{2(C_2+C_1)}K_{1m}^+N^2\mu^2}{s^2}.
		$$
	In this case, we pick
		\begin{align}\label{K0m}
		K_{0m}=2e^{2(C_2+C_1)}K_{1m}^+.
		\end{align}
	This completes the proof. 
	\end{proof}

	\begin{lemma}\label{X(t1t2]}
	There exist positive constants $K'_1$ and $K'_2$ such that for sufficiently large $N$, we have
		\begin{enumerate}
		\item $\displaystyle P\bigg(X_{3m}^{(t_1,t_2]}(t_2)\geq\frac{K'_1}{\epsilon}\cdot\frac{N\mu}{s}\hspace{0.1 cm}\Big |\hspace{0.1 cm}\mathcal{F}_{t_1}\bigg)\leq \epsilon.$
		\item $\displaystyle P\bigg(X_{3r}^{(t_1,t_2]}(t_2)\geq\frac{K'_2}{\epsilon}\cdot\frac{Nr}{s}\hspace{0.1 cm}\Big |\hspace{0.1 cm}\mathcal{F}_{t_1}\bigg)\leq \epsilon.$
		\end{enumerate}
	\end{lemma}

	\begin{proof} 
	We will first prove part 1. Let $U(t)$ and $V(t)$ be the numbers of times that the number of type $3m(t_1,t_2]$  individuals increases and decreases respectively during the time interval $[t_1,t]$. Then, for $t\geq t_1$, we define 
		\begin{align*}
		W_+(t)&=U(t)-\int_{t_1}^{t}\Big(M_3^{(t_1,t_2]}(u)+B^{(t_1,t_2]}_{3m}(u)X^{(t_1,t_2]}_{3m}(u)\Big)du,\\
		W_-(t)&=V(t)-\int_{t_1}^{t}D^{(t_1,t_2]}_{3m}(u)X^{(t_1,t_2]}_{3m}(u)du,\\
		W_m(t)&=W_+(t)-W_-(t).
		\end{align*}
	Then, both processes $(W_+(t), t\geq t_1)$ and $(W_-(t), t\geq t_1)$ are mean-zero martingales, and so is the process $(W_m(t), t\geq t_1)$. We also have that
		$$
		X^{(t_1,t_2]}_{3m}(t)=U(t)-V(t)=W_m(t)+\int_{t_1}^{t}\Big(M_3^{(t_1,t_2]}(u)+G_3(u)X^{(t_1,t_2]}_{3m}(u)\Big)du.
		$$
	Thus, from Lemma \ref{G<}, for sufficiently large $N$, if $t\in[t_1,t_2]$, then
		\begin{align*}
		E\Big[X^{(t_1,t_2]}_{3m}(t)\hspace{0.1 cm}\Big |\hspace{0.1 cm}\mathcal{F}_{t_1}\Big]
		&=0+E\bigg[\int_{t_1}^{t}\Big(M_3^{(t_1,t_2]}(u)+G_3(u)X^{(t_1,t_2]}_{3m}(u)\Big)du\hspace{0.1 cm}\bigg |\hspace{0.1 cm}\mathcal{F}_{t_1}\bigg]\\
		&\leq E\bigg[\int_{t_1}^{t}\Big(\mu(X_1(u)+X_2(u))+2sX^{(t_1,t_2]}_{3m}(u)\Big)du\hspace{0.1 cm}\bigg |\hspace{0.1 cm}\mathcal{F}_{t_1}\bigg]\\
		&\leq E\bigg [\int_{t_1}^{t}\Big(N\mu+2sX^{(t_1,t_2]}_{3m}(u)\Big)du \hspace{0.1 cm}\bigg |\hspace{0.1 cm}\mathcal{F}_{t_1}\bigg]\\
		&\leq N\mu(t_2-t_1)+\int_{t_1}^{t}2sE\Big[X^{(t_1,t_2]}_{3m}(u)\hspace{0.1 cm}\Big |\hspace{0.1 cm}\mathcal{F}_{t_1}\Big]du.
		\end{align*}
	Here, we define 
		\begin{equation}\label{K1'}
		K'_1=e^{2(C_2+C_1)}(C_2+C_1).
		\end{equation}
	From Gronwall's inequality, we have
		$$
		E\Big[X^{(t_1,t_2]}_{3m}(t_2)\hspace{0.1 cm}\Big |\hspace{0.1 cm}\mathcal{F}_{t_1}\Big]\leq N\mu(t_2-t_1)e^{2s(t_2-t_1)}=\frac{K'_1N\mu}{s},
		$$
	and by Markov's inequality, we have that
  		$$
  		P\bigg(X_{3m}^{(t_1,t_2]}(t_2)\geq\frac{K'_1}{\epsilon}\cdot\frac{N\mu}{s}\hspace{0.1 cm}\bigg |\hspace{0.1 cm}\mathcal{F}_{t_1}\bigg)\leq \epsilon.
  		$$

	Now, we will prove part 2. The proof is similar to the the proof for part 1. First, we have that there is a mean-zero martingale $(W_r(t),t\geq t_1)$ such that
		$$
		X^{(t_1,t_2]}_{3r}(t)=W_r(t)+\int_{t_1}^{t}\Big(R_3^{(t_1,t_2]}(u)+G^{(t_1,t_2]}_{3r}(u)X^{(t_1,t_2]}_{3r}(u)\Big)du,
		$$
	for all $t\geq t_1$. From (\ref{R3ab}), Lemma \ref{G<} and $r\ll s$, for sufficiently large $N$, and for $t\in[t_1,t_2]$,
		\begin{align*}
		E\Big[X^{(t_1,t_2]}_{3r}(t)\hspace{0.1 cm}\Big |\hspace{0.1 cm}\mathcal{F}_{t_1}\Big]
		&=0+E\bigg[\int_{t_1}^{t}\Big(R_3^{(t_1,t_2]}(u)+G^{(t_1,t_2]}_{3r}X^{(t_1,t_2]}_{3r}(u)\Big)du\hspace{0.1 cm}\bigg |\hspace{0.1 cm}\mathcal{F}_{t_1}\bigg]\\
		&\leq E\Big[ \int_{t_1}^{t}\Big(Nr\tilde X_1(u)\tilde X_2(u)+(2s+r)X^{(t_1,t_2]}_{3r}(u)\Big)du\hspace{0.1 cm}\Big |\hspace{0.1 cm}\mathcal{F}_{t_1}\Big]\\
		&\leq E\bigg[\int_{t_1}^{t}\Big(Nr+(2s+r)X^{(t_1,t_2]}_{3r}(u)\Big)du\hspace{0.1 cm}\bigg |\hspace{0.1 cm}\mathcal{F}_{t_1}\bigg ]\\
		&= Nr(t-t_1)+\int_{t_1}^{t}(2s+r)E\Big[X^{(t_1,t_2]}_{3m}(u)\hspace{0.1 cm}\Big |\hspace{0.1 cm}\mathcal{F}_{t_1}\Big]du\\
		&\leq Nr(t_2-t_1)+\int_{t_1}^{t}3sE\Big[X^{(t_1,t_2]}_{3m}(u)\hspace{0.1 cm}\Big |\hspace{0.1 cm}\mathcal{F}_{t_1}\Big]du.
		\end{align*}
	We define 
		\begin{equation}\label{K2'}
		K'_2=e^{3(C_2+C_1)}(C_2+C_1).
		\end{equation}
	From Gronwall's inequality, we have
		$$
		E\Big[X^{(t_1,t_2]}_{3r}(t_2)\hspace{0.1 cm}\Big |\hspace{0.1 cm}\mathcal{F}_{t_1}\Big]\leq Nr(t_2-t_1)e^{3s(t_2-t_1)}=\frac{K'_2Nr}{s},
		$$
	and the result follows from Markov's inequality.
\end{proof}

	Recall the constants $K_{0r}, K_{0m}, K_1'$ and $K_2'$ defined in (\ref{K0r}), (\ref{K0m}), (\ref{K1'}) and (\ref{K2'}). Now, we define the following events in both cases:
		\begin{align}
		A_{13}&=\bigg\{X_{3m}^{(t_1,t_2]}(t_2)<\frac{K'_1}{\epsilon_1}\cdot\frac{N\mu}{s}\bigg\}\label{A13}\\
		A_{14}&=\bigg\{X_{3r}^{(t_1,t_2]}(t_2)<\frac{K'_2}{\epsilon_1}\cdot\frac{Nr}{s}\bigg\}\label{A14}
		\end{align}
	In the recombination dominating case, we define
		\begin{equation}
		A_{15}=\Bigg\{\Big| Z_3^{[t_1]}(t_2)-X_3(t_1) \Big| < \sqrt{\frac{K_{0r}}{\epsilon}\cdot\frac{Nr\ln(Nr)}{s^2}}\Bigg\},\label{A15r}
		\end{equation}
	while in the mutation dominating case, we define
		\begin{equation}
		A_{15}=\Bigg\{\Big| Z_3^{[t_1]}(t_2)-X_3(t_1) \Big| < \sqrt{\frac{K_{0m}}{\epsilon}}\cdot\frac{N\mu}{s}\Bigg\}.\label{A15m}
		\end{equation}
	Lastly, in both cases, we define
		\begin{equation}
		A_{(2)}=A_{(1)}\cap\bigg(\bigcap_{i=12}^{15}A_i\bigg).\label{A(2)}
		\end{equation}
		
	\begin{lemma}\label{ft2}
	On the event $A_{(1)}$, for sufficiently large $N$, and for $i=1,2$, we have
		\begin{equation}\label{-6.5}
		\frac{1-\delta^2}{2}\leq\frac{\tilde X_i(t_1)}{\tilde X_1(t_1)+\tilde X_2(t_1)}\leq\frac{1+\delta^2}{2}
		\end{equation}
	and
		$$
		\frac{1-\delta^2}{1+\delta^2} \leq f(t_2)\leq 1-\delta^2.
		$$ 
	\end{lemma}
	
	\begin{proof}
	First note that if $c>0$, then the function $g(x)=x/(x+c)$ is increasing on the interval $(0,\infty)$. Then, from Proposition \ref{@t1}, on the event $A_{(1)}$, for sufficiently large $N$, 
		$$
		\frac{\tilde X_1(t_1)}{\tilde X_1(t_1)+\tilde X_2(t_1)}
		\geq\frac{(1-\delta^2)e^{-C_1}}{(1-\delta^2)e^{-C_1}+\tilde X_2(t_1)}\geq\frac{(1-\delta^2)e^{-C_1}}{(1-\delta^2)e^{-C_1}+(1+\delta^2)e^{-C_1}}=\frac{1-\delta^2}{2},
		$$
and
		$$
		\frac{\tilde X_1(t_1)}{\tilde X_1(t_1)+\tilde X_2(t_1)}
		\leq\frac{(1+\delta^2)e^{-C_1}}{(1+\delta^2)e^{-C_1}+\tilde X_2(t_1)}\leq\frac{(1+\delta^2)e^{-C_1}}{(1+\delta^2)e^{-C_1}+(1-\delta^2)e^{-C_1}}=\frac{1+\delta^2}{2}.
		$$
	By the same argument, we get the same bounds for $\tilde X_2(t_1)/(\tilde X_1(t_1)+\tilde X_2(t_1))$.
	
	Now, recall the definitions of $B$ and $f$ in (\ref{B}) and (\ref{f}). By Proposition \ref{@t1}, (\ref{-6.5}) and the definitions of $t_1, t_2$ and $C_2$ in (\ref{t1}), (\ref{C2}) and (\ref{t2}), for sufficiently large $N$, on the event $A_{(1)}$, 
		\begin{align*}
		f(t_2)
		&=\frac{1}{1+\Big(\frac{1}{\tilde X_1(t_1)+\tilde X_2(t_1)}-1\Big)e^{-s(t_2-t_1)}}\\
		&\leq \frac{1}{1+\Big(\frac{e^{C_1}}{2(1+\delta^2)}-1\Big)e^{-(C_2+C_1)}}\\
		&=1-\delta^2,
		\end{align*}
	and
		\begin{align*}
		f(t_2)
		&=\frac{1}{1+\Big(\frac{1}{\tilde X_1(t_1)+\tilde X_2(t_1)}-1\Big)e^{-s(t_2-t_1)}}\\
		&\geq \frac{1}{1+\Big(\frac{e^{C_1}}{2(1-\delta^2)}-1\Big)e^{-(C_2+C_1)}}\\
		&=\frac{1}{1+\Big(\frac{e^{C_1}}{2(1-\delta^2)}-1\Big)\Big(\frac{e^{C_1}}{2(1+\delta^2)}-1\Big)^{-1}\Big(\frac{\delta^2}{1-\delta^2}\Big)}.
		\end{align*}
	From the way we define $\delta$ and $C_1$ in (\ref{delta}) and  (\ref{C1}), we have $e^{C_1}>8/\delta^2>32>15/2$, and
		$$
		\bigg(\frac{e^{C_1}}{2(1-\delta^2)}-1\bigg)\bigg(\frac{e^{C_1}}{2(1+\delta^2)}-1\bigg)^{-1}
		\leq\bigg(\frac{2e^{C_1}}{3}-1\bigg)\bigg(\frac{2e^{C_1}}{5}-1\bigg)^{-1}
		<2.
		$$
	Thus, we have
		$$
		f(t_2)\geq\frac{1}{1+\frac{2\delta^2}{1-\delta^2}}=\frac{1-\delta^2}{1+\delta^2}.
		$$
	This completes the proof of this lemma.
	\end{proof}

\subsection{The proof of Proposition \ref{@t2}}
	\begin{proof}
	Recall the definition of $A_{(2)}$ in (\ref{A(2)}). From Proposition \ref{DNt12}, Lemma \ref{Z3[t1]} and Lemma \ref{X(t1t2]}, for sufficiently large $N$, on the event $A_{(1)}$
		$$
		P\bigg(\bigcap_{i=12}^{15}A_i\hspace{0.1 cm}\bigg|\hspace{0.1 cm}\mathcal{F}_{t_1}\bigg)\geq 1-4\epsilon.		
		$$
	Thus, from Proposition \ref{@t1}, we have
		$$
		P(A_{(2)})=P\bigg(A_{(1)}\cap \bigg(\bigcap_{i=12}^{15}A_i\bigg)\bigg)\geq(1-4\epsilon)-P(A_{(1)}^c)\geq 1-21\epsilon.
		$$
	
	From now on,  we will work on the event $A_{(2)}$. By the definition of the event $A_{12}$ in (\ref{A12}), the definition of the function $x$ in (\ref{x}), and Lemma \ref{ft2}, for sufficiently large $N$, on the event $A_{(2)}$,
		\begin{align*}
		X_1(t_2)
		&\leq x_1(t_2)N+\Big(\frac{\delta^4}{4}\Big)N\\
		&=\Big(\frac{\tilde X_1(t_1)}{\tilde X_1(t_1)+\tilde X_2(t_1)}\Big)f(t_2)N+\Big(\frac{\delta^4}{4}\Big)N\\
		&\leq \Big(\frac{1+\delta^2}{2}\Big)\cdot (1-\delta^2)N+\Big(\frac{\delta^4}{4}\Big)N\\
		&=\Big(\frac{1}{2}-\frac{\delta^4}{4}\Big)N,
		\end{align*}
	and
		\begin{align*}
		X_1(t_2)
		&\geq x_1(t_2)N-\Big(\frac{\delta^4}{4}\Big)N\\
		&=\Big(\frac{\tilde X_1(t_1)}{\tilde X_1(t_1)+\tilde X_2(t_1)}\Big)f(t_2)N-\Big(\frac{\delta^4}{4}\Big)N\\
		&\geq \Big(\frac{1-\delta^2}{2}\Big)\cdot \Big(\frac{1-\delta^2}{1+\delta^2}\Big)N-\Big(\frac{\delta^4}{4}\Big)N\\
		&=\Big(\frac{1}{2}-\frac{3\delta^2}{2}+\frac{2\delta^4}{1+\delta^2}-\frac{\delta^4}{4}\Big)N\\
		&>\Big(\frac{1}{2}-\frac{3\delta^2}{2}\Big)N.
		\end{align*}
	Both the upper and lower bounds for $X_2(t_2)$ follow from the same argument.

	Now, we prove statement 2. Assume that we are in the recombination dominating case. By the definition of $Z_3^{[t_1]}(t)$ in (\ref{Zi}), the definition of $A_{15}$ in (\ref{A15r}), the inequality (\ref{G3>-r}) and Proposition \ref{@t1}, on the event $A_{(2)}$,
		\begin{align*}
		X_3(t_2) &\geq X_3^{[t_1]}(t_2)\\
		&=Z_3^{[t_1]}(t_2)e^{\int_{t_1}^{t_2}G_3(v)dv}\\
		&\geq \bigg(X_3(t_1)- \sqrt{\frac{K_{0r}}{\epsilon}\cdot\frac{Nr\ln(Nr)}{s^2}}\bigg)e^{-r(t_2-t_1)}\\
		&\geq\bigg(\frac{K_{1r}^-Nr\ln(Nr)}{s}- \sqrt{\frac{K_{0r}}{\epsilon}\cdot\frac{Nr\ln(Nr)}{s^2}}\bigg)e^{-(C_2+C_1)\cdot\frac{r}{s}}\\
		&=\frac{Nr\ln(Nr)}{s}\cdot e^{-(C_2+C_1)\cdot\frac{r}{s}}\bigg(K_{1r}^- - \sqrt{\frac{K_{0r}}{\epsilon}\cdot\frac{1}{Nr\ln(Nr)}}\bigg).
		\end{align*}
		We define 
		$$
		K_{2r}^-=K_{1r}^-/2.
		$$
		Because $1\ll Nr$ and $r\ll s$, for sufficiently large $N$,
		$$
		X_3(t_2)\geq\frac{K_{2r}^-Nr\ln(Nr)}{s}.
		$$

	By the definitions of $A_{13}, A_{14}$ and $A_{15}$ in (\ref{A13}), (\ref{A14}) and (\ref{A15r}), and by Proposition \ref{@t1}, we have that for sufficiently large $N$, on the event $A_{(2)}$,
		\begin{align*}
		X_3(t_2)&=X_3^{[t_1]}(t_2)+X_{3m}^{(t_1,t_2]}(t_2)+X_{3r}^{(t_1,t_2]}(t_2)\\
		&=Z_3^{[t_1]}(t_2)e^{\int_{t_1}^{t_2}G_3(v)dv}+X_{3m}^{(t_1,t_2]}(t_2)+X_{3r}^{(t_1,t_2]}(t_2)\\
		&\leq \bigg(X_3(t_1)+ \sqrt{\frac{K_{0r}}{\epsilon}\cdot\frac{Nr\ln(Nr)}{s^2}}\bigg)e^{2s(t_2-t_1)}+\frac{K'_1}{\epsilon}\cdot\frac{N\mu}{s}+\frac{K'_2}{\epsilon}\cdot\frac{Nr}{s}\\
		&\leq \bigg(\frac{K_{1r}^+Nr\ln(Nr)}{s}+ \sqrt{\frac{K_{0r}}{\epsilon}\cdot\frac{Nr\ln(Nr)}{s^2}}\bigg)e^{2(C_2+C_1)}+\frac{K'_1}{\epsilon}\cdot\frac{N\mu}{s}+\frac{K'_2}{\epsilon}\cdot\frac{Nr}{s}\\
		&=\frac{Nr\ln(Nr)}{s}\cdot\Bigg(\bigg(K_{1r}^+ + \sqrt{\frac{K_{0r}}{\epsilon}\cdot\frac{1}{Nr\ln(Nr)}}\bigg)e^{2(C_2+C_1)}+\frac{K'_1}{\epsilon}\cdot\frac{\mu}{r\ln(Nr)}+\frac{K'_2}{\epsilon}\cdot\frac{1}{\ln(Nr)}\Bigg).
		\end{align*}

	We define the constant
		\begin{equation}\label{K2r+}
		K_{2r}^+=2K_{1r}^+e^{2(C_2+C_1)}.
		\end{equation}
	Because $1\ll N\mu$ and $\mu \ll r$, for sufficiently large $N$,
		$$
		X_3(t_2)\leq\frac{K_{2r}^+Nr\ln(Nr)}{s}.
		$$
	
	Lastly, consider the mutation dominating case, where we will prove statement 3. The proof is similar to the proof of part 3. First, by using (\ref{A15m}) instead of (\ref{A15r}), for sufficiently large $N$, on the event $A_{(2)}$,
		\begin{align*}
		X_3(t_2) &\geq X_3^{[t_1]}(t_2)\\
		&=Z_3^{[t_1]}(t_2)e^{\int_{t_1}^{t_2}G_3(v)dv}\\
		&\geq \bigg(X_3(t_1)- \sqrt{\frac{K_{0m}}{\epsilon}}\cdot\frac{N\mu}{s}\bigg)e^{-r(t_2-t_1)}\\
		&\geq\bigg(\frac{K_{1m}^-N^2\mu^2}{s}- \sqrt{\frac{K_{0m}}{\epsilon}}\cdot\frac{N\mu}{s}\bigg)e^{-(C_2+C_1)\cdot\frac{r}{s}}\\
		&=\frac{N^2\mu^2}{s}\cdot e^{-(C_2+C_1)\cdot\frac{r}{s}}\bigg(K_{1m}^- - \sqrt{\frac{K_{0m}}{\epsilon}}\cdot\frac{1}{N\mu}\bigg).
		\end{align*}
	We define 
		$$
		K_{2m}^-=K_{1m}^-/2.
		$$
	Since $1\ll N\mu$, for sufficiently large $N$, on the event $A_{(2)}$,
			$$
			X_3(t)\geq\frac{K_{2m}^-N^2\mu^2}{s}.
			$$  
	By the definitions of $A_{13}, A_{14}$ and $A_{15}$ in (\ref{A13}), (\ref{A14}) and (\ref{A15m}), and by Proposition \ref{@t1}, we have that for sufficiently large $N$, on the event $A_{(2)}$,
		\begin{align*}
		X_3(t_2)
		&\leq \bigg(\frac{K_{1m}^+N^2\mu^2}{s}+ \sqrt{\frac{K_{0m}}{\epsilon}}\cdot\frac{N\mu}{s}\bigg)e^{2s(t_2-t_1)}+\frac{K'_1}{\epsilon}\cdot\frac{N\mu}{s}+\frac{K'_2}{\epsilon}\cdot\frac{Nr}{s}\\
		&=\frac{N^2\mu^2}{s}\cdot\Bigg(\bigg(K_{1m}^+ + \sqrt{\frac{K_{0m}}{\epsilon}}\cdot\frac{1}{N\mu}\bigg)e^{2(C_2+C_1)}+\frac{K'_1}{\epsilon}\cdot\frac{1}{N\mu}+\frac{K'_2}{\epsilon}\cdot\frac{r}{N\mu^2}\Bigg).
		\end{align*}

	We define the constant
		\begin{equation}\label{K2m+}
		K_{2m}^+=2K_{1m}^+e^{2(C_2+C_1)}.
		\end{equation}
	Because $1\ll N\mu$ and $r\ll N\mu^2$, for sufficiently large $N$,
		$$
		X_3(t)\leq\frac{K_{2m}^+N^2\mu^2}{s}.
		$$
	This completes the proof.
	\end{proof}

\section{Phase 3 and the proof of Proposition \ref{@t3}}\label{phase3}
	In this phase, we will use martingales and submartingales to approximate the number of type 0 and type 3 individuals. The ideas of the proof are similar to those used in phase 1. At the end of this section, we will give a proof for Proposition \ref{@t3}.

	We define the constant
	\begin{equation}\label{C3}
	C_3=
	\begin{cases}
		\displaystyle{C_2-3-\ln\bigg(\frac{K_{2r}^+}{\delta^2}\bigg)} & \mbox{in the recombination dominating case} \\
		\displaystyle{C_2-3-\ln\bigg(\frac{K_{2m}^+}{\delta^2}\bigg)} & \mbox{in the mutation dominating case},
	\end{cases}
	\end{equation}
	where the constants $K_{2r}^+$ and $K_{2m}^+$ are defined in the equations (\ref{K2r+}) and (\ref{K2m+}). We define the time
	\begin{equation}\label{t3}
	t_3=
	\begin{cases}
		\displaystyle{\frac{1}{s}\ln\bigg(\frac{s^2}{\mu r\ln(Nr)}\bigg)+\frac{C_3}{s}} & \mbox{in the recombination dominating case} \\
		\displaystyle{\frac{1}{s}\ln\bigg(\frac{s^2}{N\mu^3}\bigg)+\frac{C_3}{s}} & \mbox{in the mutation dominating case}.
	\end{cases}
	\end{equation}
	Next, we define the following stopping times:
		\begin{align}
		T_4&=\inf\{t\geq t_2: X_1(t)+X_2(t)\leq (1-3\delta)N\},\label{T4}\\
		T_5&=\inf\Big\{t\geq t_2: s\int_{t_2}^t\tilde X_3(v)dv \geq 1\Big\},\label{T5}\\
		T_6&=\inf\big\{t\geq t_2: X_0(t)\geq \delta Ne^{-s(1-3\delta)(t-t_2)}\big\},\label{T6}\\
		T_{(3)}&=T_4\wedge T_5\wedge T_6.\label{T(3)}
	\end{align}
	In both cases we define
	\begin{align}
		A_{16}&=\{T_4>t_3\wedge T_{(3)}\},\label{A16}\\
		A_{17}&=\bigg\{X_0^{[t_2]}(t\wedge T_{(3)})<\frac{\delta}{2}\cdot Ne^{-s(1-3\delta)(t\wedge T_{(3)}-t_2)}, \hspace{0.2 cm}\textup{for all $t\geq t_2$}\bigg\},\label{A17}\\
		A_{18}&=\bigg\{s\int_{t_2}^{t_3\wedge T_{(3)}}\tilde X_3(v)dv<1\bigg\}.\label{A18}
	\end{align}
	In the recombination dominating case, we define the following events:
	\begin{align}
		A_{19}&=\bigg\{X_{0r}^{(t_2,t_3]}(t\wedge T_{(3)})\leq\frac{e^{3+(C_3-C_2)}}{\epsilon}\cdot\frac{N}{\ln(Nr)}\cdot e^{-s(1-3\delta)(t\wedge T_{(3)}-t_2)}, \hspace{0.2 cm}\textup{for all $t\in [t_2,t_3]$}\bigg\}\label{A19r}\\
		A_{20}&=\bigg\{X_{3m}^{(t_2,t_3]}(t_3\wedge T_{(3)}) < \bigg(\frac{\delta^2}{\epsilon K_{2r}^+}\bigg)\cdot\frac{N\mu}{r\ln(Nr)}\bigg\}\label{A20r}\\
		A_{21}&=\bigg\{X_{3r}^{(t_2,t_3]}(t_3\wedge T_{(3)}) \leq \bigg(\frac{\delta^2}{\epsilon K_{2r}^+}\bigg)\cdot\frac{N}{\ln(Nr)}\bigg\}\label{A21r}\\
		A_{22}&=\bigg\{\sup_{t\in[t_2,t_3]}\Big|Z_3^{[t_2]}(t\wedge T_{(3)})-X_3(t_2)\Big|<\sqrt{\frac{2e^4K_{2r}^+}{\epsilon}\cdot\frac{Nr\ln(Nr)}{s^2}}\bigg\}\label{A22r}
	\end{align}
	In contrast, in the mutation dominating case, we define
	\begin{align}
		A_{19}&=\bigg\{X_{0r}^{(t_2,t_3]}(t\wedge T_{(3)})\leq\frac{e^{3+(C_3-C_2)}}{\epsilon}\cdot\frac{r}{\mu^2}\cdot e^{-s(1-3\delta)(t\wedge T_{(3)}-t_2)}, \hspace{0.2 cm}\textup{for all $t\in [t_2,t_3]$}\bigg\},\nonumber\\
		A_{20}&=\bigg\{X_{3m}^{(t_2,t_3]}(t_3\wedge T_{(3)}) < \bigg(\frac{\delta^2}{\epsilon K_{2m}^+}\bigg)\cdot\frac{1}{\mu}\bigg\},\label{A20m}\\
		A_{21}&=\bigg\{X_{3r}^{(t_2,t_3]}(t_3\wedge T_{(3)}) \leq \bigg(\frac{\delta^2}{\epsilon K_{2m}^+}\bigg)\cdot\frac{r}{\mu^2}\bigg\}\label{A21m},\\
		A_{22}&=\bigg\{\sup_{t\in[t_2,t_3]}\Big|Z_3^{[t_2]}(t\wedge T_{(3)})-X_3(t_2)\Big|<\sqrt{\frac{2e^4K_{2m}^+}{\epsilon}}\cdot\frac{N\mu}{s}\bigg\}.\nonumber
	\end{align}
	Lastly, in both cases, we define
	\begin{equation}\label{A(3)}
	A_{(3)}=A_{(2)}\cap\bigg(\bigcap_{i=16}^{22}A_i\bigg).
	\end{equation}
	
	We will first give bounds on the growth rates of type 0 and type 3 populations.

	\begin{lemma}\label{Gt2-3}
	The following statements are true.
		\begin{enumerate}
		\item If $t\in[t_2,T_4)$, then $G_0(t)\leq -s(1-3\delta)$.
		\item If $t\in[t_2,T_4)$, then $-s(1+\tilde X_3(t))-r-2\mu \leq G_{0r}^{(t_2,t_3]}(t)\leq -s(1-3\delta)+r$.
		\item If $t\in[t_2,T_6)$, then $s(1-\tilde{X}_3(t))-r\leq G_3(t)\leq s\big(1+\delta e^{-s(1-3\delta)(t-t_2)}\big)$.
		\item If $t\in[t_2,T_6)$, then $s(1-\tilde{X}_3(t))-r\leq G^{(t_2,t_3]}_{3r}(t)\leq s\big(1+\delta e^{-s(1-3\delta)(t-t_2)}\big)+r$.
		\end{enumerate}
	\end{lemma}

	\begin{proof} 
	By the definition of $T_4$ in (\ref{T4}), if $t\in[t_2,T_4)$, then $\tilde X_1(t) +\tilde X_2(t)>1-3\delta$, and from (\ref{G0}), we have that $G_0(t)\leq -s(\tilde X_1(t) +\tilde X_2(t))<-s(1-3\delta)$. From (\ref{G0,0r}), if $t\in[t_2,T_4)$, then $G_{0r}^{(t_2,t_3]}(t)\leq -s(1-3\delta)+r$, and by using the fact that $\tilde X_1(u)+\tilde X_2(u) +\tilde X_3(u) \leq 1$ for all $u\geq 0$, we also have that $G_{0r}^{(t_2,t_3]}(t)\geq -s(1+\tilde X_3(t))-r-2\mu$.

	Now, from the definition of $T_6$ in (\ref{T6}), if $t\in[t_2,T_6)$, then the equation (\ref{G3}) implies that $G_3(t)\leq s(1+\tilde X_0(t)) < s\big(1+\delta e^{-s(1-3\delta)(t-t_2)}\big)$, and $G_3(t)\geq s(1-\tilde X_3(t))-r$. Part 4 follows directly from part 3 and (\ref{G3r}).
\end{proof}

\subsection{Results on type 0 individuals}
	\begin{lemma}\label{PA17}
	For sufficiently large $N$, on the event $A_{(2)}$, we have $P(A_{17}^c|\mathcal{F}_{t_2})\leq 6\delta$.
	\end{lemma}
	
	\begin{proof}
	First, from part 2 of Proposition \ref{@t2}, on the event $A_{(2)}$, we have that $ X_0(t_2)\leq N- X_1(t_2) - X_2(t_2) \leq 3\delta^2N$. From Proposition \ref{ZiMar}, the process $(Z_0^{[t_2]}(t\wedge T_{(3)}),t\geq t_2)$ is a martingale. Hence, by Lemma \ref{Gt2-3} and Doob's maximal inequality, for sufficiently large $N$, on the event $A_{(2)}$,
		\begin{align*}
		P(A_{17}^c|\mathcal{F}_{t_2})
		&=P\bigg(\sup_{t\geq t_2} X_0^{[t_2]}(t\wedge T_{(3)})e^{s(1-3\delta)(t\wedge T_{(3)}-t_2)}\geq\frac{\delta N}{2}\hspace{0.1 cm}\bigg|\hspace{0.1 cm}\mathcal{F}_{t_2}\bigg)\\
		&\leq P\bigg(\sup_{t\geq t_2} X_0^{[t_2]}(t\wedge T_{(3)})e^{-\int_{t_2}^{t\wedge T_{(3)}}G_0(v)dv}\geq\frac{\delta N}{2}\hspace{0.1 cm}\bigg|\hspace{0.1 cm}\mathcal{F}_{t_2}\bigg)\\
		&=P\Big(\sup_{t\geq t_2} Z_0^{[t_2]}(t\wedge T_{(3)})\geq \frac{\delta N}{2}\hspace{0.1 cm}\Big|\hspace{0.1 cm}\mathcal{F}_{t_2}\Big)\\
		&\leq\frac{E\Big[Z_0^{[t_2]}(t_2)\Big|\mathcal{F}_{t_2}\Big]}{\delta N/2}\\
		&=\frac{X_0(t_2)}{\delta N/2}\\
		&\leq 6\delta,
	\end{align*}
	which proves the lemma.
	\end{proof}

	\begin{lemma}\label{PA19}
	For sufficiently large $N$, we have $P(A_{19}^c|\mathcal{F}_{t_2})\leq \epsilon$.
	\end{lemma}

	\begin{proof}
	We will first prove this result in the recombination dominating case. By Proposition \ref{W}, the process $\big(W_{0r}^{(t_2,t_3]}(t\wedge T_{(3)}),t\geq t_2\big)$ is a submartingale. Also, note that from the definitions of $t_2$ and $t_3$ in (\ref{t2}) and (\ref{t3}), we have that 
		\begin{equation}\label{*3.1}
		t_3-t_2=\frac{1}{s}\ln\bigg(\frac{s}{r\ln(Nr)}\bigg)+\frac{C_3-C_2}{s}.
		\end{equation} 
	From Proposition \ref{W}, Lemma \ref{Gt2-3} part 2, (\ref{R0<}), and the definition of $T_5$ in (\ref{T5}), we have
		\begin{align}
		E\Big[W_{0r}^{(t_2,t_3]}(t_3\wedge T_{(3)})\Big|\mathcal{F}_{t_2}\Big]\nonumber
		&=E\bigg[\int_{t_2}^{t_3\wedge T_{(3)}}R_0^{(t_2,t_3]}(u)e^{-\int_{t_2}^uG_{0r}^{(t_2,t_3]}(v)dv}du\Big|\mathcal{F}_{t_2}\bigg]\nonumber\\
		&\leq E\bigg[\int_{t_2}^{t_3\wedge T_{(3)}}Nre^{\int_{t_2}^u s(1+\tilde X_3(v))+r+2\mu dv}du\Big|\mathcal{F}_{t_2}\bigg]\nonumber\\
		&\leq E\bigg[\int_{t_2}^{t_3\wedge T_{(3)}}Nre^{s(u-t_2) + \int_{t_2}^{T_5} s\tilde X_3(v)dv+(r+2\mu)(t_3-t_2)}du\Big|\mathcal{F}_{t_2}\bigg]\nonumber\\
		&\leq e^{1+(r+2\mu)(t_3-t_2)}Nr\int_{t_2}^{t_3}e^{s(u-t_2)}du\nonumber\\
		&\leq e^{1+(r+2\mu)(t_3-t_2)}\cdot e^{s(t_3-t_2)}\cdot \frac{Nr}{s}\nonumber\\
		&\leq e^{1+(r+2\mu)(t_3-t_2)}\cdot e^{C_3-C_2}\cdot \frac{N}{\ln(Nr)}.\label{-7.1}
		\end{align}
	Because $1\ll Nr$ and $r\ll s$, for sufficiently large $N$, 
		$$
		\frac{r}{s}\ln\bigg(\frac{s}{r\ln(Nr)}\bigg)\leq \frac{r}{s}\ln\bigg(\frac{s}{r}\bigg)\ll 1,
		$$
	and it follows that 
		\begin{equation}\label{r(t3-t2)}
		r(t_3-t_2)\ll 1.
		\end{equation}
	Also, since $\mu \ll r$, we have 
		\begin{equation}\label{u(t3-t2)}
		\mu(t_3-t_2)\ll 1.
		\end{equation} 
	Hence, from (\ref{-7.1}), for sufficiently large $N$, we have
		$$
		E\Big[W_{0r}^{(t_2,t_3]}(t_3\wedge T_{(3)})\Big|\mathcal{F}_{t_2}\Big]\leq e^{2+(C_3-C_2)}\cdot\frac{N}{\ln(Nr)}.
		$$
	Thus, from (\ref{r(t3-t2)}), Lemma \ref{Gt2-3} part 2 and Doob's maximal inequality, for sufficiently large $N$, 
		\begin{align*}
		P(A_{19}^c|\mathcal{F}_{t_2})
		&= P\bigg(\sup_{t\in[t_2,t_3]}X_{0r}^{(t_2,t_3]}(t\wedge T_{(3)})e^{s(1-3\delta)(t\wedge T_{(3)}-t_2)}>\frac{e^{3+(C_3-C_2)}}{\epsilon}\cdot\frac{N}{\ln(Nr)}\hspace{0.1 cm}\bigg|\hspace{0.1 cm}\mathcal{F}_{t_2}\bigg)\\
		&\leq P\bigg(\sup_{t\in[t_2,t_3]}X_{0r}^{(t_2,t_3]}(t\wedge T_{(3)})e^{s(1-3\delta)(t\wedge T_{(3)}-t_2)-r(t_3-t_2)}\geq\frac{e^{3-r(t_3-t_2)+(C_3-C_2)}}{\epsilon}\cdot\frac{N}{\ln(Nr)}\hspace{0.1 cm}\bigg|\hspace{0.1 cm}\mathcal{F}_{t_2}\bigg)\\
		&\leq P\bigg(\sup_{t\in[t_2,t_3]}X_{0r}^{(t_2,t_3]}(t\wedge T_{(3)})e^{s(1-3\delta)(t\wedge T_{(3)}-t_2)-r(t_3-t_2)}\geq\frac{e^{2+(C_3-C_2)}}{\epsilon}\cdot\frac{N}{\ln(Nr)}\hspace{0.1 cm}\bigg|\hspace{0.1 cm}\mathcal{F}_{t_2}\bigg)\\
		&\leq  P\bigg(\sup_{t\in[t_2,t_3]}X_{0r}^{(t_2,t_3]}(t\wedge T_{(3)})e^{-\int_{t_2}^{t\wedge T_{(3)}}G_{0r}^{(t_2,t_3]}(v)dv}\geq\frac{e^{2+(C_3-C_2)}}{\epsilon}\cdot\frac{N}{\ln(Nr)}\hspace{0.1 cm}\bigg|\hspace{0.1 cm}\mathcal{F}_{t_2}\bigg)\\
		&=P\bigg(\sup_{t\in[t_2,t_3]}W_{0r}^{(t_2,t_3]}(t\wedge T_{(3)})\geq\frac{e^{2+(C_3-C_2)}}{\epsilon}\cdot\frac{N}{\ln(Nr)}\hspace{0.1 cm}\bigg|\hspace{0.1 cm}\mathcal{F}_{t_2}\bigg)\\
		&\leq \epsilon.
		\end{align*}
		
	Now, for the mutation dominating case, we observe that from the definitions of $t_2$ and $t_3$ in (\ref{t2}) and (\ref{t3}), we have
		\begin{equation}\label{*3.2}
		t_3-t_2=\frac{1}{s}\ln\bigg(\frac{s}{N\mu^2}\bigg)+\frac{C_3-C_2}{s}.
		\end{equation}
	From the fact that $1\ll N\mu$ and $\mu\ll s$, we have	
		$$
		\mu(t_3-t_2)\leq \frac{\mu}{s}\ln\bigg(\frac{s}{\mu}\bigg)+\frac{(C_3-C_2)\mu}{s}\ll 1.
		$$
	Also, from $r\ll s$ and (\ref{rslogsu}), we get
		$$
		r(t_3-t_2)\leq \frac{r}{s}\ln\bigg(\frac{s}{\mu}\bigg)+\frac{(C_3-C_2)r}{s}\ll 1,
		$$
	which show that (\ref{r(t3-t2)}) and (\ref{u(t3-t2)}) also hold in this case. By following the same argument as in the recombination dominating case, we obtain that for sufficiently large $N$,
		$$
		E\Big[W_{0r}^{(t_2,t_3]}(t_3\wedge T_{(3)})\Big|\mathcal{F}_{t_2}\Big]\leq e^{2+(C_3-C_2)}\cdot\frac{r}{\mu^2},
		$$
	and $P(A_{19}^c|\mathcal{F}_{t_2}) \leq \epsilon$.
	\end{proof}

\subsection{Results on type 3 individuals}
	\begin{lemma}\label{PA20}
	For sufficiently large $N$, we have that for $t\in[t_2,t_3]$,
		$$
		E\Big[X_{3m}^{(t_2,t_3]}(t\wedge T_{(3)})\Big|\mathcal{F}_{t_2}\Big]\leq \frac{e^3N\mu}{s}\cdot e^{s(t-t_2)},
		$$
	and $P(A_{20}|\mathcal{F}_{t_2})\geq 1-\epsilon$.
	\end{lemma}

	\begin{proof}
	The proof is similar to that of Lemma \ref{EX3m}. First, recall that the process $\big(Z_{3m}^{(t_2,t_3]}(t\wedge T_{(3)}),t\geq t_2\big)$ is a mean-zero martingale by Proposition \ref{ZMar}.  By (\ref{Zm}), for all $t\geq t_2$, we have
		$$
		E\Big[e^{-\int_{t_2}^{t\wedge T_{(3)}}G_3(v)dv}X_{3m}^{(t_2,t_3]}(t\wedge T_{(3)})\Big|\mathcal{F}_{t_2}\Big]
		=E\bigg[\int_{t_2}^{t\wedge T_{(3)}}M_3^{(t_2,t_3]}(u)e^{-\int_{t_2}^uG_3(v)dv}du\bigg|\mathcal{F}_{t_2}\bigg],
		$$
	From (\ref{M3ab}), Lemma \ref{Gt2-3} part 3, and the definition of $T_5$ in (\ref{T5}), we have that for every $t\in[t_2,t_3]$,
		\begin{align}
		E\bigg[\int_{t_2}^{t\wedge T_{(3)}}M_3^{(t_2,t_3]}(u)e^{-\int_{t_2}^uG_3(v)dv}du\bigg|\mathcal{F}_{t_2}\bigg]
		&\leq E\bigg[\int_{t_2}^{t\wedge T_{(3)}} N\mu\cdot e^{-\int_{t_2}^u s(1-\tilde X_3(v))-rdv}du\bigg|\mathcal{F}_{t_2}\bigg]\nonumber\\
		&\leq N\mu \cdot E\bigg[\int_{t_2}^{t\wedge T_{(3)}} e^{-s(u-t_2)+s\int_{t_2}^{T_5}\tilde X_3(v)dv +r(t_3-t_2)}du\bigg|\mathcal{F}_{t_2}\bigg]\nonumber\\
		&\leq e^{1+r(t_3-t_2)}\cdot N\mu \cdot \int_{t_2}^t e^{-s(u-t_2)}du\nonumber\\
		&\leq e^{1+r(t_3-t_2)}\cdot \frac{N\mu}{s}. \label{*3.-1}
		\end{align}
	From (\ref{r(t3-t2)}), for sufficiently large $N$ and for all $t\in[t_2,t_3]$,
		$$
		E\bigg[\int_{t_2}^{t\wedge T_{(3)}}M_3^{(t_2,t_3]}(u)e^{-\int_{t_2}^uG_3(v)dv}du\bigg|\mathcal{F}_{t_2}\bigg]\leq \frac{e^2N\mu}{s}.
		$$
	Also, by Lemma \ref{Gt2-3} part 3, we have that for all $t\geq t_2$,
		\begin{align*}
		E\Big[e^{-\int_{t_2}^{t\wedge T_{(3)}}G_3(v)dv}X_{3m}^{(t_2,t_3]}(t\wedge T_{(3)})\Big|\mathcal{F}_{t_2}\Big]
		&\geq e^{-\int_{t_2}^t s\big(1+\delta e^{-s(1-3\delta)(v-t_2)}\big)dv}\cdot E\Big[X_{3m}^{(t_2,t_3]}(t\wedge T_{(3)})\Big|\mathcal{F}_{t_2}\Big]\\
		&\geq e^{-s(t-t_2)-\frac{\delta}{1-3\delta}}\cdot E\Big[X_{3m}^{(t_2,t_{3,r}]}(t\wedge T_{(3)})\Big|\mathcal{F}_{t_2}\Big].
		\end{align*}
	Therefore, using that $\delta<\frac{1}{4}$, for sufficiently large $N$, we have that if $t\in[t_2,t_3]$, then 
		$$
		E\Big[X_{3m}^{(t_2,t_3]}(t\wedge T_{(3)})\Big|\mathcal{F}_{t_2}\Big]\leq e^{2+\frac{\delta}{1-3\delta}}\cdot \frac{N\mu}{s}\cdot e^{s(t-t_2)}\leq\frac{e^3 N\mu}{s}\cdot e^{s(t-t_2)}.
		$$
	It follows from this inequality, along with (\ref{*3.1}), (\ref{*3.2}) and the definition of $C_3$ in (\ref{C3}) that in the recombination dominating case, for sufficiently large $N$, 
		$$
		E\Big[X_{3m}^{(t_2,t_3]}(t_3\wedge T_{(3)})\Big|\mathcal{F}_{t_2}\Big]\leq \frac{e^3N\mu}{s}\cdot e^{s(t_3-t_2)}
		=e^{3+(C_3-C_2)}\cdot\frac{N\mu}{r\ln(Nr)}
		=\bigg(\frac{\delta^2}{K_{2r}^+}\bigg)\cdot \frac{N\mu}{r\ln(Nr)},
		$$
	while in the mutation dominating, for sufficiently large $N$, 
		$$
		E\Big[X_{3m}^{(t_2,t_3]}(t_3\wedge T_{(3)})\Big|\mathcal{F}_{t_2}\Big]\leq \frac{e^3N\mu}{s}\cdot e^{s(t_3-t_2)}
		=\bigg(\frac{\delta^2}{K_{2m}^+}\bigg)\cdot \frac{1}{\mu}.
		$$
	Thus, by Markov's inequality, in both cases, we have that $P(A_{20}^c|\mathcal{F}_{t_2})\leq \epsilon$. 
	\end{proof}

	\begin{lemma}\label{PA21}
	For sufficiently large $N$, we have that for $t\in[t_2,t_3]$,
		$$
		E\Big[X_{3r}^{(t_2,t_3]}(t\wedge T_{(3)})\Big|\mathcal{F}_{t_2}\Big]\leq \frac{e^3Nr}{s}\cdot e^{s(t-t_2)},
		$$
	and $P(A_{21}|\mathcal{F}_{t_2})\geq 1-\epsilon$.
	\end{lemma}

	\begin{proof}
	The proof is almost exactly the same as that of Lemma \ref{PA20}. Recall from Proposition \ref{ZMar} that the process $\big(Z_{3r}^{(t_2,t_{3,r}]}(t\wedge T_{(3)}),t\geq t_2\big)$ is a mean-zero martingale. By (\ref{Zr}), for all $t\geq t_2$, we have
		$$
		E\Big[e^{-\int_{t_2}^{t\wedge T_{(3)}}G_{3r}^{(t_2,t_3]}(v)dv}X_{3r}^{(t_2,t_3]}(t\wedge T_{(3)})\Big|\mathcal{F}_{t_2}\Big]
		=E\bigg[\int_{t_2}^{t\wedge T_{(3)}}R_3^{(t_2,t_3]}(u)e^{-\int_{t_2}^uG_{3r}^{(t_2,t_3]}(v)dv}du\bigg|\mathcal{F}_{t_2}\bigg].
		$$
	From (\ref{R3<}), we have that $R_3^{(t_2,t_3]}(u)\leq Nr$ for all $u\geq t_2$. Using the same reason as in (\ref{*3.-1}), for every $t\in[t_2,t_{3,r}]$,
		$$
		E\bigg[\int_{t_2}^{t\wedge T_{(3)}}R_3^{(t_2,t_3]}(u)e^{-\int_{t_2}^uG_{3r}^{(t_2,t_3]}(v)dv}du\bigg|\mathcal{F}_{t_2}\bigg]
\leq e^{1+r(t_3-t_2)}\cdot \frac{Nr}{s}. 
		$$
	Also, by Lemma \ref{Gt2-3} part 4, we have that for all $t\in[t_2,t_3]$,
		\begin{align*}
		E\Big[e^{-\int_{t_2}^{t\wedge T_{(3)}}G_{3r}^{(t_2,t_3]}(v)dv}X_{3r}^{(t_2,t_3]}(t\wedge T_{(3)})\Big|\mathcal{F}_{t_2}\Big]
		&\geq e^{-\int_{t_2}^t \big(s(1+\delta e^{-s(1-3\delta)(v-t_2)})+r\big)dv}\cdot E\Big[X_{3r}^{(t_2,t_3]}(t\wedge T_{(3)})\Big|\mathcal{F}_{t_2}\Big]\\
		&\geq e^{-s(t-t_2)-\frac{\delta}{1-3\delta}-r(t_3-t_2)}\cdot E\Big[X_{3r}^{(t_2,t_3]}(t\wedge T_{(3)})\Big|\mathcal{F}_{t_2}\Big].
		\end{align*}
	Therefore, using that $\delta<\frac{1}{4}$, from (\ref{r(t3-t2)}), we have that if $t\in[t_2,t_3]$, then 
		$$
		E\Big[X_{3r}^{(t_2,t_3]}(t\wedge T_{(3)})\Big|\mathcal{F}_{t_2}\Big]\leq e^{1+2r(t_3-t_2)+\frac{\delta}{1-3\delta}}\cdot \frac{Nr}{s}\cdot e^{s(t-t_2)}
		\leq\frac{e^3Nr}{s}\cdot e^{s(t-t_2)}.
		$$
	It follows from this inequality, along with (\ref{*3.1}), (\ref{*3.2}) and the definition of $C_3$ in (\ref{C3}) that in the recombination dominating case, for sufficiently large $N$, 
		$$
		E\Big[X_{3r}^{(t_2,t_3]}(t_3\wedge T_{(3)})\Big|\mathcal{F}_{t_2}\Big]\leq \frac{e^3Nr}{s}\cdot e^{s(t_3-t_2)}
		=e^{3+(C_3-C_2)}\cdot\frac{N}{\ln(Nr)}
		=\bigg(\frac{\delta^2}{K_{2r}^+}\bigg)\cdot \frac{N}{\ln(Nr)},
		$$
	while in the mutation dominating case, for sufficiently large $N$, 
		$$
		E\Big[X_{3r}^{(t_2,t_3]}(t_3\wedge T_{(3)})\Big|\mathcal{F}_{t_2}\Big]\leq \frac{e^3Nr}{s}\cdot e^{s(t_3-t_2)}
		=\bigg(\frac{\delta^2}{K_{2m}^+}\bigg)\cdot \frac{r}{\mu^2},
		$$
	Thus, by Markov's inequality, in both cases, we have that $P(A_{21}^c|\mathcal{F}_{t_2})\leq \epsilon$. 
	\end{proof}

	Next, we will bound the probabilities of the events $A_{16}, A_{18}$ and $A_{22}$, but we will need an upper bound for the term $E\big[X_3^{[t_2]}(t\wedge T_{(3)})|\mathcal{F}_{t_2}\big]$ first.

	\begin{lemma}\label{EX3[t2]}
	For sufficiently large $N$, for all $t\geq t_2$, on the event $A_{(2)}$, we have
		$$
		E\Big[X_3^{[t_2]}(t\wedge T_{(3)})\Big|\mathcal{F}_{t_2}\Big]\leq
			\begin{cases}
		   	\displaystyle{\frac{eK_{2r}^+Nr\ln(Nr)}{s}\cdot e^{s(t-t_2)}} &\mbox{in the recombination dominating case}\\
		   	&\\
		   	\displaystyle{\frac{eK_{2m}^+N^2\mu^2}{s}\cdot e^{s(t-t_2)}} &\mbox{in the mutation dominating case}.
		   	\end{cases}
		$$
	\end{lemma}

	\begin{proof}
	From Proposition \ref{ZiMar}, we know that $\big(Z_3^{[t_2]}(t\wedge T_{(3)}),t\geq t_2\big)$ is a martingale. So, from (\ref{Zi}), Lemma \ref{Gt2-3} part 3, and the fact that $\delta <\frac{1}{4}$, for all $t\geq t_2$,
		\begin{align*}
		E\Big[Z_3^{[t_2]}(t\wedge T_{(3)})\Big|\mathcal{F}_{t_2}\Big]
		&=E\bigg[X_3^{[t_2]}(t\wedge T_{(3)})e^{-\int_{t_2}^{t\wedge T_{(3)}}G_3(v)dv}\bigg|\mathcal{F}_{t_2}\bigg]\\
		&\geq E\bigg[X_3^{[t_2]}(t\wedge T_{(3)})e^{-\int_{t_2}^{t\wedge T_{(3)}}s\big(1+\delta e^{-s(1-3\delta)(v-t_2)}\big)dv}\bigg|\mathcal{F}_{t_2}\bigg]\\
		&\geq e^{-\int_{t_2}^t s\big(1+\delta e^{-s(1-3\delta)(v-t_2)}\big)dv}E\Big[X_3^{[t_2]}(t\wedge T_{(3)})\Big|\mathcal{F}_{t_2}\Big]\\
		&\geq e^{-s(t-t_2)-\frac{\delta}{1-3\delta}}E\Big[X_3^{[t_2]}(t\wedge T_{(3)})\Big|\mathcal{F}_{t_2}\Big]\\
		&\geq e^{-s(t-t_2)-1}E\Big[X_3^{[t_2]}(t\wedge T_{(3)})\Big|\mathcal{F}_{t_2}\Big].
		\end{align*}
	Therefore, for all $t\geq t_2$,
		$$
		E\Big[X_3^{[t_2]}(t\wedge T_{(3)})\Big|\mathcal{F}_{t_2}\Big]
		\leq e^{s(t-t_2)+1} E\Big[Z_3^{[t_2]}(t\wedge T_{(3)})\Big|\mathcal{F}_{t_2}\Big]=e^{s(t-t_2)+1} Z_3^{[t_2]}(t_2)=e^{s(t-t_2)+1} X_3(t_2),
		$$
	and from the upper bound  of $X_3(t_2)$ on the event $A_{(2)}$ in Proposition \ref{@t2}, the result follows.
\end{proof}

	\begin{lemma}\label{PA18}
	For sufficiently large $N$, on the event $A_{(2)}$, we have $P(A_{18}|\mathcal{F}_{t_2})\geq 1-\delta^2$.
	\end{lemma}

	\begin{proof}
	In the recombination dominating case, from Lemmas \ref{PA20}, \ref{PA21} and \ref{EX3[t2]}, we have
		\begin{align}
		&E\bigg[s\int_{t_2}^{t_3\wedge T_{(3)}}\tilde X_3(v)dv\hspace{0.1 cm}\bigg|\hspace{0.1 cm}\mathcal{F}_{t_2}\bigg]\nonumber\\
		&\hspace{0.3 cm}\leq \int_{t_2}^{t_3}s\bigg(E\Big[\tilde X_3^{[t_2]}(v\wedge T_{(3)})\Big|\mathcal{F}_{t_2}\Big]+E\Big[\tilde X_{3m}^{(t_2,t_3]}(v\wedge T_{(3)})\Big|\mathcal{F}_{t_2}\Big]+E\Big[\tilde X_{3r}^{(t_2,t_3]}(v\wedge T_{(3)})\Big|\mathcal{F}_{t_2}\Big]\bigg)dv\nonumber\\
		&\hspace{0.3 cm}\leq \int_{t_2}^{t_3} \Big(eK_{2r}^+r\ln(Nr)\cdot e^{s(v-t_2)}+e^3\cdot \mu\cdot e^{s(v-t_2)}+e^3\cdot r\cdot e^{s(v-t_2)}\Big)dv.\label{*3.3}
		\end{align}
	Because $\mu \ll r$ and $1\ll Nr$, along with the definition of $C_3$ in (\ref{C3}), for sufficiently large $N$, on the event $A_{(2)}$, we have
		\begin{align*}
		E\bigg[s\int_{t_2}^{t_3\wedge T_{(3)}}\tilde X_3(v)dv\hspace{0.1 cm}\bigg|\hspace{0.1 cm}\mathcal{F}_{t_2}\bigg]
		&\leq \int_{t_2}^{t_3} e^3K_{2r}^+r\ln(Nr)\cdot e^{s(v-t_2)}dv\\
		&\leq \frac{e^3K_{2r}^+r\ln(Nr)}{s}\cdot e^{s(t_3-t_2)}\\
		&=e^{3+(C_3-C_2)}K_{2r}^+\\
		&=\delta^2.
		\end{align*}
	Thus, by Markov's inequality, we have $P(A_{18}^c)\leq \delta^2$.
	
	For the mutation dominating case, we can follow the same argument. Note that in this case, instead of getting (\ref{*3.3}), Lemma \ref{EX3[t2]} gives that
		$$
		E\bigg[s\int_{t_2}^{t_3\wedge T_{(3)}}\tilde X_3(v)dv\hspace{0.1 cm}\bigg|\hspace{0.1 cm}\mathcal{F}_{t_2}\bigg]
		\leq \int_{t_2}^{t_3} \Big(eK_{2m}^+N\mu^2\cdot e^{s(v-t_2)}+e^3\cdot \mu\cdot e^{s(v-t_2)}+e^3\cdot r\cdot e^{s(v-t_2)}\Big)dv.
		$$
	Because $1\ll N\mu$ and $r\ll N\mu^2$, for sufficiently large $N$, on the event $A_{(2)}$, we have
		$$
		E\bigg[s\int_{t_2}^{t_3\wedge T_{(3)}}\tilde X_3(v)dv\hspace{0.1 cm}\bigg|\hspace{0.1 cm}\mathcal{F}_{t_2}\bigg]
		\leq \int_{t_2}^{t_3} e^3K_{2m}^+N\mu^2\cdot e^{s(v-t_2)}dv,
		$$
	and by following the previous argument, we prove the result.
	\end{proof}

	\begin{lemma}\label{PA22}
	For sufficiently large $N$, on the event $A_{(2)}$, we have $P(A_{22}|\mathcal{F}_{t_2})\geq 1-\epsilon$.	
	\end{lemma}

	\begin{proof}
	We first consider the recombination dominating case. From Proposition \ref{ZiMar},  part 3 of Lemma \ref{Gt2-3}, Lemma \ref{EX3[t2]}, and (\ref{r(t3-t2)}), for sufficiently large $N$, on the event $A_{(2)}$, we have that
		\begin{align}
\textup{Var}\Big(Z_3^{[t_2]}(t_3\wedge T_{(3)})\hspace{0.1 cm}\bigg|\hspace{0.1 cm}\mathcal{F}_{t_2}\Big)
		&=E\bigg[\int_{t_2}^{t_3\wedge T_{(3)}} e^{-2\int_{t_2}^u G_3(v)dv}\Big(B_3^{[t_2]}(u)+D_3^{[t_2]}(u)\Big)X_3^{[t_2]}(u) du\hspace{0.1 cm}\bigg|\hspace{0.1 cm}\mathcal{F}_{t_2}\bigg]\nonumber\\
		&\leq E\bigg[\int_{t_2}^{t_3\wedge T_{(3)}}e^{-2\int_{t_2}^u\big(s(1-\tilde X_3(v))-r\big)dv}\cdot 2X_3^{[t_2]}(u\wedge T_{(3)}) du\hspace{0.1 cm}\bigg|\hspace{0.1 cm}\mathcal{F}_{t_2}\bigg]\nonumber\\
		&\leq E\bigg[\int_{t_2}^{t_3\wedge T_{(3)}}e^{-2s(u-t_2)+2s\int_{t_2}^u\tilde X_3(v)dv+2r(t_3-t_2)}\cdot 2X_3^{[t_2]}(u\wedge T_{(3)}) du\hspace{0.1 cm}\bigg|\hspace{0.1 cm}\mathcal{F}_{t_2}\bigg]\nonumber\\
		&\leq \int_{t_2}^{t_3}e^{-2s(u-t_2)+2+2r(t_3-t_2)}\cdot 2E\Big[X_3^{[t_2]}(u\wedge T_{(3)})\Big|\mathcal{F}_{t_2}\Big] du\nonumber\\
		&\leq 2e^{2+2r(t_3-t_2)}\cdot \int_{t_2}^{t_3}e^{-2s(u-t_2)}\cdot \frac{eK_{2r}^+Nr\ln(Nr)}{s}\cdot e^{s(u-t_2)} du\label{*3.4}\\
		&=\frac{2e^4K_{2r}^+Nr\ln(Nr)}{s}\int_{t_2}^{t_3}e^{-s(u-t_2)}du\nonumber\\
		&\leq \frac{2e^4K_{2r}^+Nr\ln(Nr)}{s^2}.\nonumber
		\end{align}
	It follows from this inequality and the $L^2$ maximal inequality 		that
		$$
		P(A_{22}^c)=P\bigg(\sup_{t\in[t_2,t_3]}\Big|Z_3^{[t_2]}(t\wedge T_{(3)})-X_3(t_2)\Big|\geq\sqrt{\frac{2e^4K_{2r}^+}{\epsilon}\cdot\frac{Nr\ln(Nr)}{s^2}}\bigg)\leq\epsilon.
		$$
	
	For the mutation dominating case, the argument is exactly the same except at (\ref{*3.4}), the upper bound from Lemma \ref{EX3[t2]} gives
		\begin{align*}
\textup{Var}\Big(Z_3^{[t_2]}(t_3\wedge T_{(3)})\hspace{0.1 cm}\Big|\hspace{0.1 cm}\mathcal{F}_{t_2}\Big)
		&\leq 2e^{2+2r(t_3-t_2)}\cdot \int_{t_2}^{t_3}e^{-2s(u-t_2)}\cdot \frac{eK_{2m}^+N^2\mu^2}{s}\cdot e^{s(u-t_2)} du\\
		&\leq \frac{2e^4K_{2m}^+N^2\mu^2}{s^2},
		\end{align*}
	and the result follows by applying the $L^2$ maximal inequality.
	\end{proof}

	\begin{lemma}\label{PA16}
	For sufficiently large $N$, on the event $A_{(2)}$, we have $P(A_{16}|\mathcal{F}_{t_2})\geq 1-\delta$.
	\end{lemma}

	\begin{proof}
	First, by the definition of $T_6$ in (\ref{T6}), we have that
		$$
		X_0(t_3\wedge T_{(3)})\leq \delta N+1<\frac{3\delta N}{2}.
		$$
	It follows from this inequality, Markov's inequality, Lemma \ref{PA20}, and Lemma \ref{PA21} that for sufficiently large $N$, on the event $A_{(2)}$, we have
		\begin{align}
		&P(T_4=t_3\wedge T_{(3)}|\mathcal{F}_{t_2})\nonumber\\
		&\hspace{0.2cm}=P\big(X_1(t_3\wedge T_{(3)})+X_2(t_3\wedge T_{(3)})\leq (1-3\delta)N\hspace{0.1 cm}\big|\hspace{0.1 cm}\mathcal{F}_{t_2}\big)\nonumber\\
		&\hspace{0.2cm}=P\big(X_0(t_3\wedge T_{(3)})+X_3(t_3\wedge T_{(3)})\geq 3\delta N\hspace{0.1 cm}\big|\hspace{0.1 cm}\mathcal{F}_{t_2}\big)\nonumber\\
		&\hspace{0.2cm}\leq P\bigg(X_0(t_3\wedge T_{(3)})\geq \frac{3\delta N}{2}\hspace{0.1 cm}\bigg|\hspace{0.1 cm}\mathcal{F}_{t_2}\bigg)+P\bigg(X_3(t_3\wedge T_{(3)})\geq \frac{3\delta N}{2}\hspace{0.1 cm}\bigg|\hspace{0.1 cm}\mathcal{F}_{t_2}\bigg)\nonumber\\
		&\hspace{0.2cm}\leq 0+E\Big[X_3(t_3\wedge T_{(3)})\hspace{0.1 cm}\Big|\hspace{0.1 cm}\mathcal{F}_{t_2}\Big]\cdot\frac{2}{3\delta N}\nonumber\\
		&\hspace{0.2cm}=\bigg(E\Big[X_3^{[t_2]}(t_3\wedge T_{(3)})\hspace{0.1 cm}\Big|\hspace{0.1 cm}\mathcal{F}_{t_2}\Big]+E\Big[X_{3m}^{(t_2,t_3]}(t_3\wedge T_{(3)})\hspace{0.1 cm}\Big|\hspace{0.1 cm}\mathcal{F}_{t_2}\Big]+E\Big[X_{3r}^{(t_2,t_3]}(t_3\wedge T_{(3)})\hspace{0.1 cm}\Big|\hspace{0.1 cm}\mathcal{F}_{t_2}\Big]\bigg)\cdot\frac{2}{3\delta N}\nonumber\\
		&\hspace{0.2cm} \leq \bigg(E\Big[X_3^{[t_2]}(t_3\wedge T_{(3)})\hspace{0.1 cm}\Big|\hspace{0.1 cm}\mathcal{F}_{t_2}\Big]+\frac{e^3N\mu}{s}\cdot e^{s(t_3-t_2)}+\frac{e^3Nr}{s}\cdot e^{s(t_3-t_2)}\bigg)\cdot\frac{2}{3\delta N}\label{*3.6}
		\end{align}
	At this point, the calculation splits between the two cases. In the recombination dominating case, by (\ref{*3.6}), (\ref{*3.1}), Lemma \ref{EX3[t2]}, and the definition of $C_3$ in (\ref{C3}), we have
		\begin{align*}
		&P(T_4=t_3\wedge T_{(3)}|\mathcal{F}_{t_2})\\
		&\hspace{0.2cm} \leq \bigg(\frac{eK_{2r}^+Nr\ln(Nr)}{s}\cdot e^{s(t_3-t_2)}+\frac{e^3N\mu}{s}\cdot e^{s(t_3-t_2)}+\frac{e^3Nr}{s}\cdot e^{s(t_3-t_2)}\bigg)\cdot\frac{2}{3\delta N}\\
		&\hspace{0.2cm}= \bigg(eK_{2r}^+e^{C_3-C_2}N+\frac{e^{3+(C_3-C_2)}N\mu}{r\ln(Nr)}+\frac{e^{3+(C_3-C_2)}N}{\ln(Nr)}\bigg)\cdot\frac{2}{3\delta N}\\
		&\hspace{0.2cm}= \frac{2e^{-2}\delta}{3} +\frac{2e^{3+(C_3-C_2)}}{3\delta}\cdot \bigg(\frac{\mu}{r\ln(Nr)}+\frac{1}{\ln(Nr)}\bigg).\nonumber
		\end{align*}
	Because $1 \ll Nr$ and $\mu \ll N\mu^2 \ll r\ln(Nr)$ , when $N$ is sufficiently large, on the event $A_{(2)}$, we have that $P(A_{16}^c|\mathcal{F}_{t_2})=P(T_4=t_3\wedge T_{(3)}|\mathcal{F}_{t_2})\leq \delta$. 
	
	The proof for the mutation dominating case is almost the same, except at (\ref{*3.6}), where Lemma \ref{EX3[t2]} gives 
		\begin{align*}
		P(T_4=t_3\wedge T_{(3)}|\mathcal{F}_{t_2})
		&\leq \bigg(\frac{eK_{2m}^+N^2\mu^2}{s}\cdot e^{s(t_3-t_2)}+\frac{e^3N\mu}{s}\cdot e^{s(t_3-t_2)}+\frac{e^3Nr}{s}\cdot e^{s(t_3-t_2)}\bigg)\cdot\frac{2}{3\delta N}\\
		&= \frac{2e^{-2}\delta}{3} +\frac{2e^{3+(C_3-C_2)}}{3\delta}\cdot \bigg(\frac{1}{N\mu}+\frac{r}{N\mu^2}\bigg).
		\end{align*}
	The result follows from the facts that $1\ll N\mu$ and $r\ll N\mu^2$.
	\end{proof}
	
	We have just finished showing that each of the events $A_{16}$ to $A_{21}$ conditioned on $\mathcal{F}_{t_2}$ occurs with probability close to 1 on the event $A_{(2)}$. In the next step, before we eventually prove Proposition \ref{@t3}, we are going to show that on the event $A_{(3)}$, we have that $T_{(3)}>t_3$.

	\begin{lemma}\label{T(3)>t3}
	For sufficiently large $N$, on the event $A_{(3)}$, we have that $T_{(3)}>t_3$.
	\end{lemma}

	\begin{proof}
	In this proof, we are working on the event $A_{(3)}$. By the definition of event $A_{16}$ in (\ref{A16}), we know that $T_4>t_3\wedge T_{(3)}$, and from the ways we define $T_5$ and $A_{18}$ as in (\ref{T5}) and (\ref{A18}), we have that $T_5>t_3\wedge T_{(3)}$. So, by the definition of $T_{(3)}$ in (\ref{T(3)}), it is left to show that $T_6>t_3\wedge T_{(3)}$.

	In the recombination dominating case, by the definitions of the events $A_{17}$ and $A_{19}$ in (\ref{A17}) and (\ref{A19r}), if $t\in [t_2,t_3]$, then
		\begin{align*}
		X_0\big(t\wedge T_{(3)}\big)
		&=X_0^{[t_2]}\big(t\wedge T_{(3)}\big)+X_{0r}^{(t_2,t_3]}\big(t\wedge T_{(3)}\big)\\
		&\leq \bigg(\frac{\delta}{2}+\frac{e^{3+(C_3-C_2)}}{\epsilon}\cdot\frac{1}{\ln(Nr)}\bigg)\cdot Ne^{-s(1-3\delta)(t\wedge T_{(3)}-t_2)}.
		\end{align*}
	Since $1 \ll Nr$, for sufficiently large $N$, we have that $X_0\big(t\wedge T_{(3)}\big) < \delta Ne^{-s(1-3\delta)(t\wedge T_{(3)}-t_2)}$, for all $t\in (t_2,t_3]$. Therefore, by the way we define $T_6$ as in (\ref{T6}), we have that $T_6>t_3\wedge T_{(3)}$.
	
	For the mutation dominating case, by following the same argument, we have that for all $t\in[t_2,t_3]$,
		$$
		X_0\big(t\wedge T_{(3)}\big)
		\leq \bigg(\frac{\delta}{2}+\frac{e^{3+(C_3-C_2)}}{\epsilon}\cdot\frac{r}{N\mu^2}\bigg)\cdot Ne^{-s(1-3\delta)(t\wedge T_{(3)}-t_2)},
		$$
	and the result follows because $r\ll N\mu^2$.
\end{proof} 

\subsection{The proof of Proposition \ref{@t3}}

	\begin{proof}
	Recall the definition of $A_{(3)}$ in (\ref{A(3)}). From Lemmas \ref{PA17}, \ref{PA19}, \ref{PA20}, \ref{PA21}, \ref{PA18}, \ref{PA22}, and \ref{PA16}, for sufficiently large $N$, on the event $A_{(2)}$, we have
		$$
		P\bigg(\bigcap_{i=16}^{22} A_i \hspace{0.1 cm}\bigg|\hspace{0.1 cm}\mathcal{F}_{t_2}\bigg)\geq 1-4\epsilon-7\delta-\delta^2.
		$$  
	Thus, by Proposition \ref{@t2}, for sufficiently large $N$,
		$$
		P(A_{(3)})=P\bigg(A_{(2)}\cap\Big(\bigcap_{i=16}^{22} A_i\Big)\bigg)\geq 1-4\epsilon-7\delta-\delta^2 -P\big(A_{(2)}^c\big)\geq 1-25\epsilon-7\delta-\delta^2.
		$$
	 Next, assume that we are on the event $A_{(3)}$. It follows from Lemma \ref{T(3)>t3} that $T_{(3)}>t_3$ when $N$ is sufficiently large. So, by the definition of $T_6$ as in (\ref{T6}), we have $X_0(t_3) < \delta Ne^{-s(1-3\delta)(t_3-t_2)}$, and by using the definition of $t_3$ in (\ref{t3}), we prove the first part of the proposition.

	For the proof of the second part of the proposition, we define
		\begin{equation}\label{K3}
		K_3=
			\begin{cases}
			\displaystyle{\frac{K_{2r}^-e^{(C_3-C_2)-2}}{2}}& \mbox{in the recombination dominating case}\\
			\\
			\displaystyle{\frac{K_{2m}^-e^{(C_3-C_2)-2}}{2}}& \mbox{in the mutation dominating case}.	
			\end{cases}
		\end{equation} 
	We will first prove the recombination dominating case. From (\ref{Zi}), the definition of the event $A_{22}$ in (\ref{A22r}), and Proposition \ref{@t2}, we have
		\begin{align}
		X_3(t_3)
		&\geq X_3^{[t_2]}(t_3)\nonumber\\
		&=Z_3^{[t_2]}(t_3)e^{\int_{t_2}^{t_3}G_3(v)dv}\nonumber\\
		&\geq \bigg(X_3^{[t_2]}(t_2)-\sqrt{\frac{8e^4K_{2r}^+}{\epsilon}\cdot\frac{Nr\ln(Nr)}{s^2}}\bigg)e^{\int_{t_2}^{t_3}G_3(v)dv}\nonumber\\
		&\geq\bigg(\frac{K_{2r}^-Nr\ln(Nr)}{s}-\sqrt{\frac{8e^4K_{2r}^+}{\epsilon}\cdot\frac{Nr\ln(Nr)}{s^2}}\bigg)e^{\int_{t_2}^{t_3}G_3(v)dv}\label{*3.7}\\
		&=\bigg(K_{2r}^--\sqrt{\frac{8e^4K_{2r}^+}{\epsilon}\cdot\frac{1}{Nr\ln(Nr)}}\bigg)\cdot \frac{Nr\ln(Nr)}{s}\cdot e^{\int_{t_2}^{t_3}G_3(v)dv}.\nonumber
		\end{align}
	Since, $1\ll Nr$, for sufficiently large $N$, 
	$$
	K_{2r}^--\sqrt{\frac{8e^4K_{2r}^+}{\epsilon}\cdot\frac{1}{Nr\ln(Nr)}}>\frac{K_{2r}^-}{2}>0.
	$$
	Hence, from Lemma \ref{Gt2-3}, the definition of $T_5$ in (\ref{T5}), inequality (\ref{r(t3-t2)}), and the definition of $K_3$ in (\ref{K3}), for sufficiently large $N$, we have that
		\begin{align*}
		X_3(t_3)
		&\geq\frac{K_{2r}^-}{2}\cdot\frac{Nr\ln(Nr)}{s}\cdot e^{\int_{t_2}^{t_3}\big(s(1-\tilde X_3(v)dv)-r\big)dv}\\
		&=\frac{K_{2r}^-}{2}\cdot\frac{Nr\ln(Nr)}{s}\cdot e^{s(t_3-t_2)-r(t_3-t_2)-s\int_{t_2}^{t_3}\tilde X_3(v)dv}\\
		&\geq \frac{K_{2r}^-}{2}\cdot\frac{Nr\ln(Nr)}{s}\cdot e^{s(t_3-t_2)-2}\\
		&=\frac{K_{2r}^-e^{(C_3-C_2)-2}N}{2}\\
		&=K_3N.
		\end{align*}
	For the upper bound for $X_3(t_3)$, from (\ref{Zi}), the definition of the event $A_{22}$ in (\ref{A22r}), Proposition \ref{@t2}, the fact that $\delta<\frac{1}{4}$, and the definitions of $C_3$ in (\ref{C3}), we have
		\begin{align}
		X_3^{[t_2]}(t_3)
		&=Z_3^{[t_2]}(t_3)e^{\int_{t_2}^{t_3}G_3(v)dv}\nonumber\\
		&\leq \bigg(X_3^{[t_2]}(t_2)+\sqrt{\frac{8e^4K_{2r}^+}{\epsilon}\cdot\frac{Nr\ln(Nr)}{s^2}}\bigg)e^{\int_{t_2}^{t_3}s\big(1+\delta e^{-s(1-3\delta)(v-t_2)}\big)dv}\nonumber\\
		&\leq \bigg(\frac{K_{2r}^+Nr\ln(Nr)}{s}+\sqrt{\frac{8e^4K_{2r}^+}{\epsilon}\cdot\frac{Nr\ln(Nr)}{s^2}}\bigg)e^{s(t_3-t_2)+\frac{\delta}{1-3\delta}}\label{*3.8}\\
		&\leq\bigg(K_{2r}^++\sqrt{\frac{8e^4K_{2r}^+}{\epsilon}\cdot\frac{1}{Nr\ln(Nr)}}\bigg)e^{(C_3-C_2)+1}N\nonumber\\
		&=\bigg(e^{-2}\delta^2+e^{(C_3-C_2)+1}\cdot\sqrt{\frac{8e^4K_{2r}^+}{\epsilon}\cdot\frac{1}{Nr\ln(Nr)}}\bigg)N.\nonumber
		\end{align}
	Since $1 \ll Nr$, for sufficiently large $N$, we have
$X_3^{[t_2]}(t_3)\leq \frac{\delta^2N}{3}$.
It follows from the definitions of the events $A_{20}$ and $A_{21}$ as defined in (\ref{A20r}) and (\ref{A21r}), along with the facts that $\mu \ll N\mu^2 \ll r\ln(Nr)$ and $1\ll Nr$, that for sufficiently large $N$, we have $X_{3m}^{(t_2,t_3]}(t_3)\leq \frac{\delta^2N}{3}$, and $X_{3r}^{(t_2,t_3]}(t_3)\leq \frac{\delta^2N}{3}$. Therefore, for sufficiently large $N$, we have $X_3(t_3)\leq \delta^2N$.

	We will now consider the mutation dominating case. Following the same argument as in the previous case, due to the differences in the definition of $A_{22}$ and the lower bound of $X_3^{[t_2]}(t_3)$ from Proposition \ref{@t2}, instead of having inequality (\ref{*3.7}), we will have
		\begin{align*}
		X_3(t_3)
		&\geq\bigg(\frac{K_{2m}^-N^2\mu^2}{s}-\sqrt{\frac{8e^4K_{2m}^+}{\epsilon}}\cdot\frac{N\mu}{s}\bigg)e^{\int_{t_2}^{t_3}G_3(v)dv}\\
		&=\bigg(K_{2m}^--\sqrt{\frac{8e^4K_{2m}^+}{\epsilon}}\cdot\frac{1}{N\mu}\bigg)\cdot \frac{N^2\mu^2}{s}\cdot e^{\int_{t_2}^{t_3}G_3(v)dv}.
		\end{align*}
	Because $1\ll N\mu$, for sufficiently large $N$, we have
		$$
		K_{2m}^--\sqrt{\frac{8e^4K_{2m}^+}{\epsilon}}\cdot\frac{1}{N\mu}>\frac{K_{2m}^-}{2}>0,
		$$
	and by using the same argument as in the previous case, we have that $X_3(t_3)\geq K_3N$. For the upper bound for $X_3(t_2)$,  due to the differences in the definition of $A_{22}$ and the lower bound of $X_3^{[t_2]}(t_3)$, instead of having inequality (\ref{*3.8}), we will have
		\begin{align*}
		X_3^{[t_2]}(t_3)
		&\leq \bigg(\frac{K_{2m}^+N^2\mu^2}{s}+\sqrt{\frac{8e^4K_{2m}^+}{\epsilon}}\cdot\frac{N\mu}{s}\bigg)e^{s(t_3-t_2)+\frac{\delta}{1-3\delta}}\\
		&\leq\bigg(K_{2m}^++\sqrt{\frac{8e^4K_{2m}^+}{\epsilon}}\cdot\frac{1}{N\mu}\bigg)e^{(C_3-C_2)+1}N\nonumber\\
		&=\bigg(e^{-2}\delta^2+e^{(C_3-C_2)+1}\cdot\sqrt{\frac{8e^4K_{2m}^+}{\epsilon}}\cdot\frac{1}{N\mu}\bigg)N,
		\end{align*}
	and because $1\ll N\mu$, for sufficiently large $N$, we have $X_3^{[t_2]}(t_3)\leq \frac{\delta^2N}{3}$.
Lastly, it follows from the definitions of the events $A_{20}$ and $A_{21}$ as defined in (\ref{A20m}) and (\ref{A21m}), along with the facts that $1 \ll N\mu$ and $r\ll N\mu^2$, that for sufficiently large $N$, we have $X_{3m}^{(t_2,t_3]}(t_3)\leq \frac{\delta^2N}{3}$, and $X_{3r}^{(t_2,t_3]}(t_3)\leq \frac{\delta^2N}{3}$. Thus, for sufficiently large $N$, we have $X_3(t_3)\leq \delta^2N$.
	\end{proof}

\section{Phase 4 and the proof of Proposition \ref{@t4}}\label{phase4}

	The main result in this phase can be proved using Theorem \ref{DNthm} as we did in phase 2. First, we define $\mathbf{X}(t), q, \alpha, \beta, b$ and $\tilde b$ as in (\ref{X(t)}), (\ref{q}), (\ref{alpha}), (\ref{beta=}), (\ref{b}) and (\ref{b-}), respectively.  Next, we define a random variable $B^*$ such that on the event that $\tilde X_3(t_3)>0$, we have
		\begin{equation}\label{B*}
		B^*=\frac{1}{\tilde{X}(t_3)}-1.
		\end{equation} 
	The definition of $B^*$ when $\tilde{X}_3(t_3)=0$ is not of interest, as we will work only on the event $A_{(3)}$, on which from Proposition \ref{@t3}, we know that $\tilde{X}_3(t_3)>0$. Next, for $t\geq t_3$, we define
		\begin{equation}\label{f*}
		f^*(t)=\frac{1}{1+B^*e^{-s(t-t_3)}}.
		\end{equation}
	and define 
		\begin{equation}\label{x*}
		x^*(t)=(x^*_1(t),x^*_2(t),x^*_3(t))
		=f^*(t)\bigg(\bigg(\frac{1-\tilde X_2(t_3)-\tilde X_3(t_3)}{\tilde X_3(t_3)}\bigg)e^{-s(t-t_3)},\bigg(\frac{\tilde X_2(t_3)}{\tilde X_3(t_3)}\bigg)e^{-s(t-t_3)},1\bigg).
		\end{equation}
	One can check that 
		\begin{equation}\label{*4.1}
			\begin{cases}
			x^*_1(t_3)=1-\tilde{X}_2(t_3)-\tilde{X}_3(t_3)
				=\tilde{X}_0(t_3)+\tilde{X}_1(t_3),\\
			x^*_2(t_3)=\tilde{X}_2(t_3),\\
			x^*_3(t_3)=\tilde{X}_3(t_3),
			\end{cases}
		\end{equation}
	and, for all $t\geq t_3$, we have 
		\begin{equation}\label{*4.1.1}
		x^*_1(t)+x^*_2(t)+x^*_3(t)=1.
		\end{equation} 
	By computation, we obtain that
		$$
		\frac{d}{dt}f^*(t)=\frac{sB^*e^{-s(t-t_3)}}{(1+B^*e^{-s(t-t_3)})^2}=sB^*e^{-s(t-t_3)}(f^*(t))^2,
		$$
	and 
		$$
		\frac{d}{dt}\big(e^{-s(t-t_3)}f^*(t)\big)=-\frac{se^{-s(t-t_1)}}{(1+B^*e^{-s(t-t_1)})^2}=-se^{-s(t-t_3)}(f^*(t))^2,
		$$
	which along with (\ref{x*}) imply that 
		$$
		\frac{d}{dt}x^*(t)=se^{-s(t-t_3)}(f^*(t))^2\bigg(-\frac{1-\tilde X_2(t_3)-\tilde X_3(t_3)}{\tilde X_3(t_3)},-\frac{\tilde X_2(t_3)}{\tilde X_3(t_3)}, B^*\bigg).		
		$$
	From (\ref{b}), (\ref{f*}), (\ref{x*}) and (\ref{*4.1.1}), for $t\geq t_3$,
		\begin{align*}
		b(x^*(t)	)&=s\big(-x_3^*(t)x_1^*(t), -x_3^*(t)x_2^*(t), (1-x_3^*(t))x_3^*(t)\big)	\\
		&=se^{-s(t-t_3)}(f^*(t))^2\bigg(-\frac{1-\tilde X_2(t_3)-\tilde X_3(t_3)}{\tilde X_3(t_3)},-\frac{\tilde X_2(t_3)}{\tilde X_3(t_3)}, B^*\bigg).
		\end{align*}
	Therefore, for $t\geq t_3$, we have $\frac{d}{dt}x^*(t)=b(x^*(t))$, and
		$$
		x^*(t)=x^*(t_3)+\int_{t_3}^tb(x^*(s))ds.
		$$

	Lastly, we define
		\begin{align}
		C_4&=C_3+\ln\bigg(\Big(\frac{1}{\delta^2}-1\Big)\Big(\frac{1}{K_3}-1\Big)\bigg),\label{C4}\\
		t_4&=
			\begin{cases} 
				\displaystyle{\frac{1}{s}\ln\bigg(\frac{s^2}{\mu r\ln(Nr)}\bigg)+\frac{C_4}{s}} &\hspace{ 1 cm}\mbox{in the recombination dominating case} \\
				\\
				\displaystyle{\frac{1}{s}\ln\bigg(\frac{s^2}{N\mu^3}\bigg)+\frac{C_4}{s}} &\hspace{ 1 cm}\mbox{in the mutation dominating case},
			\end{cases}\label{t4}\\
		A_{23}&=\bigg\{\sup_{t\in [t_3,t_4]}|X_i(t)-x^*_i(t)N|\leq \frac{K_3^2N}{4\delta^2}\hspace{0.2 cm}\textup{for $i=1,2,3$}\bigg\},\label{A23}\\
		A_{(4)}&=A_{(3)}\cap A_{23},\label{A(4)}
		\end{align}
	where $K_3$ is a positive constant that was defined in (\ref{K3}).

	\begin{lemma}\label{PA23}
	For sufficiently large $N$, on the event $A_{(3)}$, we have $P(A_{23}|\mathcal{F}_{t_3})\geq 1-\epsilon$.
	\end{lemma}

	\begin{proof}
	The proof is almost exactly the same as the proof of \ref{DNt12}. Recall from section \ref{phase2} that $k$ is a constant not depending on $N$ such that $ks$ is a Lipschitz constant of the function $b$. We define
		$$
		\Delta^*=\frac{K_3^2e^{-k(C_4-C_3)}}{12\delta^2},
		$$
	and $L=48/N$. We also define
		\begin{align*}
			\Omega^*_0&=\{|\mathbf{X}(t_3)-x^*(t_3)|\leq\Delta^*\}\\
			\Omega^*_1&=\bigg\{\int_{t_3}^{t_4}|\beta(\mathbf{X}(t))-b(\mathbf{X}(t))|dt\leq \Delta^*\bigg\}\\
			\Omega^*_2&=\bigg\{\int_{t_3}^{t_4}\alpha(\mathbf{X}(t))dt\leq L(t_4-t_3)\bigg\}.
		\end{align*}
	First, we consider the event $\Omega^*_0$. From Proposition \ref{@t3}, for sufficiently large $N$, on the event $A_{(3)}$, we have $X_3(t_3)>0$, which means $x^*(t)$ is well-defined. So, by (\ref{*4.1}), for sufficiently large $N$, on the event $A_{(3)}$, we have
		\begin{align*}
			|\mathbf{X}(t_3)-x^*(t_3)|
			&\leq |\tilde{X_1}(t_3)-x^*_1(t_3)|+|\tilde{X_2}(t_3)-x^*_2(t_3)|+|\tilde{X_3}(t_3)-x^*_3(t_3)| \\
			&=\tilde X_0(t_3).
		\end{align*}
	From the upper bound of $X_0(t_3)$ in Proposition \ref{@t3}, along with the facts that $r\ln(Nr)\ll s$ in the recombination dominating case and  $N\mu^2\ll s$ in the mutation dominating case, for sufficiently large $N$, on the event $A_{(3)}$ we have $|\mathbf{X}(t_3)-x^*(t_3)|\leq \Delta^*$. So, for sufficiently large $N$, we have $\Omega_0^{*c} \subseteq A_{(3)}^c$. 
	
	Next, by similar arguments to those used to prove that $\Omega_1^c=\emptyset$ and $\Omega_2^c=\emptyset$ in Proposition \ref{@t3}, for sufficiently large $N$, we have that $\Omega_1^{*c}=\emptyset$ and $\Omega_2^{*c}=\emptyset$.
	Therefore, by Theorem \ref{DNthm}, the definitions of $t_3$ and $t_4$ in (\ref{t3}) and (\ref{t4}), along with the fact that $1\ll Ns$, for sufficiently large $N$, on the event $A_{(3)}$, we have
		$$
			P(A_{23}^c|\mathcal{F}_{t_4})
			\leq \frac{4A(t_4-t_3)}{\Delta^{*2}}
			=\bigg(\frac{192(C_4-C_3)}{\Delta^{*2}}\bigg)\bigg(\frac{1}{Ns}\bigg)
			\leq \epsilon,
		$$
	which proves the result.
	\end{proof}

	Here, we will give a proof for Proposition \ref{@t4}.
	\begin{proof}[Proof of Proposition \ref{@t4}] 
	First, from the definition of $A_{(4)}$ in (\ref{A(4)}), and Propositions \ref{@t3} and \ref{PA23}, for sufficiently large $N$, we have
		$$
		P(A_{(4)})=P(A_{(3)}\cap A_{23})\geq 1-\epsilon-P(A_{(3)}^c)\geq 1-26\epsilon-7\delta-\delta^2.
		$$
	From this point, we will work on the event $A_{(4)}$. From the definition of $B^*$ in (\ref{B*}) and Proposition (\ref{@t3}), we have
		\begin{equation}\label{*4.2}
		\frac{1}{\delta^2}-1\leq B^*\leq \frac{1}{K_3}-1.
		\end{equation}
	By the definitions of $f^*(t), t_3, t_4, C_3$ and $C_4$ in (\ref{f*}), (\ref{t3}), (\ref{t4}), (\ref{C3}) and (\ref{C4}), respectively, along with the inequality (\ref{*4.2}), we obtain that
		\begin{equation}\label{*4.3}
		f^*(t_4)
		=\frac{1}{1+B^*e^{-(C_4-C_3)}}=\frac{1}{1+B^*\Big(\frac{1}{\delta^2}-1\Big)^{-1}\Big(\frac{1}{K_3}-1\Big)^{-1}}\\
		\leq\frac{1}{1+\Big(\frac{1}{K_3}-1\Big)^{-1}}=1-K_3,
		\end{equation}
	and
		\begin{equation}\label{*4.4}
		f^*(t_4)=\frac{1}{1+B^*\Big(\frac{1}{\delta^2}-1\Big)^{-1}\Big(\frac{1}{K_3}-1\Big)^{-1}}\geq\frac{1}{1+\Big(\frac{1}{\delta^2}-1\Big)^{-1}}=1-\delta^2.
		\end{equation}
	Note that from Proposition \ref{@t3}, it is clear that $K_3\leq \delta^2$. Using this fact, the definitions of $A_{23}$ in (\ref{A23}), the definition of $x^*_3(t)$ in (\ref{x*}), along with  (\ref{*4.3}) and (\ref{*4.4}) , we have
		$$
		X_3(t_4)\leq x^*_3(t_4)N+\frac{K_3^2N}{4\delta^2}= f^*_3(t_4)N+\frac{K_3^2N}{4\delta^2}\leq \bigg(1-K_3+\frac{K_3^2}{4\delta^2}\bigg)N\leq\bigg(1-\frac{3K_3}{4}\bigg)N,
		$$
	and
		$$
		X_3(t_4)\leq x^*_3(t_4)N-\frac{K_3^2N}{4\delta^2}= f^*_3(t_4)N-\frac{K_3^2N}{4\delta^2}\geq \bigg(1-\delta^2-\frac{K_3^2}{4\delta^2}\bigg)N\geq\bigg(1-\frac{5\delta^2}{4}\bigg)N.
		$$
	Lastly, using that $K_3\leq \delta^2$, the definitions $x^*_3(t)$ and $A_{23}$ in (\ref{x*}) and (\ref{A23}), along with (\ref{*4.1.1}) and (\ref{*4.3}), we obtain that
		\begin{align*}
		X_1(t_4)+X_2(t_4)
		&\geq (x^*_1(t_4)+x^*_2(t_4))N-\frac{K_3^2N}{2\delta^2}\\
		&=(1-x^*_3(t_4))N-\frac{K_3^2N}{2\delta^2}\\
		&=(1-f^*(t_4))N-\frac{K_3^2N}{2\delta^2}\\
		&\geq \bigg(K_3-\frac{K_3^2}{2\delta^2}\bigg)N\\
		&\geq \frac{K_3N}{2}.
		\end{align*}
	This completes the proof of this lemma.
	\end{proof}

\section{Phase 5 and the proof of Theorem \ref{THM}}\label{phase5}

	The technique used in the proof involves coupling with a branching process, similar to the proof of Lemma \ref{3mlemma}. We begin by defining
		\begin{align}
			t_{5+}&=t_4+\frac{1}{1-2\delta^2}\cdot \frac{1}{s}\ln(Ns),\label{t5+}\\
			t_{5-}&=t_4+(1-\delta)\cdot \frac{1}{s}\ln(Ns),\label{t5-}\\
			T_7&=\inf\{t\geq t_4:X_3(t)=N\},\label{T7}\\
			T_8&=\inf\{t\geq t_4:X_3(t)\leq N-\lfloor 2\delta^2 N \rfloor \},\nonumber\\
			A_{(5)}&=A_{(4)}\cap \{t_{5-}<T_7<t_{5+}\}.\nonumber
		\end{align}
	First, we will show that with probability close to 1, $T_7<T_8$ and $T_7\leq t_{5+}$.

	\begin{lemma} \label{T7<t5+}
	The following statements hold: 
		\begin{enumerate}
			\item For sufficiently large $N$, on the event $A_{(4)}$, we have $P(T_7<T_8|\mathcal{F}_{t_4})\geq 1-\epsilon$.
			\item For sufficiently large $N$, on the event $A_{(4)}$, we have $P(T_7\leq t_{5+}|\mathcal{F}_{t_4})\geq 1-\epsilon-\delta$.
		\end{enumerate}
	\end{lemma}

	\begin{proof}
	We are going to consider the process $(N-X_3(t),t\geq t_4)$. For $t\geq 0$, let $B(t)$ and $D(t)$ be the rates the this process increases and decreases by 1 at time $t$. This process increases by 1 when a type 3 individual dies and is replaced by an individual that is not type 3. Type 3 individuals die at total rate of $(1-2s)X_3(t)$, and the probability that the replacement is a type 3 individual is 
		$$
		(1-r)\tilde X_3(t)+r(\tilde X_1(t)+\tilde X_3(t))(\tilde X_2(t)+\tilde X_3(t)).
		$$
Hence, this process increases by 1 at rate
		$$
			B(t)=(1-2s)X_3(t)\big(1-(1-r)\tilde X_3(t)-r(\tilde X_1(t)+\tilde X_3(t))(\tilde X_2(t)+\tilde X_3(t))\big).
		$$
	The process decreases by 1 when an individual that is not of type 3 dies and is replaced by a type 3, or a mutation occurs on a type 1 or 2 individual. This occurs at rate
		\begin{align*}
			D(t)
			&=\big(X_0(t)+(1-s) X_1(t)+(1-s) X_2(t)\big)\cdot\big((1-r)\tilde X_3(t)+r(\tilde X_1(t)+\tilde X_3(t))(\tilde X_2(t)+\tilde X_3(t))\big)\nonumber\\
			&\hspace{0.5 cm}+\mu(X_1(t)+ X_2(t)).
		\end{align*}
	Then, for all $t\geq 3$, we have
		\begin{align*}
			B(t)
			&=(1-2s)X_3(t)\big(1-\tilde X_3(t)+r(\tilde 	X_0(t)\tilde X_3(t)-\tilde X_1(t)\tilde X_2(t))\big)\nonumber\\
			&\leq (1-2s)X_3(t)\big(1-\tilde X_3(t)+r\tilde X_0(t)\big)\nonumber\\
			&\leq (1-2s)(1+r)X_3(t)(1-\tilde X_3(t))\nonumber\\
			&\leq (1-2s+r)X_3(t)(1-\tilde X_3(t)),
		\end{align*}
	and
		\begin{align*}
			D(t)&\geq(1-s)(X_0(t)+X_1(t)+X_2(t))\cdot(1-r)\tilde X_3(t)\nonumber\\
			&=(1-s)(1-r)X_3(t)(1-\tilde X_3(t))\nonumber\\
			&\geq (1-s-r)X_3(t)(1-\tilde X_3(t)).
		\end{align*}
	Hence, we can think of the process $(N-X_3(t),t\in [t_4,T_7])$ as a birth-death process in which each individual gives birth at rate bounded above by $(1-2s+r)\tilde X_3(t)$ and dies at rate bounded below by $(1-s-r)\tilde X_3(t)$.
	
	Let $(Y(t),t\geq t_4)$ be a birth-death process in which each individual gives birth at rate $b(t)=(1-2s+r)\tilde X_3(t)$, and dies at rate $d(t)=(1-s-r)\tilde X_3(t)$, and $Y(0)=N-X_3(t_4)$. It is possible to couple the process $(Y(t),t\geq t_4)$ with the process $(N-X_3(t),t\geq t_4)$ such that for any time $t\geq t_4$, we have $Y(t)\geq N-X_3(t)$. This implies that if the process $Y$ reaches 0 before $\lfloor 2\delta^2 N \rfloor$, then the process $N-X_3$ will also reach 0 before $\lfloor 2\delta^2 N\rfloor$, which means that $T_7<T_8$.

	Here, since we are only interested in the probability that the process $Y$ reaches 0 before $\lfloor 2\delta^2 N\rfloor$, we will consider the induced discrete-time jump process of $(Y(t),t\in[t_4,T_7\wedge T_8))$. It is an asymmetric random walk process that jumps up by 1 with probability
		$$
		\frac{b(t)}{b(t)+d(t)}=\frac{1-2s+r}{2-3s},
		$$
	and jumps down by 1 with probability
		$$
		\frac{d(t)}{b(t)+d(t)}=\frac{1-s-r}{2-3s}.
		$$
	On the event $A_{(4)}$, we have from Proposition \ref{@t4} that $N-X_3(t_4)\leq 5\delta^2N/4$. Let $q=(1-s-r)/(1-2s+r)$, and note that because $r\ll s$, for sufficiently large $N$, we have 
		$$
		q\geq\frac{1-1.1s}{1-1.9s}>1.
		$$ 
	For sufficiently large $N$, on the event $A_{(4)}$, conditioning on the event $N-X_3(t_4)=k$, the probability that this asymmetric random walk reaches 0 before $\lfloor 2\delta^2 N\rfloor$ is
		$$
		1-\frac{q^k-1}{q^{\lfloor 2\delta^2 N\rfloor}-1}
		\geq 1- q^{k-2\delta^2 N}
		\geq 1- q^{(\frac{5\delta^2}{4}-2\delta^2)N}
		=1-q^{-3\delta^2N/4}
		\geq1-\bigg(\frac{(1-1.9s)^{1/s}}{(1-1.1s)^{1/s}}\bigg)^{3\delta^2Ns/4},
		$$
	and note that this upper bound is no longer depends on $k$. Since $s\ll 1$, when $N\rightarrow \infty$, we have
		$$
		\frac{(1-1.9s)^{1/s}}{(1-1.1s)^{1/s}}\rightarrow \frac{e^{-1.9}}{e^{-1.1}}=e^{-0.8}.
		$$
	Also, because $Ns \gg 1$, it follows that when $N\rightarrow \infty$, we have
		$$
		\bigg(\frac{(1-1.9s)^{1/s}}{(1-1.1s)^{1/s}}\bigg)^{(\eta-2\delta)Ns}\rightarrow 0.
		$$
	Thus, on the event $A_{(4)}$, for sufficiently large $N$, the probability that the asymmetric random walk reaches 0 before $\lfloor 2\delta^2 N\rfloor$ is bounded below by $1-\epsilon$. Therefore, through the coupling, for sufficiently large $N$, on the event $A_{(4)}$, we have $P(T_7<T_8|\mathcal{F}_{t_4})\geq 1-\epsilon$ .

	We will now prove part 2 of this lemma. It follows from part 1 that, for sufficiently large $N$, on the event $A_{(4)}$, 
		\begin{align*}
			P(T_7\leq t_{5+}|\mathcal{F}_{t_4})
			&\geq P(\{T_7\leq t_{5+}\}\cap \{T_7<T_8\}|\mathcal{F}_{t_4})\\
			&=P(T_7<T_8|\mathcal{F}_{t_4})-P(t_{5+}<T_7<T_8|\mathcal{F}_{t_4})\\
			&\geq 1-\epsilon-P(t_{5+}<T_7\wedge T_8|\mathcal{F}_{t_4}).
		\end{align*}
	So, we only need to show that for sufficiently large $N$, on the event $A_{(4)}$, 
		\begin{equation}\label{*5.1}
		P(t_{5+}<T_7\wedge T_8|\mathcal{F}_{t_4})\leq \delta.
		\end{equation}

	Now, for $t\in[0,(T_7\wedge T_8)-t_4]$, we define
		$$
		\lambda (t)=\int_0^t \tilde X_3(t_4+v) dv,
		$$
	and for $t\in[0,\lambda((T_7\wedge T_8)-t_4))$, we define $Y^*(t)=Y(\lambda^{-1}(t))$. The process $(Y^*(t),t\in[0,\lambda((T_7\wedge T_8)-t_4))$ is a birth-death process satisfying $Y^*(0)=N-X_3(t_4)$, where each individual gives birth at rate
		$$
		b^*(t)=b(\lambda^{-1}(t))(\lambda^{-1}(t))'=1-2s+r,
		$$
and each individual dies at rate
		$$
		d^*(t)=d(\lambda^{-1}(t))(\lambda^{-1}(t))'=1-s-r.
		$$
	For sufficiently large $N$, on the event that $t_{5+}<T_7\wedge T_8$, we have
		$$
		\lambda(t_{5+}-t_4)
		=\int_0^{t_{5+}-t_4}\tilde X_3(t_4+v)dv
		>\bigg(1-\frac{\lfloor 2\delta^2N \rfloor}{N}\bigg)(t_{5+}-t_4)\geq(1-2\delta^2)(t_{5+}-t_4)=\frac{1}{s}\ln(Ns).
		$$
	It follows that,
		\begin{align*}
			P(t_{5+}<T_7\wedge T_8|\mathcal{F}_{t_4})
			&=P(\{Y(t_{5+}-t_4)>0\}\cap \{t_{5+}<T_7\wedge T_8\}|\mathcal{F}_{t_4})\\
			&=P(\{Y^*(\lambda(t_{5+}-t_4))>0\}\cap \{t_{5+}<T_7\wedge T_8\}|\mathcal{F}_{t_4})\\
			&\leq P\bigg(Y^*\bigg(\frac{1}{s}\ln(Ns)\bigg)>0\bigg|\mathcal{F}_{t_4}\bigg).
		\end{align*}
	By the same reason we obtain (\ref{-5.2.1}) which gives the probability that the birth and death process survives until time $t$, if the process starts with one individual, we can generalize to the process that starts with any finite number of individuals. If $k\leq 5\delta^2N/4$, then
		\begin{align}
		P\bigg(Y^*\bigg(\frac{1}{s}\ln(Ns)\bigg)>0\bigg|Y^*(0)=k\bigg)
		&=1-\bigg(1-\frac{(1-2s+r)-(1-s-r)}{(1-2s+r)-(1-s-r)e^{-((1-2s+r)-(1-s-r))\cdot \frac{1}{s}\ln(Ns)}}\bigg)^k\nonumber\\
		&=1-\bigg(1-\frac{s-2r}{(1-s-r)e^{-\frac{2r}{s}\ln(Ns)}Ns-(1-2s+r)}\bigg)^k\nonumber\\
		&\leq 1-\bigg(1-\frac{s}{(1-s-r)e^{-\frac{2r}{s}\ln(Ns)}Ns-(1-2s+r)}\bigg)^{5\delta^2N/4},\label{-9.1}
		\end{align}
	and note that this upper bound does not depend on $k$. Now, by using the facts that $r\ll s\ll 1$ and $1 \ll Ns$ along with (\ref{rslogNs}), when $N$ is sufficiently large, on the event $A_{(4)}$, on which we know from Proposition \ref{@t4} that $Y^*(0)=N-X_3(t_4)\leq 5\delta^2N/4$, we have
		\begin{equation}\label{*5.2}
		P\bigg(Y^*\bigg(\frac{1}{s}\ln(Ns)\bigg)>0\bigg|\mathcal{F}_{t_4}\bigg)
		\leq 1-\bigg(1-\frac{s}{0.5Ns}\bigg)^{5\delta^2 N/4}=1-\bigg(1-\frac{2}{N}\bigg)^{5\delta^2 N/4}.
		\end{equation}
	Note that when $N\rightarrow\infty$, by using that $\delta\in(0,\frac{1}{4})$, we have
		$$
		1-\bigg(1-\frac{2}{N}\bigg)^{5\delta^2 N/4}\rightarrow 1-e^{-5\delta^2 N/2}\leq \frac{5\delta^2}{2}<\delta.
		$$
	This fact along with (\ref{*5.2}) prove the inequality (\ref{*5.1}).
	\end{proof}

	Next, we are going to show that $t_{5-}<T_7\wedge T_8$ with probability close to 1.

	\begin{lemma}\label{PA(5)}
	The following statements hold:
		\begin{enumerate}
			\item For sufficiently large $N$, on the event $A_{(4)}$, we have $P(t_{5-}<T_7\wedge T_8|\mathcal{F}_{t_4})\geq 1-2\epsilon $ .
			\item For sufficiently large $N$, we have $P(A_{(5)})\geq 1-29\epsilon-8\delta-\delta^2 $.
			\end{enumerate}
	\end{lemma}

	\begin{proof} 
	The proof is similar to the proof of Lemma \ref{T7<t5+}. In this proof, we are going to consider the process $(X_1(t)+X_2(t),t\geq t_4)$. For $t\geq t_4$, let $B(t)$ and $D(t)$ be the rates at which the process increases or decreases by 1. We will now give a lower bound for $B(t)$ and an upper bound for $D(t)$. For the increasing rate, one way to increase $X_1(t)+X_2(t)$ is by having a type 0 or type 3 individual die, which occurs at the total rate $X_0(t)+(1-2s)X_3(t)$, and the new individual is type 1 or 2 that is created without recombination, which occurs with probability $(1-r)(\tilde X_1(t)+\tilde X_2(t))$. Then,
		\begin{align*}
			B(t)
			&\geq (X_0(t)+(1-2s)X_3(t))\cdot(1-r)(\tilde X_1(t)+\tilde X_2(t))\\
			&\geq (1-2s)(1-r)(X_0(t)+ X_3(t))(\tilde X_1(t)+\tilde X_2(t))\\
			&\geq (1-2s-r)(\tilde X_0(t)+ \tilde X_3(t))(X_1(t)+X_2(t)).
		\end{align*}
	To decrease $X_1(t)+X_2(t)$, one way is by having a type 1 or type 2 die, and this occurs at total rate $(1-s)(X_1(t)+X_2(t))$, and the new individual cannot be type 1 or 2, which occurs with probability bounded above by $1-(1-r)(\tilde X_1(t)+\tilde X_2(t))$. Another way to decrease $X_1(t)+X_2(t)$ by having a type 1 or 2 mutate to type 3, which occurs at rate $\mu(X_1(t)+X_2(t))$. So,
		\begin{align*}
			D(t)
			&\leq (1-s)(X_1(t)+X_2(t))\cdot (1-(1-r)(\tilde X_1(t)+\tilde X_2(t)))+\mu(X_1(t)+X_2(t))\\
			&=\big((1-s)(\tilde X_0(t)+\tilde X_3(t)+r(\tilde X_1(t)+\tilde X_2(t)))+\mu\big)\cdot(X_1(t)+X_2(t)).
		\end{align*}
	When $t\in[t_4,T_7\wedge T_8]$, we have 
		$$
		\tilde X_1(t)+\tilde X_2(t)
		\leq 1-\tilde X_3 (t) \leq 
		\frac{\lfloor 2\delta^2N \rfloor}{N} 
		\leq 1-\frac{\lfloor 2\delta^2N \rfloor}{N}
		\leq \tilde X_3 (t)
		\leq \tilde X_0 (t)+\tilde X_3 (t),
		$$
	and 
		$$
		\mu\leq 2\bigg(1-\frac{\lfloor 2\delta^2N \rfloor}{N}\bigg)\mu \leq 2(\tilde X_0 (t)+\tilde X_3(t))\mu.
		$$
Hence, when $t\in[t_4,T_7\wedge T_8]$, 
		\begin{align*}
			D(t)
			&\leq (1-s)(1+r+2\mu)(\tilde X_0(t)+ \tilde X_3(t))(X_1(t)+X_2(t))\\
			&\leq (1-s+r+2\mu)(\tilde X_0(t)+ \tilde X_3(t))(X_1(t)+X_2(t)).
		\end{align*}

	Let $(Y(t),t\geq t_4)$ be a birth-death process such that $Y(t_4)=X_1(t_4)+X_2(t_4)$, in which each individual gives birth at rate $b(t)=(1-2s-r)(\tilde X_0(t)+ \tilde X_3(t))$ and each individual dies at rate $d(t)=(1-s+r+2\mu)(\tilde X_0(t)+ \tilde X_3(t))$. We can couple this process with $(X_1(t)+X_2(t),t\geq t_4)$ such that for any $t\in[t_4,T_7\wedge T_8]$, we have $Y(t)\leq X_1(t)+X_2(t)$, which means that if $Y(t)>0$, the $X_1(t)+X_2(t)>0$. Now, we consider the induced discrete time jump process of $(Y(t),t\in[t_4, T_7\wedge T_8])$. It is an asymmetric walk that jumps up with probability
		$$
		\frac{b(t)}{b(t)+d(t)}=\frac{1-2s-r}{2-3s+2\mu},
		$$
	and jumps down with probability 
		$$
		\frac{d(t)}{b(t)+d(t)}=\frac{1-s+r+2\mu}{2-3s+2\mu}.
		$$

	Next, for $t\in[0,(T_7\wedge T_8)-t_4]$, we define
		$$
		\lambda(t)=\int_0^t\big(\tilde X_0(t_4+v)+ \tilde X_3(t_4+v)\big)dv.
		$$
	Since $\tilde X_0(t_4+v)+ \tilde X_3(t_4+v)\leq 1$ for all $v\geq 0$, it follows that for $t\in[0,(T_7\wedge T_8)-t_4]$, we have $\lambda(t)\leq t$. Now, we define $Y^*(t)=Y(\lambda^{-1}(t))$. It follows that the process $(Y^*(t), t\in[0,\lambda((T_7\wedge T_8)-t_4])$ is a birth-death process such that $Y^*(0)=X_1(t_4)+X_2(t_4)$, in which each individual gives birth at rate
		\begin{equation}\label{*5.3}
		b^*(t)=b(\lambda^{-1}(t))(\lambda^{-1}(t))'=1-2s-r,
		\end{equation}
	and each individual dies at rate
		\begin{equation}\label{*5.4}
		d^*(t)=d(\lambda^{-1}(t))(\lambda^{-1}(t))'=1-s+r+2\mu.
		\end{equation}
	With these birth and death rates, we can extend the process $Y^*$ to be the birth-death process that is defined for all times $t\in[0,\infty)$, where the rates at which each individual gives birth and dies are given in (\ref{*5.3}) and (\ref{*5.4}), respectively.

	We will first show that for sufficiently large $N$, on the event $A_{(4)}$, 
		\begin{equation}\label{*5.5}
		P\bigg(Y^*\bigg((1-\delta)\cdot\frac{1}{s}\ln(Ns)\bigg)>0\bigg| \mathcal{F}_{t_4}\bigg)\geq 1-\epsilon.
		\end{equation}
	Similar to the way we get (\ref{-9.1}), if $k\geq \frac{K_3N}{2}$, then
		\begin{align*}
			&P\bigg(Y^*\bigg((1-\delta)\cdot\frac{1}{s}\ln(Ns)\bigg)>0\bigg|Y^*(0)=k\bigg)\\
			&\hspace{0.5 cm}=1-\bigg(1-\frac{(1-2s-r)-(1-s+r+2\mu)}{(1-2s-r)-(1-s+r+2\mu)e^{-((1-2s-r)-(1-s+r+2\mu))\cdot \frac{1-\delta}{s}\ln(Ns)}}\bigg)^k\\
			&\hspace{0.5 cm}=1-\bigg(1-\frac{s+2r+2\mu}{(1-s+r+2\mu)e^{\frac{2(1-\delta)r}{s}\ln(Ns)+\frac{2(1-\delta)\mu}{s}\ln(Ns)}(Ns)^{1-\delta}-(1-2s-r)}\bigg)^k\\
			&\hspace{0.5 cm}\geq 1-\bigg(1-\frac{s}{(1-s+r+2\mu)e^{\frac{2(1-\delta)r}{s}\ln(Ns)+\frac{2(1-\delta)\mu}{s}\ln(Ns)}(Ns)^{1-\delta}-(1-2s-r)}\bigg)^{\frac{K_3 N}{2}},
		\end{align*}
	and note that this lower bound does not depend on $k$. Note that from  Proposition \ref{@t4}, we know that on the event $A_{(4)}$, we have $Y^*(0)=Y(t_4)=X_1(t_4)+X_2(t_4)\geq K_3N/2$. Using the facts that $\mu\ll s$, $r\ll s$, $s \ll 1$ and using (\ref{rslogNs}), when $N$ is sufficiently large, on the event $A_{(4)}$,
		\begin{equation}\label{*5.6}
		P\bigg(Y^*\bigg((1-\delta)\cdot\frac{1}{s}\ln(Ns)\bigg)>0\bigg|\mathcal{F}_{t_4}\bigg)
		\geq 1-\bigg(1-\frac{s}{2(Ns)^{1-\delta}}\bigg)^{\frac{K_3 N}{2}}=1-\bigg(1-\frac{0.5(Ns)^{\delta}}{N}\bigg)^{\frac{K_3 N}{2}}.
		\end{equation}
	Note that because $1\ll Ns$, when $N\rightarrow\infty$,
		$$
		1-\bigg(1-\frac{0.5(Ns)^{\delta}}{N}\bigg)^{\frac{K_3 N}{2}}\rightarrow 1.
		$$
	This fact along with (\ref{*5.6}) proves (\ref{*5.5}).

	Lastly, by using the couplings and from part 1, the fact that $\lambda(t)\leq t$, part 1 of Lemma \ref{T7<t5+}, and the definition of $T_7$ in (\ref{T7}), for sufficiently large $N$, on the event $A_{(4)}$,
		\begin{align*}
			&P\bigg(Y^*\bigg((1-3\delta)\cdot\frac{1}{s}\ln(Ns)\bigg)>0\bigg| \mathcal{F}_{t_4}\bigg)\\
			&\hspace{0.5cm}=P(Y^*(t_{5-}-t_4)>0| \mathcal{F}_{t_4})\\
			&\hspace{0.5cm}=P(\{Y^*(t_{5-}-t_4)>0\}\cap\{t_{5-}<T_7\wedge T_8\}| \mathcal{F}_{t_4})+P(\{Y^*(t_{5-}-t_4)>0\}\cap\{T_7\wedge T_8\leq t_{5-}\}| \mathcal{F}_{t_4})\\
			&\hspace{0.5cm}\leq P(t_{5-}<T_7\wedge T_8|\mathcal{F}_{t_4})+P(\{Y^*(T_7-t_4)>0\}\cap\{T_7<T_8\}| \mathcal{F}_{t_4})+P(T_7\geq T_8|\mathcal{F}_{t_4})\\
			&\hspace{0.5cm}\leq P(t_{5-}<T_7\wedge T_8|\mathcal{F}_{t_4})+P(\{Y^*(\lambda(T_7-t_4))>0\}\cap\{T_7<T_8\}| \mathcal{F}_{t_4})+\epsilon\\
			&\hspace{0.5cm}= P(t_{5-}<T_7\wedge T_8|\mathcal{F}_{t_4})+P(\{Y(T_7-t_4)>0\}\cap\{T_7<T_8\}| \mathcal{F}_{t_4})+\epsilon\\
			&\hspace{0.5cm}\leq P(t_{5-}<T_7\wedge T_8|\mathcal{F}_{t_4})+P(\{X_1(T_7)+X_2(T_7)>0\}\cap\{T_7<T_8\}| \mathcal{F}_{t_4})+\epsilon\\
			&\hspace{0.5cm}= P(t_{5-}<T_7\wedge T_8|\mathcal{F}_{t_4})+\epsilon.				
		\end{align*}
	Therefore, for sufficiently large $N$, on the event $A_{(4)}$, we have $P(t_{5-}<T_7\wedge T_8|\mathcal{F}_{t_4})\geq 1-2\epsilon $.
	
	Lastly, to prove part 2, by using part 2 of Lemma \ref{T7<t5+} and part 1 of this lemma, for sufficiently large $N$, on the event $A_{(4)}$, we have that $P(t_{5-}<T_7<t_{5+}|\mathcal{F}_{t_4})\geq 1-3\epsilon-\delta$. With this fact and Proposition \ref{@t4}, for sufficiently large $N$, we have $P(A_{(5)})=P(A_{(4)}\cap\{t_{5-}<T_7<t_{5+}\})\geq1-29\epsilon-8\delta-\delta^2$.
	\end{proof}
	
	\begin{proof}[Proof of Theorem \ref{THM}]
	First, for every subsequence $(N_k)_{k=1}^{\infty}$, there is a further subsequence that satisfies (\ref{Con2.1}), or there is a further subsequence that satisfies (\ref{Con2.2}). By a subsequence argument, it is enough to prove Theorem \ref{THM} in the recombination dominating case and the mutation dominating case. 
	Now, recall that the stopping time $T$ defined in Theorem \ref{THM} is the first time that type 3 individuals have fixated in the population. We will show that if $\theta\in(0,1)$, then for sufficiently large $N$, we have
		$$
		P\big((1-\theta)t^*_N(r_N) \leq T \leq (1+\theta)t^*_N(r_N)\big)\geq 1-38\epsilon.
		$$
	We choose $\delta$ to be small enough so that 1) $\delta<\epsilon$, 2) $(1-\delta^2)^{-1}<1+\theta$ and 3) $1-2\delta>1-\theta$. From part 2 of Lemma \ref{PA(5)}, for sufficiently large $N$, we have $P(A_{(5)})\geq 1-29\epsilon-8\delta-\delta^2 \geq 1-38\epsilon$. Note that from the definition of $T_7$ in (\ref{T7}), we have $T_7=T\vee t_4$. Also, by the definition of $t_{5-}$ and the fact that $1\ll N\mu \ll Ns$, for sufficiently large $N$, we have $t_{5-}>t_4$. Thus, for sufficiently large $N$,  we have
		$$
		P(t_{5-}<T<t_{5+})=P(t_{5-}<T_7<t_{5+})\geq P(A_{(5)})\geq 1-38\epsilon.
		$$
	It is enough to show that $(1-\theta)t^*_N(r_N)\leq t_{5-}$ and $t_{5+}<(1+\theta)t^*_N(r_N)$.	
	
	Recall the definition of $t_N^*$ in (\ref{t*}). Because of (\ref{Con2.1}), in the recombination dominating case, for sufficiently large $N$,
		\begin{equation}\label{t*r}
		t_N^*(r_N)=\frac{1}{s_N}\ln\bigg(\frac{Ns_N^3}{\mu_N\cdot r_N\ln(Nr_N)}\bigg).
		\end{equation}
	Next, in the mutation dominating case,
		\begin{equation}\label{t*u<}
		t_N^*(r_N)\leq \frac{1}{s_N}\ln\bigg(\frac{Ns_N^3}{\mu_N\cdot N\mu_N^2}\bigg),
		\end{equation}
	and because of (\ref{Con2.2}), we have
		\begin{equation}\label{t*u>}
		t_N^*(r_N)\geq \frac{1}{s_N}\ln\bigg(\frac{Ns_N^3}{\mu_N\cdot (1\vee C)N\mu_N^2}\bigg)=\frac{1}{s_N}\ln\bigg(\frac{Ns_N^3}{\mu_N\cdot N\mu_N^2}\bigg)-\frac{\ln(1\vee C)}{s_N}.
		\end{equation} 
	
	From the definitions of $t_4$ and $t_{5+}$ in (\ref{t4}) and (\ref{t5+}), we have that 
		\begin{align*}
		t_{5+}
		&=t_4+\frac{1}{1-2\delta^2}\cdot \frac{1}{s_N}\ln(Ns_N)\\
		&=
			\begin{cases}
				\displaystyle{\frac{1}{s_N}\ln\bigg(\frac{s_N^2}{\mu_N r_N\ln(Nr_N)}\bigg)+\frac{C_4}{s_N}} +\frac{1}{1-2\delta^2}\cdot \frac{1}{s_N}\ln(Ns_N)&\hspace{ 0.2 cm}\mbox{in the recombination dominating case} \\
				\\
				\displaystyle{\frac{1}{s_N}\ln\bigg(\frac{s_N^2}{N\mu_N^3}\bigg)+\frac{C_4}{s_N}}+\frac{1}{1-2\delta^2}\cdot \frac{1}{s_N}\ln(Ns_N) &\hspace{ 0.2 cm}\mbox{in the mutation dominating case}
			\end{cases}\\
		&\leq
			\begin{cases}
				\displaystyle{\frac{1}{1-2\delta^2}\cdot \frac{1}{s_N}\ln\bigg(\frac{Ns_N^3}{\mu_N\cdot r_N\ln(Nr_N)}\bigg)+\frac{C_4}{s_N}} &\hspace{ 2.3 cm}\mbox{in the recombination dominating case} \\
				\\
				\displaystyle{\frac{1}{1-2\delta^2}\cdot \frac{1}{s_N}\ln\bigg(\frac{Ns_N^3}{\mu_N\cdot N\mu_N^2}\bigg)+\frac{C_4}{s_N}} &\hspace{ 2.3 cm}\mbox{in the mutation dominating case}.
			\end{cases}
		\end{align*}
	From (\ref{t*r}) and (\ref{t*u>}), we have
		\begin{equation}\label{th-1}
		t_{5+}\leq \frac{1}{1-2\delta^2}t^*_N(r_N)+\frac{1}{s}\bigg(\frac{\ln(1\vee C)}{1-2\delta^2}+C_4\bigg).
		\end{equation}
	Because $1\ll N\mu_N \ll Ns_N$ and $\mu_N\ll N\mu^2_N \ll s_N$, along with $r_N\ln_+(Nr_N)\ll s_N$, we have
		\begin{equation}\label{th-2}
		t^*_N(r_N)=\frac{1}{s_N}\ln\bigg(Ns_N\cdot \frac{s_N}{\mu_N}\cdot\frac{s_N}{\max\{N\mu_N^2,r_N\ln_+(Nr_N)\}}\bigg)\gg \frac{1}{s_N}.
		\end{equation}
	From (\ref{th-1}) and the way we choose $\theta$, for sufficiently large $N$,
		$$
		t_{5+}\leq \frac{1}{1-\delta^2}t^*_N(r_N) \leq (1+\theta)t^*_N(r_N).
		$$
	
	By a similar argument, from the definitions of $t_4$ and $t_{5-}$ in (\ref{t4}) and (\ref{t5-}), we have that 
		\begin{align*}
		t_{5-}
		&=t_4+(1-\delta)\cdot \frac{1}{s_N}\ln(Ns_N)\\
		&\geq
			\begin{cases}
				\displaystyle{(1-\delta)\cdot \frac{1}{s_N}\ln\bigg(\frac{Ns_N^3}{\mu_N\cdot r_N\ln(Nr_N)}\bigg)+\frac{C_4}{s_N}} &\hspace{ 1 cm}\mbox{in the recombination dominating case} \\
				\\
				\displaystyle{(1-\delta)\cdot \frac{1}{s_N}\ln\bigg(\frac{Ns_N^3}{\mu_N\cdot N\mu_N^2}\bigg)+\frac{C_4}{s_N}} &\hspace{ 1 cm}\mbox{in the mutation dominating case}.
			\end{cases}
		\end{align*}
	From (\ref{t*r}), (\ref{t*u<}), and (\ref{th-2}), for sufficiently large $N$, we have
		$$
		t_{5-}\geq (1-\delta)t^*_N(r_N)+\frac{C_4}{s_N}\geq(1-2\delta)t^*_N(r_N)\geq (1-\theta)t^*_N(r_N),
		$$
	which completes the proof.
	\end{proof}
		
\section*{Acknowledgement}
	I am grateful to Professor Jason Schweinsberg for teaching me several techniques, especially the one from one of his previous papers, which was used intensively in my proof. I also want to thank him for spending time reviewing this manuscript several times, which improved this manuscript significantly.

%+Bibliography 

%-Bibliography

\end{document}